\documentclass[12pt, reqno]{amsart}
\usepackage{amsmath,amsfonts,amssymb,amsthm,enumerate,appendix,caption}
\usepackage{cmap,mathtools}
\usepackage{paralist,bm}
\usepackage{graphics} 
\usepackage{epsfig} 
\usepackage{graphicx}
\usepackage{braket}
\usepackage{amscd}
\usepackage{verbatim}
\usepackage{mathrsfs}
\usepackage{epstopdf}
\usepackage[colorlinks=true]{hyperref}
\usepackage{xcolor}
\definecolor{myblue}{rgb}{0.0, 0.0, 1.0}
\definecolor{mygreen}{rgb}{0.01,0.75,0.20}
\hypersetup{ linkcolor=myblue, colorlinks=true, urlcolor=black, citecolor=red}
\usepackage{hyperref}
\usepackage{tabularx,ragged2e}
\newcolumntype{C}{>{\Centering\arraybackslash}X} 

\newtheorem{theorem}{Theorem}[section]
\newtheorem{corollary}[theorem]{Corollary}

\newtheorem{lemma}[theorem]{Lemma}
\newtheorem{proposition}[theorem]{Proposition}

\newtheorem{definition}[theorem]{Definition}

\newtheorem{remark}[theorem]{Remark}

\theoremstyle{definition}

\newcommand{\ep}{\varepsilon}
\newcommand{\eps}[1]{{#1}_{\varepsilon}}
\renewcommand\baselinestretch{1.10}
\usepackage[margin=1in]{geometry}

\numberwithin{equation}{section}
\newcommand{\diver}{\mathrm{div}\,}
\newcommand{\dx}{\,\mathrm{d}x}
\newcommand{\dy}{\,\mathrm{d}y}
\newcommand{\dz}{\,\mathrm{d}z}
\newcommand{\dt}{\,\mathrm{d}t}
\newcommand{\du}{\,\mathrm{d}u}
\newcommand{\dv}{\,\mathrm{d}v}
\newcommand{\dV}{\,\mathrm{d}V}
\newcommand{\ds}{\,\mathrm{d}s}
\newcommand{\dr}{\,\mathrm{d}r}
\newcommand{\dS}{\,\mathrm{d}S}
\newcommand{\dk}{\,\mathrm{d}k}
\newcommand{\dphi}{\,\mathrm{d}\phi}
\newcommand{\dtau}{\,\mathrm{d}\tau}
\newcommand{\dxi}{\,\mathrm{d}\xi}
\newcommand{\deta}{\,\mathrm{d}\eta}
\newcommand{\dsigma}{\,\mathrm{d}\sigma}
\newcommand{\dtheta}{\,\mathrm{d}\theta}
\newcommand{\dnu}{\,\mathrm{d}\nu}
\newcommand{\dmu}{\,\mathrm{d}\mu}
\newcommand{\drho}{\,\mathrm{d}\rho}
\newcommand{\dvrho}{\,\mathrm{d}\varrho}
\newcommand{\dkappa}{\,\mathrm{d}\kappa}
\newcommand{\dGamma}{\,\mathrm{d}\Gamma}
\newcommand{\dtildeGamma}{\,\mathrm{d}\tilde\Gamma}
\newcommand{\dtildeGa}{\,\mathrm{d}\tilde\Gamma}
\newcommand{\core}{C_0^{\infty}(\Omega)}
\newcommand{\sob}{W^{1,p}(\Omega)}
\newcommand{\sobloc}{W^{1,p}_{\mathrm{loc}}(\Omega)}
\def\ga{\alpha}     \def\gb{\beta}       \def\gg{\gamma}
\def\gc{\chi}       \def\gd{\delta}      \def\ge{\epsilon}
\def \gth{\theta}                         \def\vge{\varepsilon}
\def\gf{\phi}       \def\vgf{\varphi}    \def\gh{\eta}
\def\gi{\iota}      \def\gk{\kappa}      \def\gl{\lambda}
\def\gm{\mu}        \def\gn{\nu}         \def\gp{\pi}
\def\vgp{\varpi}    \def\gr{\rho}        \def\vgr{\varrho}
\def\gs{\sigma}     \def\vgs{\varsigma}  \def\gt{\tau}
\def\gu{\upsilon}   \def\gv{\vartheta}   \def\gw{\omega}
\def\gx{\xi}        \def\gy{\psi}        \def\gz{\zeta}
\def\Gg{\Gamma}     \def\Gd{\Delta}      \def\Gf{\Phi}
\def\Gth{\Theta}
\def\Gl{\Lambda}    \def\Gs{\Sigma}      \def\Gp{\Pi}
\def\Gw{\Omega}     \def\Gx{\Xi}         \def\Gy{\Psi}

\DeclarePairedDelimiter\abs{\lvert}{\rvert}%
\DeclarePairedDelimiter\norm{\lVert}{\rVert}%
\makeatletter
\let\oldnorm\norm
\def\norm{\@ifstar{\oldnorm}{\oldnorm*}}
\DeclarePairedDelimiter{\closedopen}{[}{)}
\DeclareMathOperator*{\esssup}{ess\,sup}
\DeclareMathOperator*{\lowlim}{\underline{lim}}
\DeclareMathOperator*{\uplim}{\overline{lim}}
\newcommand{\al} {\alpha}
\newcommand{\aln} {\alpha_n}
\newcommand{\pa} {\partial}
\newcommand{\be} {\beta}
\newcommand{\de} {\delta}
\newcommand{\De} {\Delta}
\newcommand{\Dep} {\Delta_p}
\newcommand{\Ga} {\Gamma}
\newcommand{\om} {\omega}
\newcommand{\Om} {\Omega}
\newcommand{\en} {\eta^n}
\newcommand{\ei} {\eta^i}
\newcommand{\la} {\lambda}
\newcommand{\La} {\Lambda}
\newcommand{\si} {\sigma}
\newcommand{\Si} {\Sigma}
\newcommand{\Gr} {\nabla}
\newcommand{\Grx} {\nabla_{x^{\prime}}}
\newcommand{\Grz} {\nabla_z}
\newcommand{\no} {\nonumber}
\newcommand{\noi} {\noindent}
\newcommand{\vph} {\varphi}
\newcommand{\ra} {\rightarrow}
\newcommand{\embd} {\hookrightarrow}
\newcommand{\wra} {\rightharpoonup}
\newcommand{\cd}{\circledast}

\newcommand{\pq}{\mathcal{H}_{p,q}(\Om)}
\newcommand{\pqo}{\mathcal{H}_{p,q_1}(\Om)}
\newcommand{\pqt}{\mathcal{H}_{p,q_2}(\Om)}
\newcommand{\tq}{\mathcal{H}_{2,q}(\Om)}
\newcommand{\pp}{\mathcal{H}_{p,p}(\Om)}
\newcommand{\two}{\mathcal{H}_{2,2}}
\newcommand{\pqR}{\mathcal{H}_{p,q}(\mathbb{R}^N)}
\newcommand{\pqab}{\mathcal{H}^{p,q}(\Om_{a,b,S})}
\newcommand{\dpp}{\D_0^{1,p}}
\newcommand{\DpA}{\D_A^{1,p}(\Om)}

\newcommand{\Dp}{\D_0^{1,p}(\Om)}
\newcommand{\DN}{\D_0^{1,N}(\Om)}
\newcommand{\Dtwo}{\D_0^{1,2}(\Om)}
\newcommand{\intRn}{\displaystyle{\int_{\mathbb{R}^N}}}
\newcommand{\intRnn}{\displaystyle{\int_{\mathbb{R}^{N-1}}}}
\newcommand{\intRnu}{\displaystyle{\int_{\mathbb{R}^N_{+}}}}
\newcommand{\intRk}{\displaystyle{\int_{\mathbb{R}^k}}}
\newcommand{\intRnk}{\displaystyle{\int_{\mathbb{R}^{N-k}}}}
\newcommand{\wrastar} {\overset{\ast}{\rightharpoonup}}
\newcommand{\intRno}{\displaystyle{\int_{\mathbb{R}^{N-1}}}}
\newcommand{\intRnone}{\displaystyle{\int_{\mathbb{R}^{N-1}}}}
\newcommand{\intRz}{\displaystyle{\int_{\mathbb{R}^{N}_{z}}}}
\newcommand{\intRx}{\displaystyle{\int_{\mathbb{R}^{N}_{x^{\prime}}}}}
\newcommand{\Bc}{B^c_1}
\newcommand{\intBc}{\displaystyle{\int_{B_1^c}}}
\newcommand{\Int}{{\displaystyle \int}}
\newcommand{\intom}{\displaystyle{\int_{|\om|=1}}}
\newcommand\restr[2]{{
		\left.\kern-\nulldelimiterspace 
		#1 
		\right|_{#2} 
}}
\newcommand{\Halu}{H_{\widetilde{\alpha}}}
\newcommand{\Omu}{\Om_{H_{\widetilde{\alpha}}}}
\def\vol{\mathrm{vol}}
\def\B{{\widetilde B}}
\def\w{{\widetilde w}}
\def\t{\tau}
\def \tb {\textcolor{blue}}
\def \tr {\textcolor{red}}
\def \tg {\textcolor{green}}
\def \tm {\textcolor{magenta}}
\def \ty {\textcolor{yellow}}
\def \uphi {{\underline{\phi}}}
\def \ophi {{\overline{\phi}}}
\def \ug {{\underline{g}}}
\def \uv {{\underline{V}}}
\def \og {{\overline{g}}}
\def \ov {{\overline{V}}}

\def\inpr#1{\left\langle #1\right\rangle}
\def\wp{{W^{1,p}_0(\Om)}}
\def\w2{{W^{1,2}_0(\Om)}}
\def\hh2{{H^1_0(\Om)}}
\def\wRN{{W^{1,p}(\RN)}}
\def\dpR{{\D^{1,p}_0(\RN)}}
\def\dpo{{\D^{1,p}_0(\Omega_1 \times \Omega_2)}}
\def\wpb{{W^{1,p}(B_1^c)}}
\def\wNb{{W^{1,N}(B_1^c)}}
\def\wN{{W^{1,N}_0(\Om)}}
\def\wpl{{{ W}^{1,p}_{{\mathrm loc}}(\Om)}}
\def\Bf{\mathbf}
\def\A{{\mathcal A}}
\def\C{{\mathcal C}}
\def\D{{\mathcal D}}
\def\E{{\mathcal E}}
\def\mH{{\mathcal H}}
\def\J{{\mathcal J}}
\def\K{{\mathcal K}}
\def\N{{\mathbb N}}
\def\F{{\mathcal F}}
\def\M{{\mathcal M}}
\def\S{\mathbb{S}}
\def\T{{\mathcal T}}
\def\mR{{\mathcal R}}
\def\hs{\hspace}
\def\cp{{\rm Cap}_{\bf{u}}}

\def\cset{{\subset \subset }}
\def\R{{\mathbb R}}
\def\RN{{\mathbb R}^N}
\def\Rk{{\mathbb R}^k}
\def\RNk{{\mathbb R}^{N-k}}
\def\RNz{{\mathbb R}^N_{z}}
\def\Rx{{\mathbb R}^N_{x^{\prime}}}
\def\({{\Big(}}
\def\){{\Big)}}
\def\wsd{{\F_d}}
\def\wsp{{\F_{\frac{N}{p}}}}
\def\ws2{{\F_{\frac{N}{2}}}}
\def\cc{{\C_c^\infty}}
\def\c1{{\C_c^1}}
\def\dis{{\displaystyle \int_{\Omega}}}
\def\disab{{\displaystyle \int_{\Omega_{a,b,S}}}}
\def\disone{{\displaystyle \int_{\Omega_1}}}
\def\distwo{{\displaystyle \int_{\Omega_2}}}
\def\disR{{\displaystyle \int_{\mathbb{R}^N}}}
\def\disC{{\displaystyle \int_{\C}}}
\def\disH{{\displaystyle \int_{\mH}}}
\def\tphi{{\tilde{\phi}}}
\def\tv{{\tilde{V}}}
\def\psix{{\psi_{\xp}}}
\def\p{{p^{\prime}}}
\def\q{{q^{\prime}}}
\def\r{{r^{\prime}}}
\def\Np{{N^{\prime}}}
\def\f{{\tilde{f}}}
\def\g{{\tilde{g}}}
\def\d{{\rm d}}
\def\dr{{\rm d}r}
\def\dS{{\rm d}S}
\def\dSk{{{\rm d}S}_k}
\def\dSNk{{{\rm d}S}_{N-k}}
\def\dl{{\rm d}\lambda}
\def\ds{{\rm d}s}
\def\dt{{\rm d}t}
\def\dz{{\rm d}z}
\def\dx{{\rm d}x}
\def\dy{{\rm d}y}
\def\dvr{{\rm d}\varrho}
\def\drho{{\rm d}\rho}
\def\dxp{{\rm d}x^{\prime}}
\def\dxN{{\rm d}x_N}
\def\dJ{{\rm d}J}
\def\xp{{x^{\prime}}}
\def\locN{{L^N_{{\mathrm loc}}(B_1^c)}}
\def\locp{{L^N_{{\mathrm loc}}(B_1^c)}}
\def\gz{{\tilde{g}_{z}}}
\def\vpp{{\vph^{\prime}}}
\def\G{{\tilde{G}}}
\def\H{{\mathcal{H}}}
\def\h{{\tilde{h}}}
\def\hg{{\hat{g}_{z}}}

\newcommand\blue[1]{\textcolor{blue}{#1}}
\newcommand\red[1]{\textcolor{red}{#1}}
\newcommand\supp{\mathrm{supp}\,}
\newcommand{\Rd}{\color{red}}
\newcommand{\Bk}{\color{black}}
\newcommand{\Bu}{\color{blue}}
\newcommand{\todo}[1]{\vspace{5 mm}\par \noindent
	\marginpar{\textsc{}} \framebox{\begin{minipage}[c]{0.95
				\textwidth} \tt #1
	\end{minipage}}\vspace{5 mm}\par}

\newcommand{\Hmm}[1]{\leavevmode{\marginpar{\Tiny%
			$\hbox to 0mm{\hspace*{-0.5mm}$\leftarrow$\hss}%
			\vcenter{\vrule depth 0.1mm height 0.1mm width \the\marginparwidth}%
			\hbox to
			0mm{\hss$\rightarrow$\hspace*{-0.5mm}}$\\\relax\raggedright #1}}}
\newcommand{\loc}{{\rm loc}}

\begin{document}
	\title[Minimizers for weighted $L^p$-Hardy inequality]{On existence of minimizers for weighted $L^p$-Hardy inequalities on $C^{1,\gamma}$-domains with compact boundary}
	
	\author{Ujjal Das}
	
	\address {Ujjal Das, Department of Mathematics, Technion - Israel Institute of
		Technology,   Haifa, Israel}
	
	\email {ujjaldas@campus.technion.ac.il, getujjaldas@gmail.com}

	\author{Yehuda Pinchover}
	
	\address{Yehuda Pinchover,
		Department of Mathematics, Technion - Israel Institute of
		Technology,   Haifa, Israel}
	
	\email{pincho@technion.ac.il}

	\author{Baptiste Devyver}
	
	\address{Baptiste Devyver, Institut Fourier, Universit\'e Grenoble Alpes, 100 rue des maths 38610 Gi\`eres, France}
	
	\email{baptiste.devyver@univ-grenoble-alpes.fr}	
	
	%
	\begin{abstract}
		Let $p \in (1,\infty)$, $\alpha\in \mathbb{R}$, and  $\Omega\subsetneq  \mathbb{R}^N$ be a $C^{1,\gamma}$-domain with a compact boundary $\partial \Omega$, where $\gamma\in (0,1]$. Denote by $\delta_{\Omega}(x)$ the distance of  a point $x\in \Omega$ to $\partial \Omega$. 
Let $\widetilde{W}^{1,p;\alpha}_0(\Omega)$ be the closure of $C_c^{\infty}(\Omega)$  in $\widetilde{W}^{1,p;\alpha}(\Omega)$, where
 $$\widetilde{W}^{1,p;\alpha}(\Omega):= \left\{\varphi \in   {W}^{1,p}_{\mathrm{loc}} (\Omega) \mid \left( \| \, |\nabla \varphi \, |\|_{L^p(\Omega;\delta_{\Omega}^{-\alpha})}^p + \|\varphi\|_{L^p(\Omega;\delta_{\Omega}^{-(\alpha+p)})}^p\right)<\infty \!\right\}.$$
 We study the following two variational constants:  the {\em weighted Hardy constant}
\begin{align*} 
 H_{\alpha,p}(\Omega): =\!\inf \left\{\int_{\Omega} |\nabla \varphi|^p \delta_{\Omega}^{-\alpha} \mathrm{d}x  \biggm| \int_{\Omega} |\varphi|^p \delta_{\Omega}^{-(\alpha+p)} \mathrm{d}x\!=\!1, \varphi \in \widetilde{W}^{1,p;\alpha}_0(\Omega) \right\} ,
\end{align*}
  and the {\it{weighted Hardy constant at infinity}} 
\begin{align*}
	\lambda_{\alpha,p}^{\infty}(\Omega) :=\sup_{K\Subset \Omega}\, 
	\inf_{W^{1,p}_{c}(\Omega\setminus \bar K)} \left\{\int_{\Omega\setminus \bar K} |\nabla \varphi|^p \delta_{\Omega}^{-\alpha} \mathrm{d}x  \biggm| \int_{\Omega\setminus \bar K} |\varphi|^p \delta_{\Omega}^{-(\alpha+p)} \mathrm{d}x=1 \right\}.
\end{align*}
We show that $H_{\alpha,p}(\Omega)$ is attained if and only if the spectral gap $\Gamma_{\alpha,p}(\Omega):= \lambda_{\alpha,p}^{\infty}(\Omega)-H_{\alpha,p}(\Omega)$ is strictly positive. Moreover, we obtain tight decay estimates for the corresponding minimizers. Furthermore, when $\Omega$ is {\em bounded} and $\alpha+p=1$, then $\lambda_{1-p,p}^{\infty}(\Omega)=0$ (no spectral gap) and the associated operator $-\Delta_{1-p,p}$ is null-critical in $\Omega$ with respect to the weight $\delta_\Omega^{-1}$, whereas, if $\alpha +p < 1$, then $\lambda_{\alpha,p}^{\infty}(\Omega)=\left|\frac{\alpha+p-1}{p}\right|^{p}>0=H_{\alpha,p}(\Omega)$ (positive spectral gap) and $-\Delta_{\alpha,p}$ is positive-critical in $\Omega$ with respect to the weight $\delta_\Omega^{-(\alpha+p)}$. 
		
		\medskip
		
		\noindent  2000  \! {\em Mathematics  Subject  Classification.}
		Primary  \! 49J40; Secondary  35B09, 35J20, 35J92.\\[1mm]
		
		{\color{blue} \ \ \ \qquad \qquad \qquad \ \ \ \ \ \ \ \  J. Spectr. Theory 15 (2025), 1089--1138}
	\end{abstract}
	\maketitle
	
\section{Introduction}
Let $N \geq 2$ and $\Gw \subsetneq \R^N$ be a $C^{1,\gamma}$-domain with a compact boundary $\partial \Om$, where $\gamma\in (0,1]$. Denote by $\delta_{\Gw}(x)$ the distance of  a point $x\in \Gw$ to $\partial \Gw$. 
Fix $p \in (1,\infty)$ and $\alpha\in \mathbb{R}$. 
We say that the $L^{\al,p}$-{\it{Hardy inequality}} (or the {\it weighted Hardy inequality}) is satisfied in $\Om$ if there exists $C>0$ such that
\begin{align} \label{Lp_Hardy}
  \int_{\Om} |\nabla \varphi|^p \ \delta_{\Om}^{-\al} \dx \geq  C  \int_{\Om} |\varphi|^p \ \delta_{\Om}^{-(\al+p)} \dx \qquad  \forall \varphi \in C_c^{\infty}(\Om).  
\end{align}
The one-dimensional weighted Hardy inequality was proved by Hardy, (see, \cite[p.~329]{Hardy}).See also the celebrated papers \cite{Muckenhoupt,Tomaselli} for  one-dimensional weighted Hardy-type inequalities involving general weights in place of the distance function to the boundary.  A comprehensive review on weighted Hardy inequalities is presented in \cite{Balinsky}.
For $\al=0$, Inequality \eqref{Lp_Hardy} is often called the {\it Hardy inequality} for domains with boundary. We note that  \eqref{Lp_Hardy} can also be viewed as a geometric {\em Caffarelli-Kohn-Nirenberg type inequality} for domains with boundary (cf. \cite{DELT09}). The validity of \eqref{Lp_Hardy} indeed depends on $\alpha$ and the domain $\Gw$.
For instance, if $\al+p \leq 1$, then \eqref{Lp_Hardy} does not hold on bounded Lipschitz domains \cite{Leherback1}, see also Proposition \ref{Prop:to_be_added}. On the other hand, for $\al+p>1$, \eqref{Lp_Hardy} is established for various types of domains: bounded Lipschitz domains \cite{Necas}, domains with H\"older boundary \cite{Kufner}, general H\"older conditions \cite{Wannebo},  unbounded John domains \cite{Leherback2}, domains with sufficiently large {\it{visual boundary}} \cite{Leherback1}, domains having uniformly $p$-fat complement \cite{Lewis, Wannebo2} and also for uniform domains with a locally uniform {\it{Ahlfors regular}} boundary and $p=2$ \cite{Robinson}, see also the references therein. 

Let  
$$\widetilde{W}^{1,p;\ga}(\Omega):= \left\{\varphi \in   {W}^{1,p}_\loc (\Omega) \mid \|\varphi\|_{\widetilde{W}^{1,p;\alpha}(\Omega)} :=\left( \| \, |\nabla \varphi \, |\|_{L^p(\Omega;\delta_{\Omega}^{-\alpha})}^p + \|\varphi\|_{L^p(\Omega;\delta_{\Omega}^{-(\alpha+p)})}^p\right)^{\!{1}/{p}}\!<\!\infty \!\right\},$$
and let $\widetilde{W}^{1,p;\alpha}_0(\Omega)$ be the closure of $C_c^{\infty}(\Omega)$  in $\widetilde{W}^{1,p;\alpha}(\Omega)$. Set 
\begin{align} \label{Lp_Hardy_const}
 H_{\al,p}(\Om) =\inf \left\{\int_{\Om} |\nabla \varphi|^p \delta_{\Om}^{-\al} \dx  \biggm| \int_{\Om} |\varphi|^p \delta_{\Om}^{-(\al+p)} \dx\!=\!1, \varphi \in \widetilde{W}^{1,p;\alpha}_0(\Omega) \right\} .
\end{align}
It is clear that $H_{\al,p}(\Om) \geq 0$. As mentioned above, if $\al+p \leq 1$, then the $L^{\al,p}$-Hardy inequality does not hold for a bounded smooth domain (in other words,  $H_{\al,p}(\Om)=0$). However, if  $H_{\al,p}(\Om)>0$ for a given domain $\Gw$, then by means of a standard density argument, it turns out that  \eqref{Lp_Hardy} holds with $H_{\al,p}(\Om)$ as the best constant for inequality \eqref{Lp_Hardy}. We call the constant $H_{\al,p}(\Om)$  the {\em $L^{\al,p}$-Hardy constant} (or simply the weighted Hardy constant) of $\Om$. 

For $m\in \N$, let
$$c_{\al,p,m}:=\left|\frac{\al+p-m}{p}\right|^{p}.$$
We recall that if $\al+p>N$, then for {\em any} domain $\Om \subsetneq \R^N$, we have 
$H_{\al,p}(\Om)\geq c_{\al,p,N}$ \cite[Theorem 5]{Avkhadiev} and \cite[Theorem 1.1]{GPP}. 
  Further, if $\Om \subsetneq \R^N$ is a {\em convex} domain and $\al+p>1$, then $H_{\al,p}(\Om) = c_{\al,p,1}$ \cite[Theorem A]{Avkhadiev_sharp}. There have been many efforts towards the  estimation of $H_{\al,p}(\Om)$ for various domains, see for example, \cite{Avkhadiev_Makarov,Barbatis,Cone} and references therein.

\medskip

The main aim of the present article is to address the question of the attainment in $\widetilde{W}^{1,p;\alpha}_0(\Omega)$ of $H_{\al,p}(\Om)$ for $C^{1,\gamma}$-domains with {\em compact} boundary and study the asymptotic behavior of the corresponding minimizers when exist. Consider another variational constant, the {\em weighted Hardy constant at infinity}, namely,
 \begin{align*}
 \la_{\al,p}^{\infty}(\Om) :=\sup_{K\Subset \Gw}\, 
 \inf_{W^{1,p}_{c}(\Om\setminus \bar K)} \left\{\int_{\Om\setminus \bar K} |\nabla \varphi|^p \delta_{\Om}^{-\al} \dx  \biggm| \int_{\Om\setminus \bar K} |\varphi|^p \delta_{\Om}^{-(\al+p)} \dx=1 \right\}.
 \end{align*}
	\begin{remark} 
	{\em At this point, the above definition of the weighted Hardy constant at infinity requires a short explanation.  By the Agmon-Allegretto-Piepen\-brink-type theorem (see, Theorem~\ref{aap_thm}), we have 
			\begin{multline*}
			\la_{\al,p}^{\infty}(\Om)=\sup \left\{ \lambda \in {\mathbb{R}}\mid \exists  K\Subset \Omega\ {\rm and} \ u\in W^{1,p}_{\loc}(\Omega\setminus \bar K ) \mbox{ such that }  u>0  \mbox{ and }  \right. \\  \left.  \
			  -{\rm div} \left(\gd_\Gw^{-\ga} |\nabla u|^{p-2}\nabla u\right) - \frac{\lambda }{\delta^{\ga+p}}u^{p-1} \geq 0\  {\rm in }\  \Omega \setminus \bar K\right\}.
			\end{multline*}
	Indeed, this key property will play a central role throughout the paper.}
\end{remark} 
 
Consider the {\em spectral gap} $\Gg_{\al,p}(\Om)$, defined by 
$$\Gg_{\al,p}(\Om)= \la_{\al,p}^{\infty}(\Om)-H_{\al,p}(\Om).$$
Since any $u \in W^{1,p}_{c}(\Om\setminus \bar K)$ for $K \Subset \Om$ can be extended by zero outside $K$ and considered as a  $\widetilde{W}^{1,p;\alpha}_0(\Omega)$-function, it follows that $\la_{\al,p}^{\infty}(\Om) \geq H_{\al,p}(\Om)$, i.e., $\Gg_{\al,p}(\Om) \geq 0$.
	\begin{remark} 
	{\em The $(\ga,p)$-Laplacian 
	$$\De_{\al,p}(u):={\rm div} \left(\gd_\Gw^{-\ga} |\nabla u|^{p-2}\nabla u\right)$$
	is clearly related to the energy functional $\int_{\Om} |\nabla \varphi|^p \delta_{\Om}^{-\al} \dx$. Consider the linear case $p=2$, and let $P$ be the Friedrichs extension of $-\delta_{\Om}^{(\al+2)}\De_{\al,2}$ in $L^2(\Omega;\delta_{\Omega}^{-(\alpha+2)})$.  
Then the best Hardy constant $H_{\al,2}(\Om)$ and  $\gl_{\al,2}^\infty(\Om)$ are, respectively, the bottom of the spectrum and the bottom of the essential spectrum of $P$ in 
 $L^2(\Omega;\delta_{\Omega}^{-(\alpha+2)})$ (see for example \cite{Agmon}). Hence, $\Gg_{\al,2}(\Om)$ is indeed a spectral gap of $P$. }
	\end{remark}    
We show that for any domain $\Gw\subsetneq \R^N$ with compact $C^{1,\gg}$-boundary, the following ``{\it{gap phenomenon}}" holds true.
\begin{theorem}\label{thm:main1}
Let $\Gw\subsetneq \R^N$ be a domain with compact $C^{1,\gg}$-boundary. Then the variational problem \eqref{Lp_Hardy_const} admits a unique (up to a multiplicative constant) positive minimizer in $\widetilde{W}^{1,p;\alpha}_0(\Omega)$ if and only if  $\Gg_{\al,p}(\Om)>0$.  
\end{theorem}
{Moreover, we obtain the tight decay estimates of the minimizers (when exist) for bounded domain in Theorem \ref{Thm:existence_bdd} and for unbounded domain in Theorem \ref{Thm:existence_exterior}. These results extend the main {results} of \cite{Lamberti}, which deals with the case $\ga=0$, to the case $\ga\in \R$. The proofs of the theorems strongly rely on the fact that for $C^{1,\gg}$-domains with compact boundary, one can compute explicitly the weighted Hardy constant at infinity $\la_{\al,p}^{\infty}(\Om)$.  We will first prove the case of bounded $C^{1,\gg}$-domains, and then the case of unbounded $C^{1,\gamma}$-domains with compact boundary, which we call {\em exterior domains} by analogy with the case where the boundary is connected.

\medskip

In the course of the paper, we also show some other results that are significant in their own right and which we now describe. Besides the fact that the explicit estimate of the weighted Hardy constant at infinity is crucial for our approach, there is independent interest to analyze this constant in the weighted case ($\al \neq 0$), see \cite{Robinson}. First, we prove that  if $\Gw$ is a $C^{1,\gg}$-bounded domain, then $\la_{\al,p}^{\infty}(\Om)= c_{\al,p,1}$ (Corollary \ref{Cor:lam_inf_est}), while for a $C^{1,\gg}$-exterior domain (i.e., an unbounded domain with compact boundary), we have (Theorem~\ref{Thm:la_inf_est2})
 $$\la_{\al,p}^{\infty}(\Om)=c_{\al,p}:= \min\{c_{\al,p,1}, c_{\al,p,N}\}.$$
 Note that if $\al+p=1$, then $c_{\al,p,1}=0$ and hence $H_{\al,p}(\Om)=0$, i.e., the Hardy inequality \eqref{Lp_Hardy} fails to hold in this case. In fact, the Hardy inequality \eqref{Lp_Hardy} does not hold for  $C^{1,\gg}$-bounded domains when $\al+p \leq 1$ (see \cite{Leherback1} and also Theorem \ref{Thm:nonexistence_bdd}). Moreover, for such $(\al,p)$, we show} that $-\De_{\al,p}$ is null-critical (resp. positive critical)  in $\Gw$ with respect to the weight $\delta_{\Om}^{-(\al+p)}$  if $\al+p=1$ (resp. $\al+p<1$), see  Proposition~\ref{Prop:to_be_added}. Similarly, in $C^{1,\gamma}$-exterior domains, if $\al+p \in \{1,N\}$, then  $c_{\al,p}=0$ and the Hardy inequality \eqref{Lp_Hardy} does not hold  (Corollary \ref{Non-exist_Half_space2}  and  Theorem \ref{Thm:la_inf_est2}). 

Let us discuss now the behavior of a minimizer.  We show in Theorem~\ref{Thm:existence_bdd} that a positive minimizer $u$ (if exists) satisfies $u\asymp \gd_\Gw^\gn$ near the compact boundary $\partial \Gw$, where either $\gn \in (\frac{\al+p-1}{p},\frac{\al+p-1}{p-1}]$ if $\ga+p>1$ or $\gn \in (\frac{\al+p-1}{p},0]$ if $\ga+p<1$ is the unique solution of the (transcendental) indicial equation  $$|\nu|^{p-2}\nu[\al+(1-\gn)(p-1)]=H_{\al,p}(\Om).$$ 
In fact, if  $\ga+p<1$, then $H_{\al,p}(\Om)=0$ for $C^{1,\gamma}$-bounded domains, and therefore, $\gn=0$.
 In addition, when $\Gw$ is a $C^{1,\gg}$-exterior domain, we prove in Theorem~\ref{Thm:existence_exterior} that the minimizer $u$ also satisfies $u\asymp |x|^{\tilde \gn}\sim \gd_\Gw^{\tilde \gn}$ near $\infty$, where
either ${\tilde \gn} \in (\frac{\al+p-N}{p-1},\frac{\al+p-N}{p}]$  if $\al+p<N$ or ${\tilde \gn} \in [0,\frac{\al+p-N}{p})$ if $\al+p>N$  is the  unique solution of the indicial equation
 $$|{\tilde \gn}|^{p-2} {\tilde \gn} [(\al-N+1)+(1-{\tilde \gn})(p-1)] = H_{\al,p}(\Om).$$
 
\medskip

The gap phenomenon for Hardy-type inequalities has a long history. For instance, for $\al=0$ and bounded $C^2$-domains, the gap phenomenon has been established in \cite[$p=2$]{MMP} and \cite[$p\in (1,\infty)$]{Itai}. In  $C^2$-exterior domains, one-way implication of the gap phenomenon was established in \cite{Chabrowski}, which was extended to the weighted case for $p=2$ in \cite{Colin}. One of the crucial difficulties occurs in the weighted case due to the simultaneous {\it concentrations}
 at the boundary and at infinity, see \cite{Colin}. A related gap phenomenon for weighted Hardy-type inequalities for $C^2$-bounded domains and $\ga+p>1$ is proved in \cite{Ando}. See also \cite{Byeon} where the gap phenomenon is established under the Neumann boundary condition. It is important to mention that indeed there are $C^{1,\gg}$-bounded domains with a positive spectral gap (Remark \ref{Ex:bdd}). However, to our knowledge, it has not been clear whether there are exterior domains with a positive spectral gap. In fact, we observe that there is no such $C^{1,\gg}$-exterior domain if $\al+p\geq N$. Nevertheless, when $\al+p<N$, we show that such exterior domains exist, 
 see Remark \ref{Ex:unbdd}.

Now we briefly describe our approach, which is based on the criticality theory. Having a $C^2$-domain, the authors in \cite{MMP,Itai} used the existence of  {\it tubular coordinates} near the boundary and the $C^2$-smoothness of the distance function near the boundary, which allowed them to construct suitable sub- and supersolutions of the corresponding Euler-Lagrange equation using the so called {\it Agmon trick}, see \cite{MMP,Itai} for more details.  Therefore, going from $C^2$-domains to $C^{1,\gamma}$ for some $\gamma \in (0,1]$ was indeed significantly challenging. In \cite{Lamberti}, the authors used criticality theory to prove the gap phenomenon in a $C^{1,\gamma}$-domain with a compact boundary and $\ga=0$. One of the key steps of their proof is to show that the corresponding Euler-Lagrange equation admits an {\it Agmon ground state} $u$ provided there is a spectral gap, and under this assumption, $u$ in fact belongs to the right function space, see \cite[theorems 4.1, 4.4 \& 5.1, 5.4]{Lamberti} for the details. For the first part of the proof, the authors used criticality theory \cite[Lemma 2.3]{Lamberti}, and for the second one they used Agmon's trick to construct suitable sub- and supersolutions of  (\cite[Lemma 3.4 \& 3.5]{Lamberti}) of the corresponding Euler-Lagrange equation followed by a weak comparison principle due to \cite{Itai} to compare the Agmon ground state and these sub- and supersolutions. Since the distance function is only Lipschitz continuous near a $C^{1,\gg}$-boundary, they replaced the distance function by the {\it{Green function}} of $-\De_p$ (for exterior domains the authors also used certain power functions of $|x|$ to investigate the behavior near infinity). It is important to note that this replacement was effective due to the fact that the Green's function of $-\De_p$ is $C^{1,\gg}$-smooth near $\partial \Gw$ and satisfies the Hopf-type lemma and therefore behaves asymptotically near $\partial \Gw$ as the distance function  (\cite[Lemma 3.2]{Lamberti}).

\medskip

 Our main tool of the present paper for proving the gap phenomenon is criticality theory for Equation \eqref{main_eq}. As it is done in \cite{Lamberti} for the case of $\al =0$, we first recall that if there is a spectral gap, then the operator $-\De_{\al,p}- H_{\al,p}(\Om)\gd_\Gw^{-(\ga+p)}\mathcal{I}_p$ is critical in $\Gw$ and therefore, it admits an Agmon ground state (Lemma \ref{lem-spectral-gap}). In the next step, we construct, with the help of the Agmon trick, suitable sub- and supersolutions of the corresponding Euler-Lagrange equation \eqref{main_eq}, in order to obtain the asymptotic behavior of a minimizer near the boundary. To perform the Agmon trick in the weighted case ($\al \neq 0$), we first derive the basic tool, a chain rule \eqref{weak_eq} for $\Delta_{\al,p}$, in Lemma \ref{weak_lapl}. In view of \eqref{weak_eq}, one may anticipate using the powers of Green's function of $-\De_{\al,p}$ (or the powers of $\gd_\Gw$) to exploit the Agmon trick when the domain is bounded. But, this does not help as the Green's function of $-\De_{\al,p}$ (or the powers of $\gd_\Gw$) for certain $\ga$ and $p$ does not satisfy Hopf-type lemma, see Remark \ref{Rmk:8.1}-$(ii)$. 
 
  One of our crucial ideas is to realize that the powers of the Green function of $-\De_p$ in $\Gw$ are the right comparison functions for our purpose near $\partial \Gw$, see Lemma \ref{lem:sub_sup_per_Agmon2}. This realization comes from the identity \eqref{Eq:id} in Remark \ref{rem_Gr}. As the identity \eqref{Eq:id} shows, one needs to take care of the extra singular terms in this identity to successfully perform the Agmon trick in the weighted case. Additionally, for exterior domains, a standard analysis of the function $t \mapsto |t|^{p-2}t [(\al-N+1)+(1-t)(p-1)]$ on a certain compact interval leads us to determine the appropriate power function of $|x|$ to apply the Agmon trick near infinity, see Lemma \ref{lem:agmon_subsup_inf}. Finally, in order to compare the Agmon ground state and these sub- and supersolutions, we derive the weak comparison results in Lemma \ref{lem:comparison} and Lemma \ref{lem:comparison1}, which extend the weak comparison results of \cite{Itai} to the weighted case.  Certain estimates of the integral of some powers of the distance function near the boundary play a crucial role in extending these results to the weighted case.

\medskip

{In Table \ref{table1}, we summarize the spectral gap phenomenon as a function of the boundedness of $\Gw$ and the values of $\al+p$. We use the notation ${\Gamma_{\al,p}(\Om) \gneq 0}$ for cases in which we give examples of domains with and without a positive spectral gap.}
\footnotesize
\begin{table}[!ht]
\setlength\extrarowheight{2pt} 
\begin{tabularx}{\textwidth}{| 
 p{\dimexpr.16\linewidth-4.8\tabcolsep-1.3333\arrayrulewidth}   |  p{\dimexpr.20\linewidth-4\tabcolsep-1.3333\arrayrulewidth}   | 
 p{\dimexpr.20\linewidth-6.8\tabcolsep-1.3333\arrayrulewidth} | p{\dimexpr.20\linewidth-3.7\tabcolsep-1.3333\arrayrulewidth} 
 | p{\dimexpr.20\linewidth-4\tabcolsep-1.3333\arrayrulewidth} |  p{\dimexpr.20\linewidth-3.7\tabcolsep-1.3333\arrayrulewidth}  | 
 p{\dimexpr.20\linewidth-2\tabcolsep-1.3333\arrayrulewidth} |}
\hline
 Domain \ \ \ vs. $\al+p$ & \ \ \ {$\alpha+p < 1$} & 
 \ \ \ $\alpha+p = 1$ & \ \ $1\!<\!\alpha+p \! < \! N$ & \ \ \ $\alpha+p = N$ & \ \ \ $\alpha+p > N$  \\
\hline
{\bf{{\small $\Gw\in C^{1,\gamma}$ bounded domain}}} & ${H_{\al,p}(\Om) =  0}$,
${ \la_{\al,p}^{\infty}(\Om)   =   c_{\al,p,1} }$ ${\Gamma_{\al,p}(\Om)>0}$  {\tiny{Minimizer always exists}}. & ${H_{\al,p}(\Om)=0}$ ${\la_{\al,p}^{\infty}(\Om)=0}$ ${\Gamma_{\al,p}(\Om)=0}$  {\tiny {Minimizer $\! \! \!$ does not exist}}. & 
${H_{\al,p}(\Om)>0}$ ${\la_{\al,p}^{\infty}(\Om)=c_{\al,p,1}}$ ${\Gamma_{\al,p}(\Om) \gneq 0}$  {\tiny $\ \ $ {Minimizer exists $\ \ \ $ $\iff$ $\Gamma_{\al,p}(\Om) > 0$.}}  & 
${H_{\al,p}(\Om)>0}$  ${\la_{\al,p}^{\infty}(\Om)=c_{\al,p,1}}$ ${\Gamma_{\al,p}(\Om) \gneq 0}$ {$ \! \!$}  {\tiny $\ \ $ {Minimizer exists $\ \ \ $ $\iff$ $\Gamma_{\al,p}(\Om) > 0$}} & 
${H_{\al,p}(\Om)>0}$  ${\la_{\al,p}^{\infty}(\Om)=c_{\al,p,1}}$ ${\Gamma_{\al,p}(\Om) \gneq 0}$  {\tiny $\ \ $ {Minimizer exists $\ \ \ $ $\iff$ $\Gamma_{\al,p}(\Om) > 0$}}.  \\
\hline
{\bf{{\small $\Gw\in C^{1,\gamma}$ exterior domain}}} &  
${H_{\al,p}(\Om)>0}$ ${\la_{\al,p}^{\infty}(\Om)=c_{\al,p,1}}$ ${\Gamma_{\al,p}(\Om) \geq 0}$ {\tiny $\ \ $ {Minimizer exists $\ \ \ $ $\iff$ $\Gamma_{\al,p}(\Om) > 0$}}.  & ${H_{\al,p}(\Om)=0}$ ${\la_{\al,p}^{\infty}(\Om)=0}$ ${\Gamma_{\al,p}(\Om) = 0}$ {\tiny{Minimizer does not exist}}. & 
${H_{\al,p}(\Om)>0}$ ${\la_{\al,p}^{\infty}(\Om)=c_{\al,p}}$ ${\Gamma_{\al,p}(\Om) \gneq 0}$ {\tiny$\ \ $ {Minimizer exists $\ \ \ $ $\iff$ $\Gamma_{\al,p}(\Om) > 0$}}.   & 
${H_{\al,p}(\Om)=0}$ ${\la_{\al,p}^{\infty}(\Om)=0}$ ${\Gamma_{\al,p}(\Om) = 0}$ {\tiny{Minimizer does not exist}}. & 
${H_{\al,p}(\Om)=c_{\al,p,N}}$ ${\la_{\al,p}^{\infty}(\Om)=c_{\al,p,N}}$ ${\Gamma_{\al,p}(\Om) = 0}$ {\tiny{Minimizer does not exist}}.   \\
\hline
\end{tabularx}
\caption{Gap phenomenon depending on the domain and on $\al+p$}\label{table1}
\end{table}
\normalsize

\medskip

Throughout the paper, we use the following notation and conventions.
\begin{itemize}
	\item For $R> 0$, we denote by $B_R\subset \R^N$ the open ball of radius $R$ centered at $0$, and let $B_R^c=\R^N \setminus \overline{B_R}$.
	\item For $r >0$, we set $\Om_r \!:=\!\{x\in \Om: 0<\delta_{\Om}(x)<r\}$, $\Om^r:=\{x\in \Om: \delta_{\Om}(x)>r\}$, and 
	$$\Sigma_r \!:=\! \{x\in \Om: \delta_{\Om}(x)=r\} , \ \ \  D_r\!:=\!\{x \in  \Om:  {r}/{2}<\delta_{\Om}(x)<r\}.$$
	\item $\chi_S$ denotes the characteristic function of a set $S\subset \R^N$.	
	\item We write $A_1 \Subset A_2$ if  $\overline{A_1}$ is a compact set, and $\overline{A_1}\subset A_2$.
\item For a Lebesgue measurable set $A\subset \R^N$, its Lebesgue measure is denoted by  $|A|_N$.
	\item $C$ refers to a positive constant which may vary from line to line.
	\item Let $g_1,g_2$ be two positive functions defined in $\Gw$. We write $g_1\asymp g_2$ in
	$\Gw$ if there exists a positive constant $C$ such
	that $C^{-1}g_{2}(x)\leq g_{1}(x) \leq Cg_{2}(x)$ for all  $x\in \Gw$. 
 \item Let $g_1,g_2$ be two functions defined in $\Gw$. If $x_0\in \overline{\Gw}\cup \{\infty\}$ and  $\lim_{x\to x_0}g_1(x)/g_2(x)=1$, we write $g_1\sim g_2$ as $x\to x_0$.  	
	\item For any real valued measurable function $u$ and $\Omega\subset \R^N$, we define
	$$\inf_{\Omega}u=\mathrm{ess}\inf_{\Omega}u, \quad \sup_{\Omega}u=\mathrm{ess}\sup_{\Omega}u, \quad  u^+=\max(0,u), \quad u^-=\max(0,-u).$$
	\item For a subspace $X(\Om)$ of measurable functions on $\Gw$, 
	$X_c(\Om) :=\{f \!\in \! X(\Om)\!\mid\! \supp f \!\Subset\! \Om\}$.  
	
	
	\item For a real valued function $u$, we define $\mathcal{I}_p(u) = |u|^{p-2}u$.
 \item   The operator $\De_{p}(u):={\rm div} \left(|\nabla u|^{p-2}\nabla u\right)$) is called the {\em $p$-Laplacian}.
	\item   The operator $\De_{\al,p}(u):={\rm div} \left(\gd_\Gw^{-\ga} |\nabla u|^{p-2}\nabla u\right)$  is called the {\em $(\ga,p)$-Laplacian}.
 \item $L^p(\Omega;\om)$ denotes the weighted $L^p$ space on $\Om$ with respect to the weight function $\om$.
\end{itemize}
 \section{Preliminaries}
 Let $V \in L^{\infty}_{\loc}(\Om)$. Consider the functional 
 $$\mathcal{Q}_{\ga,p,V}(\vgf):= \int_{ \Om } \left(\delta_{\Om}^{-\al} |\nabla \varphi|^p + V |\varphi|^p \right) \dx  \qquad \forall \varphi\in  C_c^\infty(\Om).$$
 In particular,  for $V=-\gl  \gd_\Gw^{-(\ga+p)}$ with $\gl\in\R$, we study the functional
 $$\mathcal{Q}_{\ga,p,-\gl \gd_\Gw^{-(\ga+p)}}(\vgf)= \int_{ \Om } \left(\delta_{\Om}^{-\al} |\nabla \varphi|^p - \gl \delta_{\Om}^{-(\al+p)} |\varphi|^p\right) \dx  \quad \forall \varphi\in  C_c^\infty(\Om).$$
 The associated Euler-Lagrange equations (up to the multiplicative constant $p$) are  
  \begin{align*} 
  {Q}_{\ga,p,V}(w):=\left(-\De_{\al,p}+ V \mathcal{I}_p \right)w =0 \qquad \mbox{in } \Gw,
 \end{align*}
 and
 \begin{align} \label{main_eq}
 {Q}_{\ga,p,-\gl \gd_\Gw^{-(\ga+p)}}(w):=\left(-\De_{\al,p} - \frac{\la }{ \delta_{\Om}^{\al+p}}\mathcal{I}_p \right) w =0 \qquad \mbox{in } \Gw.
 \end{align}
  A function $u \in W^{1,p}_{\loc}(\Om)$ is called a  {\em (weak) subsolution (resp., supersolution)} of the equation
\begin{equation}\label{eq_QV}
\left(-\De_{\al,p} + V\mathcal{I}_p \right)w = 0 \qquad \mbox{in } \Gw 
\end{equation}
 if
 $$\int_{\Om} \left(\delta_{\Om}^{-\al} |\nabla u|^{p-2} \nabla u \cdot \nabla \varphi + V |u|^{p-2}u \varphi\right) \dx\, \leq 0 \, (\mbox{resp., } \geq  0) $$
  for all nonnegative $\varphi \in C_c^{\infty}(\Om)$, and in this case we write ${Q}_{\ga,p,V}(u) \!\leq \!0$ (resp., $\geq \!0$).
  A function  $u \in W^{1,p}_{\loc}(\Om)$ is a {\em (weak) solution} of \eqref{eq_QV} if $u$ is both subsolution and supersolution of \eqref{eq_QV}. Furthermore, we write  
${Q}_{\ga,p,V}  \!\geq \!0$   in $\Om$ 
if the equation ${Q}_{\ga,p,V}(w)\!=\!0$ in $\Omega$ admits a positive (super)solution in $\Gw$. 
\subsection{Basic notions in criticality theory}
First, we quote the Agmon-Allegretto-Piepen\-brink (AAP)-type theorem (see \cite[Theorem~4.3]{Yehuda_Georgios}) stated for our particular case.
 \begin{theorem}[AAP-type theorem] \label{aap_thm} 
 	The following assertions are equivalent:
 	\begin{itemize}
 		\item[(i)] $\mathcal{Q}_{\ga,p,V}\geq 0$ on $C_c^\infty(\Omega)$;
 		\item[(ii)] the equation ${Q}_{\ga,p,V}(w)=0$ in $\Omega$ admits a positive (weak) solution;
 		\item[(iii)] the equation ${Q}_{\ga,p,V}(w)=0$ in $\Omega$ admits a positive (weak) supersolution.
 	\end{itemize}	
 \end{theorem}
We recall some basic results of criticality theory for the operator ${Q}_{\ga,p,V}$ in $\Gw$. The operator ${Q}_{\ga,p,V}\geq 0$ in $\Gw$ is {\em subcritical} in $\Gw$ if there exists a nonzero nonnegative $W\in  C_c^{\infty}(\Om)$ such that ${Q}_{\ga,p,V-W}\geq 0$ in $\Gw$,  otherwise ${Q}_{\ga,p,V}$ is {\em critical} in $\Gw$. It follows from the AAP theorem (Theorem~\ref{aap_thm}), that ${Q}_{\ga,p,V}$ is {\em critical} in $\Gw$ if and only if the equation ${Q}_{\ga,p,V}(w)=0$ in $\Gw$ admits a unique (up to a multiplicative constant) positive supersolution, and this supersolution is in fact a (unique) positive solution of the above equation (see \cite{Yehuda_Georgios} and references therein). It is called the {\em Agmon ground state} (or simply {\em ground state}) of ${Q}_{\ga,p,V}$ in $\Gw$.
\begin{definition}[Positive solution of minimal growth at infinity]{\em 
		Let $K_0$ be a compact set in $\Omega$ such that $\Gw\setminus K_0$ is connected.  A positive solution $u$
		of the equation $Q_{\ga,p,V}(w)=0$ in $\Omega\setminus K_0$ is said to be a
		{\it  positive solution of minimal growth in a neighborhood of
			infinity in} $\Omega$ (and denote it by $u\in \mathcal{M}_{\Gw\setminus K_0}$) if
		for any compact set $K$ in $\Omega$, with a smooth boundary, such
		that $\Gw\setminus K$ is connected and  $K_0 \Subset \mathrm{int}(K)$, and any positive supersolution
		$v\in C((\Omega\setminus K)\cup
		\partial K)$ of the equation $Q_{\ga,p,V}(w)=0$ in $\Omega\setminus K$,
		the inequality $u\le v$ on $\partial K$ implies that $u\le v$ in
		$\Omega\setminus K$.

A positive solution $u$ of minimal growth at infinity with respect to $K_0=\emptyset$ ($u\in \mathcal{M}_{\Gw}$) is called a {\em global minimal solution}. 
} 
\end{definition}
It turns out that $Q_{\alpha,p,V}$ admits a global minimal solution in $\Gw$ if and only if $Q_{\alpha,p,V}$ is critical in $\Gw$ (\cite[Theorem~5.9]{Yehuda_Georgios}). Hence, a global minimal solution is a ground state of the corresponding critical operator $Q_{\alpha,p,V}$. In \cite[Theorem~5.9]{Yehuda_Georgios}, it was moreover proved, under the assumption that $Q_{\ga,p,V}$ is nonnegative, that for any given $x_0\in\Om$, there exists a positive solution $u_{x_0}\in  \mathcal{M}_{\Gw\setminus \{x_0\}}$. However, a subtle issue concerning $u_{x_0}$ was left aside in this work and appears to be lacking also in the current literature, namely, the possibility that $u_{x_0}$ has a removable singularity at $x_0$ (i.e. extends to a global solution in the whole $\Omega$), but is not in $\mathcal{M}_{\Gw}$. We address this issue in Appendix~\ref{AppendixD}; namely, we show that if $u_{x_0}\in  \mathcal{M}_{\Gw\setminus \{x_0\}}$ has a removable singularity at $x_0\in \Gw$, then $u_{x_0}\in  \mathcal{M}_{\Gw}$, and therefore, $Q_{\ga,p,V}$ is critical. As a corollary, we can complete the results of \cite[Theorem~5.9]{Yehuda_Georgios} and obtain the following theorem:

    \begin{theorem}[\cite{Yehuda_Georgios} and Theorem~\ref{thm_AppendixD}]
	Let  $V\in L^{\infty}_{\loc}(\Om)$, and assume that $Q_{\ga,p,V}$ is nonnegative in $\Om$. Then for any $x_0\in\Om$, the equation $Q_{{\alpha},p,V}(w)=0$ admits a positive solution $u_{x_0}$ in $\Gw\setminus \{x_0\}$ of minimal growth in a neighborhood of
	infinity in $\Omega$. 
	
	Moreover, we have the following dichotomy: 
	
	{\em (i)} either $u_{x_0}$ has a removable singularity in $x_0$, and this occur if and only if $Q_{\ga,p,V}$ is critical in $\Gw$, and $u_{x_0}$ is an Agmon ground state, or,  
	
	{\em (ii)} $u_{x_0}$ has a nonremovable singularity at $x_0$, and this occur if and only if $Q_{\ga,p,V}$ is subcritical in $\Gw$,  and $u_{x_0}$ is called a minimal positive Green function of the operator  $Q_{\ga,p,V}$ in $\Om$ with singularity at $x_0 \in \Om$.   
\end{theorem}	
\begin{remark} \label{rem_Gr0}\rm
		If $\Gw \subsetneq \R^N$ is a $C^{1,\gamma}$-domain, then clearly $-\De_{p}$ is subcritical in $\Gw$.
	Let $0<\Psi \in W^{1,p}_{\loc}(\Om)$ be a minimal positive Green function of the operator  $-\De_{p}$ in $\Om$ with a singularity at some $x_0 \in \Om$.  It is known that $\Psi\in C^{1,\tilde{\gamma}}(\bar\Gw\setminus\{x_0\})$ for some $0<\tilde{\gamma}\leq 1$. Furthermore, the Hopf boundary point lemma holds for $-\De_{p}$ in $\Gw$ (see \cite[Section 2]{Lamberti}).
\end{remark}
 Next we recall the notion of a null-sequence.
\begin{definition}[Null-sequence] \label{def-gs}{\em 
		A nonnegative sequence $(\varphi_n) \in W^{1,p}(\Om) \cap C_c(\Om)$   is called a {\em null-sequence} with respect to the nonnegative functional $\mathcal{Q}_{\ga,p,V}$ if
		\begin{itemize}
			\item there exists a subdomain $O \Subset \Om$ such that $\|\varphi_n\|_{L^p(O)}  \asymp 1$ for all $n \in \N,$ and
			\item $\lim_{n \ra \infty} \mathcal{Q}_{\ga,p,V}(\varphi_n)=0$. 
		\end{itemize}
	}
\end{definition}
\begin{definition}[Null vs. positive-criticality] \rm
We call an operator $Q_{\ga,p,V}$ {\em null-critical} (respectively, {\em positive-critical}) in $\Gw$ with respect to a weight function $W\gneq 0$ if $Q_{\ga,p,V}$ is critical in $\Gw$ with a ground state $\Phi$ satisfying $\Phi \not\in L^p(\Gw, W\dx)$ (respectively, $\Phi \in L^p(\Gw, W\dx)$).
\end{definition}
\begin{remark}\label{rem-gs}{\em 
		The nonnegative functional $\mathcal{Q}_{\ga,p,V}$ is critical in $\Gw$ if and only if it admits a null-sequence in $\Gw$. Moreover, any null-sequence converges weakly in $L^{p}_\loc(\Gw)$ to the unique (up to a multiplicative constant) positive (super)solution of the equation $Q_{\ga,p,V}(w)=0$ in $\Gw$, hence, it converges to the ground state.  Furthermore,   there exists a null-sequence which converges locally uniformly in $\Gw$ to the ground state \cite{Yehuda_Georgios}. }
\end{remark}

\medskip 

The following lemma, which is not new, asserts that if the spectral gap is strictly positive, then the operator ${Q}_{\ga,p,- H_{\al,p}(\Om)\gd_\Gw^{-(\ga+p)}}$ is critical in $\Gw$.   
\begin{lemma}\label{lem-spectral-gap}
If the spectral gap condition  
$$\Gg_{\al,p}(\Om) = \la_{\al,p}^{\infty}(\Om)-H_{\al,p}(\Om)>0 $$ is satisfied, then 
the operator ${Q}_{\ga,p,- H_{\al,p}(\Om)\gd_\Gw^{-(\ga+p)}}$ is critical in $\Gw$. In particular, the equation
$$(-\Delta_{\ga,p} - H_{\al,p}(\Om)\gd_\Gw^{-(\ga+p)}\mathcal{I}_p)w=0$$ admits a ground state in $\Gw$. 
\end{lemma}
\begin{proof}
	The proof follows similar arguments as in \cite[Lemma~2.3]{Lamberti}.
\end{proof}
It turns out that a ground state of a critical operator admits a null-sequence which is pointwise bounded by the ground state. 
\begin{lemma}[{\cite[Lemma 5.5]{Das_Pinchover1}}]\label{lem-smaller_null-seq}
Let $(\varphi_n) \in W^{1,p}(\Om) \cap C_c(\Om)$ be a null-sequence with respect to the nonnegative functional $\mathcal{Q}_{\ga,p,V}$, and let $\Phi \in W^{1,p}_{\loc}(\Om) \cap C(\Om)$ be a  corresponding Agmon ground state. For each $n \in \N$, let $\hat{\varphi}_n=\min\{\varphi_n, \Phi\}$.  Then $(\hat\varphi_n)$ is a null-sequence for $\mathcal{Q}_{\ga,p,V}$.
\end{lemma}

\subsection{A chain rule} The following lemma is an extension of \cite[Lemma~2.10]{DP} to the operator $\De_{\al,p}$ which will be used frequently in this article.
 \begin{lemma}\label{weak_lapl}
 	Let $0<u\in W^{1,p}_\loc \cap C(\Omega)$, and let $F\in C^2(\R_+)$ satisfy $F' \geq 0$ and $(F')^{p-2}F''$ is continuous on $(0,+\infty)$. Suppose that $|F'(u)|^{p-2}|\nabla F'(u)| \in L^1_{\loc}(\Gw)$.
 	Then the following formula holds in the weak sense
 	\begin{equation}\label{weak_eq}
 	-\De_{\al,p}(F(u))=-|F'(u)|^{p-2}\left[(p-1)F''(u)\delta_{\Om}^{-\al}|\nabla u|^p+F'(u) \De_{\al,p}(u)\right].
 	\end{equation}
 	Moreover, if $\De_{\al,p}(u)\in L^1_\loc(\Gw)$, then $\De_{\al,p}(F(u)) \in L^1_\loc(\Gw)$.
 \end{lemma}
 \begin{proof}
 	Denote $-\De_{\al,p}(u)$ by $g$, and let $\varphi\in C_c^{\infty}(\Om)$. By the product and chain rules, we have
 	\begin{multline*}
 	\int_\Omega |\nabla F(u)|^{p-2}\nabla F(u) \!\cdot \! \nabla \varphi\dnu =\\
 	-\!\!\int_\Omega \!|\nabla u|^{p\!-\!2} \nabla u \!\cdot \! \nabla \!\left(|F'(u)|^{p\!-\!2}\!F'(u)\right)\!\varphi \dnu
 	+\int_\Omega\!\! |\nabla u|^{p\!-\!2}\nabla u\!\cdot \!\nabla \!\left(|F'(u)|^{p\!-\!2}\!F'(u)\varphi\right)\!\dnu ,
 	\end{multline*}
 	where $\d\nu = \delta_{\Om}^{-\al} \dx.$
 	By our assumption on $F$ and $u$, it follows that 
 	$|F'(u)|^{p-2}F'(u)\varphi \in W_c^{1,p}(\Gw)\cap C(\Gw)$. Consequently, the second term of the right-hand side in the above equality equals 
 	$$\int_\Omega |\nabla u|^{p-2}  \nabla u\cdot \nabla \left(|F'(u)|^{p-2}F'(u)\varphi\right)\dnu=
 	\int_\Omega g|F'(u)|^{p-2}F'(u)\varphi \, \dx.$$
 	Therefore,
 	\begin{multline*}
 	\int_\Omega |\nabla F(u)|^{p-2}  \nabla F(u)\cdot \nabla \varphi\dnu =\\
 	 -\int_\Omega |\nabla u|^{p-2}\nabla u\cdot \nabla \left(|F'(u)|^{p-2}F'(u)\right)\varphi \dnu
 	 +\int_\Omega g|F'(u)|^{p-2}F'(u)\varphi \, \dx.
 	\end{multline*}
 	Consequently, in the weak sense we have
 	$$-\De_{\al,p}(F(u))=-\delta_{\Om}^{-\al}|\nabla u|^{p-2} \nabla u\cdot \nabla \left(|F'(u)|^{p-2}F'(u)\right)-\De_{\al,p}(u)|F'(u)|^{p-2}F'(u).$$
 	Note that for $p\geq 2$ and $s\geq 0$, the function $\mathcal{I}_p(s):=|s|^{p-2}s$ is continuously differentiable, and $\mathcal{I}_p^{'}(s)=(p-1)|s|^{p-2}$, so the chain rule in $W^{1,p}$ implies that in this case 
 	$$\nabla \left(|F'(u)|^{p-2}F'(u)\right)=(p-1)|F'(u)|^{p-2} F''(u) \nabla u,$$
 	Therefore, in the weak sense,
 	\begin{equation}\label{derive_2}
 	\delta_{\Om}^{-\al} |\nabla u|^{p-2} \nabla u\cdot \nabla \left(|F'(u)|^{p-2}F'(u)\right)=(p-1)|F'(u)|^{p-2} F''(u) \delta_{\Om}^{-\al}|\nabla u|^p.
 	\end{equation}
 	Consider the case $p< 2$.  In light of \eqref{eq:distrib} in Appendix~\ref{app_1} (or by \cite[Theorem 2.2.6]{Kesavan_PDE} which is applicable under the hypotheses on $F$),  it follows that \eqref{derive_2} still holds true. This implies \eqref{weak_eq}, and hence clearly completes the proof of Lemma~\ref{weak_lapl}.
 \end{proof}
  \begin{remark} \label{rem_Gr}\rm
  Let $\Psi>0$ be a minimal positive Green function of the $p$-Laplacian in $\Om$ with a singularity at some $x_0 \in \Om$. We have in $\Gw\setminus \{x_0\}$
\begin{align}
    \De_{\al,p}(\Psi) \!=\!  {\rm{div}}(\delta_{\Om}^{-\al}|\nabla \Psi|^{p-2} \nabla \Psi) &=  \delta_{\Om}^{-\al} \De_p(\Psi) - \al |\nabla \Psi|^{p-2}  \frac{\nabla \Psi \cdot \nabla \delta_{\Om}}{\delta_{\Om}^{\al+1}} \label{Eq:id} \\
    &=   - \al |\nabla \Psi|^{p-2}  \frac{\nabla \Psi \cdot \nabla \delta_{\Om}}{\delta_{\Om}^{\al+1}} \nonumber \, .
\end{align}
Taking $u=\Psi$ in Lemma~\ref{weak_lapl} with  $F\in C^2(\R_+)$ satisfying its assumptions, we obtain that the following formula holds in the weak sense in $\Gw\setminus \{x_0\}$
 	\begin{equation}\label{weak_eq:Green}
 	-\De_{\al,p}(F(\Psi))=-\delta_{\Om}^{-\al}|F'(\Psi)|^{p-2}  |\nabla \Psi|^{p-2} \left[(p-1)F''(\Psi)|\nabla \Psi|^2 - \al F'(\Psi)   \frac{\nabla \Psi \cdot \nabla \delta_{\Om}}{\delta_{\Om}}\right].
 	\end{equation}
 In the sequel, we will apply \eqref{weak_eq:Green} to $F(t)=t^{\nu} \pm t^\beta>0$, where $\nu$ and $\beta$ will be two well-chosen real exponents such that  $F$ meets near $\partial \Gw$ all the requirements of Lemma ~\ref{weak_lapl}  needed to obtain \eqref{weak_eq:Green}. 
\end{remark} 
\subsection{Some integrability results} We say a function $g:\Om \rightarrow \R$ is {\em integrable near $\partial \Om$} if there exists $t>0$ such that $\int_{\Om_t}|g| \dx < \infty$. Let $f: (0,\infty) \to (0,\infty)$ be a Lipschitz function. In the following proposition, we discuss the integrability of  $f\circ \delta_{\Om}$ near $\partial \Gw$.
\begin{proposition}\label{prop_integrable}
Let $\Om$ be a domain with compact $C^{1,\gamma}$-boundary and $f\!:\!(0,\infty) \!\to\! (0,\infty)$ be a Lipschitz continuous function. Then, 
$f\circ \delta_{\Om}$ is integrable near $\partial \Om$ if and only if $f$ is integrable near $0$.
\end{proposition}
\begin{proof}
 Let $a>0$, by the coarea formula we have
\begin{align} \label{Coarea}
\int_{\Om_a} f(\delta_{\Om}) \dx = \int_{0}^a  \mathbb{H}^{N-1}(\Sigma_t) f(t) \dt \,,
\end{align} 
where $\Sigma_t:=\{x \in \Om: \delta_{\Om}(x)=t\}$ and $\mathbb{H}^{N-1}$ denotes the $(N-1)$-dimensional Hausdorff measure. Let $\tilde{\Gw}_t:=\{x \in \Om: 0< \delta_{\Om}(x)<t\} \cup \Om^c$. By \cite[Theorem 3]{Kraft}, we have 
$$\mathbb{H}^{N-1}(\Sigma_t) \geq P(\tilde{\Om}_t, \R^N)$$
for all $t >0$, where $P(\cdot,\R^N)$ is the De Giorgi perimeter with respect to $\R^N$ \cite[Definition 3.35]{AFP}. From the lower semicontinuity of $P$ \cite[Proposition 3.38]{AFP}, we obtain
\begin{align} \label{lower_bd}
\mathbb{H}^{N-1}(\Sigma_t) \geq P(\Om, \R^N)= P(\Om^c, \R^N)= \mathbb{H}^{N-1}(\partial \Om)  \,,
\end{align}
 where the equality $P(\Om, \R^N)= P(\Om^c, \R^N)$ is well known \cite[Proposition 3.38]{AFP} and the last equality holds as $\Om \in C^{1,\gamma}$ \cite[see the discussion on p. 246]{Delfour}. Further, since $\partial \Om$ is compact, we have $\mathbb{H}^{N-1}(\partial \Om) < \infty$ \cite[Theorem 5.7]{Delfour}, and subsequently,  $P(\Om, \R^N)= P(\Om^c, \R^N) < \infty$. Note that either $\Om$ or $\Om^c$ is bounded in $\R^N$. Hence, it follows from the isoperimetric inequality \cite[Theorem 3.46]{AFP} that $\mathbb{H}^{N-1}(\partial \Om)>0$.

Now we claim that $\Om$ has {\it{uniform lower density with respect to the boundary}}, i.e., there exists $t_0, \theta >0$ such that
\begin{align} \label{uni_den}
\theta \leq \frac{|B_r(x) \cap \overline{\Om}^c|_N}{|B_r(x)|_N}
\end{align}
for all $r \in (0,t_0)$ and $x \in \partial \Om$ \cite{Kraft}. Note that for any $C^{1,\gamma}$-domain $\Om$, 
the {\it{reduced boundary}} coincides with the usual boundary (see for example \cite[Example 3.1]{Farah}) and hence, $C^{1,\gamma}$-domains satisfy \eqref{uni_den} with $\theta =\frac{1}{2}$ and $t_0=1$ \cite[see the proof of Theorem 3.61]{AFP}.  Thus, it follows from \cite[theorems 3 and 4]{Kraft} that there exists $t_0, C >0$ such that
\begin{align} \label{upper_bd}
\mathbb{H}^{N-1}(\Sigma_t) \leq C < \infty \qquad  \forall t \in (0,t_0) \,,
\end{align}
provided $P(\overline{\Om}^c,\R^N)<\infty$. Using \cite[Proposition 3.38]{AFP}, we observe that $P(\overline{\Om}^c,\R^N) = P(\overline{\Om},\R^N)=P(\Om,\R^N)$, which is shown above to be equal $\mathbb{H}^{N-1}(\partial \Om) < \infty$. Hence, \eqref{upper_bd} holds.

Therefore, \eqref{Coarea}, together with \eqref{lower_bd}, and \eqref{upper_bd} imply the proposition.
\end{proof}

\begin{remark} \label{Rmk:integrability} \rm
Let $\Om$ be a domain with compact $C^{1,\gamma}$ boundary. As an immediate consequence of the above proposition, we infer that $\delta_{\Om}^{-a}$ is integrable near the boundary if $a \in [0,1)$ and it is not integrable if $a=1$. This fact will be used extensively in the subsequent sections.
\end{remark}

\section{The case $\al+p \leq 1$ }\label{sec_5}

In the present section we show that for $C^{1,\gamma}$-bounded domain, the weighted Hardy inequality \eqref{Lp_Hardy} {\em does not hold} if $\al+p \leq 1$.  In fact, in this case, the corresponding operator is critical. 

The following theorem is known for bounded Lipschitz domains \cite{Leherback1}. However, we give a short proof for the reader's convenience.
\begin{theorem} \label{Thm:nonexistence_bdd}
Let $\Gw\subset \R^N$ be a $C^{1,\gamma}$-bounded domain. If $\al+p \leq 1$, then $H_{\al,p}(\Om)=0$.  
\end{theorem}
\begin{proof}
	 Following \cite{Avkhadiev}, we consider the test function $u_{\vge}=\delta_{\Om}^{{\vge}/{p}}$, where $\vge>0$ is fixed. Clearly, $u_\vge\in C^{0,1}_{\loc}(\Om)$ (that is, $u_\vge$ is locally Lipschitz continuous function in $\Gw$).  We claim that $u_\vge \in \widetilde{W}^{1,p;\alpha}_0(\Omega)$. Since $\Gw$ is a $C^{1,\gamma}$-bounded domain, it follows from Proposition~\ref{prop_integrable} that $u_\vge \in \widetilde{W}^{1,p;\alpha}(\Om)$. Moreover, note that $C_c^{0,1}(\Om) \subset \widetilde{W}^{1,p;\alpha}_0(\Omega)$, and consider a function $G \in C^{\infty}(\R)$ such that $G(t) \leq |t|$, $G(t) =0$ if $|t| \leq 1$ and $G(t)=t$ if $|t| \geq 2$. Define $u_{n,\vge}:=\frac{1}{n}G(nu_{\vge})$. Obviously, $u_{n,\vge} \in C_c^{0,1}(\Om)$. The dominated convergence theorem implies that $u_{n,\vge} \ra u_{\vge}$ in $\widetilde{W}^{1,p;\alpha}_0(\Omega)$. Hence $u_{\vge} \in \widetilde{W}^{1,p;\alpha}_0(\Omega)$.
	
	Next we use the function $u_{\vge} \in \widetilde{W}^{1,p;\alpha}_0(\Omega)$ to prove that $H_{\al,p}(\Om)=0$. Observe that
	\begin{align*}
	\left(\frac{\vge}{p}\right)^p \int_{\Om} \frac{u_{\vge}^p}{\delta_{\Om}^{\ga+p}} \dx = \int_{\Om} \frac{|\nabla u_{\vge}|^p}{\delta_{\Om}^{\ga}} \dx.
	\end{align*}
	This implies that $H_{\al,p}(\Om) \leq ({\vge}/{p})^p$ for all $\vge>0$. Hence, $H_{\al,p}(\Om)=0$. 
\end{proof}
 Actually, for $C^{1,\gamma}$-bounded domains the following stronger result holds. 
\begin{proposition} \label{Prop:to_be_added}  
Let $\Om$ be a $C^{1,\gamma}$-bounded domain and suppose that $\al+p \leq 1$. Then
\begin{enumerate}[(i)]
\item $-\De_{\al,p}$ is positive-critical in $\Gw$ with respect to the weight $\gd_{\Om}^{-(\al+p)}$ when $\al+p<1$,
\item $-\De_{\al,p}$ is null-critical in $\Gw$ with respect to the weight $\gd_{\Om}^{-1}$ when $\al+p=1$.
\end{enumerate}
 In particular, under the above conditions, we have  $H_{\ga,p}(\Gw)=0$.
\end{proposition}

\begin{proof}
It can be verified that the sequence $(u_{n,\frac{1}{n}})$ in the proof of Theorem~\ref{Thm:nonexistence_bdd} (with $\vge=1/n$), is in fact a null-sequence with respect to the functional $\mathcal{Q}_{\al,p,0}$ when $\al+p \leq 1$. Hence,   $-\De_{\al,p}$ is critical with a ground state $\Phi=1$ (alternatively, one can simply see  that $1$ is a global solution of $-\De_{\al,p}(\varphi)=0$ of minimal growth at infinity in $\Gw$). 

Now, if $\al+p < 1$, then $\Phi=1 \in \widetilde{W}^{1,p;\alpha}(\Omega)$, and by the dominated convergence theorem, it follows that $u_{n,\frac{1}{n}} \ra \Phi=1$ in $\widetilde{W}^{1,p;\alpha}_0(\Omega)$. Hence, $\Phi\in \widetilde{W}^{1,p;\alpha}_0(\Omega)$ and it is a minimizer. This proves $(i)$.

On the other hand,  if $\al+p=1$, then as we just noticed, $\Phi=1$ is a ground state of $-\De_{\ga,p}$, and clearly, $\Phi\not\in \widetilde{W}^{1,p;\alpha}_0(\Omega)$ (by Remark \ref{Rmk:integrability}). Therefore, in this case, $-\De_{\ga,p}$ is null-critical in $\Gw$ with respect to the weight $\gd_{\Om}^{-1}$. This completes our proof.
\end{proof}
The next proposition gives the weighted Hardy constant and the Hardy constant at infinity for the case $\Gw=\R^N_+$ (the half-space). The value of the weighted Hardy constant $H_{\al,p}(\R_+^N)$ is well-known for $\al+p>1$ and it seems to be also known for $\al+p \leq 1$ (\cite[Inequality (5)]{Avkhadiev_selected}, \cite[Assertion 3.2.5, for $\al=0$]{Balinsky}). However, since we could not find an appropriate proof, we provide it here for the reader's convenience. The result for $\lambda^\infty_{\alpha,p}(\R^N_+)$ seems to be new.
\begin{proposition}\label{Non-exist_Half_space1}
	$H_{\ga,p}(\R^N_+)=\lambda^\infty_{\alpha,p}(\R^N_+)=c_{\ga,p,1}$, where $\R^N_+$ is the half-space in $\R^N$.	
\end{proposition} 
\begin{proof}
We first prove that $H_{\ga,p}(\R^N_+)=c_{\ga,p,1}$. It is well known that $H_{\al,p}(\R_+)=c_{\al,p,1}$, i.e.,
\begin{align} \label{1d_Hardy}
c_{\al,p,1} \int_0^{\infty} \frac{|\varphi(t)|^p}{t^{\al+p}}\dt \leq  \int_0^{\infty} \frac{|\varphi'(t)|^p}{t^{\al}}\dt
\qquad \forall \varphi \in C_c^{\infty}(\R_+),
\end{align}
 where the constant $c_{\al,p,1}$ is sharp \cite{Dosly}. Since any positive solution of the equation 
 $$\frac{\mathrm d}{\mathrm dt}\left( t^{-\alpha} |f^\prime|^{p-2}f^\prime\right) -c_{\al,p,1} t^{-\alpha-p}|f|^{p-2}f=0 
 \qquad \mbox{in } \R_+,$$
clearly yields a positive solution of the equation
$$\bigg(-\De_{\al,p} - \frac{c_{\al,p,1}}{ \delta_{\R^N_+}^{\al+p}}\mathcal{I}_p \bigg) w =0 \qquad \mbox{in } \R^N_+,$$
 it follows from the AAP-type theorem (Theorem~\ref{aap_thm}) that  
 $$H_{\al,p}(\R^N_+) \geq H_{\al,p}(\R_+)=c_{\al,p,1}.$$

Next, we show the converse inequality. We can of course assume that $N\geq 2$, otherwise the problem is trivial. Our proof relies on the construction of suitable ``almost minimizing'' sequences. Fix an arbitrary $\vge >0$; we will show that $H_{\al,p}(\R^N_+) \leq H_{\al,p}(\R_+)+\vge $, by evaluating both sides of the Hardy inequality along a suitable sequence $(u_n)_{n\in\N}$. We first make some preliminary remarks. Consider a function $u\in C_c^\infty(\R^N_+)$, which is of the form

$$u(x)=v(x')\varphi(x_N),$$
where $x=(x',x_N)$ and $v\in C_c^\infty(\R^{N-1})$, $\varphi\in C_c^\infty((0,\infty))$. Then, by Fubini, one has

\begin{equation}\label{eq:1d1}
\int_{\R^N_+} \frac{|u(x)|^p}{x_N^{\alpha+p}}\dx=||v||_p^p\int_0^\infty \frac{|\varphi(t)|^p}{t^{\alpha+p}}\dt.
\end{equation}
On the other hand,
$$|\nabla u|^p=(|\varphi\nabla_{x'}v|^2+|v\varphi'|^2)^{p/2}.$$
Upon using the elementary inequality
$$(a^2+b^2)^{p/2}\leq C(\vge )a^p+(1+\vge )b^p,\quad a>0,\,b>0,$$
and integrating, we conclude that
\begin{equation}\label{eq:1d2}
\int_{\R_+^N}\frac{|\nabla u(x)|^p}{x_N^\alpha}\dx\leq C(\vge ) || \, | \nabla_{x'} v| \, ||_p^p\left(\int_0^\infty \frac{|\varphi(t)|^p}{t^{\alpha}}\dt\right)+(1+\vge )||v||_p^p\left(\int_0^\infty \frac{|\varphi'(t)|^p}{t^{\alpha}}\dt\right) .
\end{equation}
We now fix the function $\varphi\in C_c^\infty((0,+\infty))$, and let $(v_n)_{n\in\N}$ be a sequence of functions with $v_n\in C_c^\infty(\R^{N-1})$ such that $||v_n||_p=1$ and $\| \, |\nabla v_n| \, \|_p\to 0$ as $n\to+\infty$. Such a sequence is easily constructed e.g. by taking $0\leq \chi\in C_c^\infty(\R)$ with $$\chi(x)=\begin{cases}
1& \  |x|\leq 1,\\
0& \ |x|\geq 2,
\end{cases}$$
and letting $v_n(x)=\frac{c}{n^{N-1}} \chi\left(\frac{|x|}{n}\right)$, $c=\left(\int_{\R}\chi\right)^{-1}$. Now, considering the quotient of \eqref{eq:1d2} by \eqref{eq:1d1} with the choice $u(x)=v_n(x')\varphi(x_N)$ and letting $n\to+\infty$ yields the inequality
$$H_{\alpha,p}(\R^N_+) \int_0^\infty \frac{|\varphi(t)|^p}{t^{\alpha+p}}\dt\leq (1+\vge ) \int_0^\infty \frac{|\varphi'(t)|^p}{t^{\alpha}}\dt.$$
Taking the minimum over all functions $\varphi\in C_c^\infty((0,\infty))$, it follows that
$$H_{\alpha,p}(\R^N_+)\leq (1+\vge )H_{\alpha,p}(\R_+).$$
Finally, letting $\vge \to 0$ gives 
$$H_{\alpha,p}(\R^N_+)\leq H_{\alpha,p}(\R_+).$$
We now prove that $\lambda^\infty_{\alpha,p}(\R_+^N)=H_{\alpha,p}(\R^N_+)$. First, one trivially has
$$\lambda^\infty_{\alpha,p}(\R_+^N)\geq H_{\alpha,p}(\R^N_+).$$
For the converse, let us fix $K\Subset \R_+^N$, and let $\varphi\in C_c^\infty(\R_+^N)$ be such that $\int_{\R^N}\frac{|\varphi(x)|^p}{x_N^{\alpha+p}}\dx=1$. Let $s>0$, and consider the function 
$$\psi(x):=s^{-\frac{\alpha+p+1}{p}}\varphi(s x).$$
Then
$$\int_{\R^N}\frac{|\psi(x)|^p}{x_N^{\alpha+p}}\dx=1,\quad \int_{\R^N}\frac{|\nabla \psi(x)|^p}{x_N^\alpha}\dx=\int_{\R_+^N}\frac{|\nabla \varphi(x)|^p}{x_N^\alpha}\dx.$$
Moreover, if one chooses the parameter $s$ large enough, then $\psi$ has compact support in $\R_+^N\setminus K$. Hence,
$$\lambda_{\alpha,p}^\infty(\R_+^N)\leq \int_{\R^N}\frac{|\nabla \psi(x)|^p}{x_N^\alpha}\dx=\int_{\R_+^N}\frac{|\nabla \varphi(x)|^p}{x_N^\alpha}\dx.$$
Taking the infimum over all $\varphi$'s, we obtain the inequality
$$\lambda_{\alpha,p}^\infty(\R_+^N)\leq H_{\alpha,p}(\R_+^N).$$
Thus, $\lambda_{\alpha,p}^\infty(\R_+^N)= H_{\alpha,p}(\R_+^N)$, which concludes the proof.
\end{proof}
By localizing everything near a boundary point of a $C^{1,\gamma}$ domain, which admits a tangent hyperplane at every boundary point, we obtain the following result.
\begin{corollary}\label{Non-exist_Half_space2}
	
	Let $\Om$ be a $C^{1,\gg}$-bounded or exterior domain. Then, 
	$$H_{1-p,p}(\Om)=\la_{1-p,p}^{\infty}(\Om) = H_{1-p,p}(\R^N_+)= 0.$$
\end{corollary}
\begin{proof}
	In light of Proposition~\ref{Non-exist_Half_space1}, it is sufficient to prove that $\la_{1-p,p}^{\infty}(\Om) = 0$. We give only a rough sketch of the proof and refer to the literature for more details. Since any boundary point admits a tangent hyperplane, one can apply the same arguments as in \cite[Theorem 4.1]{Lamberti} and \cite[Theorem 5]{MMP} and use Proposition~\ref{Non-exist_Half_space1} for $\ga+p=1$ to show that locally around any point $x_0\in \Gw$ it is possible to construct a minimizing sequence concentrating at $x_0$ such that the corresponding Rayleigh-Ritz quotient tends to zero. 	
\end{proof}

 \section{Local a-priori estimates and weak comparison principles}
First, we prove a local integral estimate for positive (super)solutions of $-\De_{\al,p}(\varphi) =0$ in $\Om$. An analogous result for $\al=0$ is proved in \cite[Proposition 2.1]{Itai}. 
\begin{lemma} \label{lem:local_est}
Let $\Om \subsetneq \R^N$ be a domain, and $0< u \in W^{1,p}_\loc(\Om)$ be such that $-\De_{\al,p}(u) \geq 0$ in $\Om$ in the weak sense. Then the following statements hold.
\begin{enumerate}[(i)]
    \item There exists $C_0>0$ such that, for every $ x \in \Om$,
    \begin{align} \label{local_est_1}
        \int_{B_{r/2}(x)} \left(\frac{|\nabla u|}{u} \right)^p \delta_{\Om}^{-\al} \dx \leq C_0 r^{-\al+N-p} \qquad  \forall r \in (0,\delta_{\Om}(x)).
    \end{align}
    \item  If, in addition, $\partial \Om$ is $C^{1,\gamma}$ for some $\gamma \in (0,1]$ and compact, then there exists $C_1>0$ such that  
    \begin{align} \label{local_est_2}
       \int_{D_r} \left(\frac{|\nabla u|}{u} \right)^{p-1} \delta_{\Om}^{-\al} \dx \leq C_1 r^{2-\al-p}
       \qquad \forall r >0,
    \end{align}
where $D_r:=\{x\in\Om\,;\,\frac{r}{2}<\delta_\Omega<r\}$.
\end{enumerate}
\end{lemma}
\begin{proof}
$(i)$ Fix $ x \in \Om $ and $r \in (0,\delta_{\Om}(x))$. Consider a cutoff function $\theta \in C^1(B_r(x))$ satisfying
\begin{align} \label{cut_off}
    0\leq \theta \leq 1 , \ \theta = 1 \ \mbox{on} \ B_{{r}/{2}}(x) , \ {\rm{supp}}(\theta)\subset B_{{2r}/{3}}(x), \ \sup (|\nabla \theta|)\leq \frac{C}{r} \,.
\end{align}
For every $\varepsilon >0$, $\frac{\theta^p}{(u+\varepsilon)^{p-1}} \in W^{1,p}_0(B_r(x))$. Testing $-\De_{\al,p}(u) \geq 0$ against this function, we obtain
\begin{align*}
    \int_{B_{r}(x)} |\nabla u|^{p-2} \nabla u \cdot \nabla \left( \frac{\theta^p}{(u+\varepsilon)^{p-1}}\right) \delta_{\Om}^{-\al} \dx \geq 0 .
\end{align*}
This implies
\begin{align*}
    \int_{B_{r}(x)} |\nabla u|^{p-2}  \left( p\left( \frac{\theta}{u+\varepsilon}\right)^{p-1} \nabla u \cdot \nabla \theta + (1-p) \theta^p \frac{|\nabla u|^2}{(u+\varepsilon)^p} \right) \delta_{\Om}^{-\al} \dx \geq 0 .
\end{align*}
Now by \eqref{cut_off}, recalling in particular that $\theta$ is supported in $B_{2r/3}(x)$, and using the H\"older inequality we get
\begin{align*}
    (p-1) \int_{B_r(x)} \left( \frac{\theta |\nabla u|}{u+\varepsilon}\right)^p \delta_{\Om}^{-\al} \dx & \leq \frac{Cp}{r} \int_{B_r(x)} \left( \frac{\theta |\nabla u|}{u+\varepsilon}\right)^{p-1} \delta_{\Om}^{-\al} \dx \\
    & \leq \frac{Cp}{r} \left[\int_{B_{2r/3}(x)}  \delta_{\Om}^{-\al} \dx \right]^{1/p} \left[\int_{B_r(x)} \left( \frac{\theta |\nabla u|}{u+\varepsilon}\right)^{p} \delta_{\Om}^{-\al} \dx \right]^{(p-1)/p} .
\end{align*}
 Since $\delta_\Om \geq \frac{r}{3}$ on $B_{2r/3}(x)$, we easily get
$$\int_{B_{2r/3}(x)}  \delta_{\Om}^{-\al} \leq C r^{-\al+N}.$$
Therefore, $ \displaystyle \int_{B_r(x)} \left( \frac{\theta |\nabla u|}{u+\varepsilon}\right)^p \delta_{\Om}^{-\al} \dx \leq Cr^{-\al+N-p}$ for some $C>0$ independent of $\varepsilon$. Thus, letting $\varepsilon \ra 0$  and recalling that $\theta\equiv1$ in $B_{r/2}(x)$, we obtain \eqref{local_est_1}.

$(ii)$ Let $x \in D_r$. By H\"older inequality and \eqref{local_est_1}
\begin{align*}
    \int_{B_{r/2}(x)} \left( \frac{|\nabla u|}{u}\right)^{p-1} \delta_{\Om}^{-\al} \dx & \leq \left[\int_{B_{r/2}(x)} \delta_{\Om}^{-\al} \dx \right]^{1/p} \left[\int_{B_{r/2}(x)} \left( \frac{|\nabla u|}{u}\right)^{p} \delta_{\Om}^{-\al} \dx \right]^{(p-1)/p} \\
    & \leq C r^{\frac{-\al+N}{p}} r^{(-\al+N-p)\frac{(p-1)}{p}} = C r^{1-\al+N-p} \,.
\end{align*}
Since $\partial\Gw$ is compact, the set $D_r$ can be covered by a finite number of balls belonging to the family $\{B_{r/3}(x) \mid x \in  \Om \mbox{ s.t. }\delta_{\Om}(x)={3r}/{4}\}$. Let $N(r)$ be the minimal number of balls needed for such a cover.  Since $\Gw$ is a $C^{1,\gg}$-domain with a compact boundary, it follows that  $N(r) \leq Cr^{1-N}$, where $C>0$ depends only on the geometry of $\Om$  (see Appendix~\ref{appendix1}).  Thus, we have \eqref{local_est_2}.
\end{proof}
Next, we prove two `non standard' weak comparison principles (WCP). The first one concerns WCP in a neighborhood of $\partial \Gw$ (similar result for $\al=0$ can be found in \cite[Proposition 3.1]{Itai}) and the second concerns WCP near infinity. 
\begin{lemma} \label{lem:comparison}
Let $\Om\subsetneq \R^N$ be a domain with compact $C^{1,\gamma}$-boundary $ \partial \Om$, and let  $\mathcal{N}_{\partial \Om}\subset \Gw$ be a relative  neighborhood of $\partial \Om$. Assume that 
	$g \in L^{\infty}(\mathcal{N}_{\partial \Om})$ and 
	$u_1,u_2 \in W^{1,p}_\loc(\mathcal{N}_{\partial \Om}) \cap C(\mathcal{N}_{\partial \Om})$ be two positive functions such that
\begin{align*} 
    \left(-\De_{\al,p}-\frac{g}{\delta_{\Om}^{\al+p}}\mathcal{I}_p \right)u_1 \leq 0 \leq \left(-\De_{\al,p}-\frac{g}{\delta_{\Om}^{\al+p}}\mathcal{I}_p \right)u_2  \ \ \mbox{in} \ \mathcal{N}_{\partial \Om} \cap \Om .
\end{align*}
In addition, suppose that the following growth condition holds true
\begin{align} \label{extra_cond}
    \liminf_{r\ra 0} \frac{1}{r}\int_{D_r} u_1^p \left[\left( \frac{|\nabla u_1|}{u_1}\right)^{p-1} + \left( \frac{|\nabla u_2|}{u_2}\right)^{p-1} \right] \delta_{\Om}^{-\al} \dx =0 ,
\end{align}
where $D_r=\{x \in  \Om: {r}/{2} < \delta_{\Om}(x)<r\}$. If 
\begin{align} \label{sub-sup:ineq}
    u_1 \leq u_2 \qquad  \mbox{ on }  \Sigma_a,
\end{align}
for some $a>0$ sufficiently small such that $\Om_a \cup \Sigma_a \subset \Om \cap \mathcal{N}_{\partial \Om}$, then 
\begin{align} \label{rslt:comp}
    u_1 \leq u_2 \qquad  \mbox{ in } \Om_a.
\end{align}
\end{lemma}
\begin{proof}
Let $a>0$ be such that $\Om_a \cup \Sigma_a \subset \Om \cap \mathcal{N}_{\partial \Om}$.
First, let us assume 
\begin{align} \label{sub-sup:strictineq}
    u_1 < u_2 \qquad \mbox{ on }  \Sigma_a .
\end{align}
Consider $h \in C^1(\R)$ such that
\begin{align} \label{h-fn}
0 \leq h \leq 1, \quad h(t) = 1 \ \mbox{ if } t \geq 1, \quad  h(t) = 0 \ \mbox{ if }  t \leq \frac{1}{2}, \quad h'(t) \geq 0  \ \mbox{ if }  t> 0.  
\end{align}
For $r \in (0,a)$, let $\psi_r$ be the function given by $\psi_r(x)=h(\frac{\delta_{\Om}(x)}{r})$ for $x \in  \Om.$ Then $\psi_r$ is a Lipschitz continuous function and
\begin{align} \label{est:hadsh}
    r \|\nabla \psi_r\|_{\infty} \leq C_0:= \sup |h'| .
\end{align}
Notice that $\psi_r=0$ in $ \{x \in \Om: 0<\delta_{\Om}(x)<\frac{r}{2}\}.$ 
By continuity, the inequality $u_1<u_2$ still holds in a neighborhood of $\Sigma_a$. In what follows, we assume that $r\in (0,a)$ is close enough to $a$, so that $u_1<u_2$ on $\Sigma_r$.  
 For such an $r$, let 
  $$w= \chi_{\Om_r}(u_{1}^p-u_2^p)_+ .$$ 
 Since $u_{1},u_2 \in W^{1,p}_\loc(\mathcal{N}_{\partial \Om} \cap \Om) \cap C(\mathcal{N}_{\partial \Om}\cap \Om)$ are positive, it follows that the functions $u_{1}^p,u_2^p \in W^{1,p}_\loc(\mathcal{N}_{\partial \Om} \cap \Om) \cap C(\mathcal{N}_{\partial \Om}\!\cap\! \Om)$. Thus, in view of \eqref{sub-sup:ineq}, $w \in W^{1,p}_\loc(\mathcal{N}_{\partial \Om} \cap \Om) \cap C(\mathcal{N}_{\partial \Om}\cap \Om)$ and
 \begin{align}\label{eq:nabla_w}
     \nabla w =  \begin{cases}
     \nabla (u_{1}^p-u_2^p)  \ \ \mbox{if} \ x \in \mathcal{N}_{\partial \Om} \cap \Om \ \mbox{and} \ u_{1}>u_2 , \\
      0 \ \ \ \ \  \qquad  \qquad  \mbox{otherwise}.
     \end{cases}
 \end{align}
 Note that, in view of \eqref{sub-sup:strictineq}, $w =0$ in a neighborhood of $\Sigma_a$. Hence, $\psi_r w u_2^{1-p} \in W^{1,p}_0( \Om_a)$. Testing the inequality $\left(-\De_{\al,p}-\frac{g}{\delta_{\Om}^{\al+p}}\mathcal{I}_p \right)u_2 \geq 0$ against this test function, we obtain
\begin{align*}
 \int_{\Om_a} |\nabla u_2|^{p-2} \nabla u_2 \cdot \nabla (\psi_r w u_2^{1-p}) \delta_{\Om}^{-\al} \dx  \geq \int_{\Om_a} \frac{g}{\delta_{\Om}^{\al+p}}\psi_r w \dx.
\end{align*}
Testing the inequality $\left(-\De_{\al,p}-\frac{g}{\delta_{\Om}^{\al+p}}\mathcal{I}_p \right)u_1 \leq 0$ against the function $\psi_r w u_1^{1-p} \in W^{1,p}_0( \Om_a)$, we obtain
\begin{align*}
 \int_{ \Om_a} |\nabla u_1|^{p-2} \nabla u_1 \cdot \nabla (\psi_r w u_1^{1-p}) \delta_{\Om}^{-\al} \dx  \leq  \int_{\Om_a} \frac{g}{\delta_{\Om}^{\al+p}}\psi_r w \dx.
\end{align*}
Subtracting the latter two inequalities, we get 
\begin{align} \label{ineq:subtract}
 \int_{ \Om_a} |\nabla u_2|^{p-2} \nabla u_2 \cdot \nabla (\psi_r w u_2^{1-p}) \delta_{\Om}^{-\al} \dx -   \int_{ \Om_a} |\nabla u_1|^{p-2} \nabla u_1 \cdot \nabla (\psi_r w u_1^{1-p}) \delta_{\Om}^{-\al} \dx  \geq 0 .
\end{align}
Now suppose that \eqref{rslt:comp} does not hold. Then
$E=\{x \in   \Om_a \mid u_1>u_2 \}$ has a positive measure.
Consider 
\begin{align}
    I_1(r) & = \int_{E} \left[|\nabla u_2|^{p-2} \nabla u_2 \cdot \nabla (wu_2^{1-p})-|\nabla u_1|^{p-2} \nabla u_1 \cdot \nabla (wu_1^{1-p})\right] \psi_r \delta_{\Om}^{-\al} \dx, \\
I_2(r) &= \int_{E}  w \left[ \frac{|\nabla u_2|^{p-2} \nabla u_2}{u_2^{p-1}} -\frac{|\nabla u_1|^{p-2} \nabla u_1}{u_1^{p-1}}\right] \cdot \nabla \psi_r \delta_{\Om}^{-\al} \dx  .
\end{align}
By \eqref{ineq:subtract}, we have
$$I_1(r)+I_2(r) \geq 0 , \ \forall r \in (0,a) .$$
Using \eqref{est:hadsh}, we estimate
\begin{align*}
  |I_2(r)| \leq \frac{C_0}{r} \int_{D_r} u_1^p\left[\left( \frac{|\nabla u_1|}{u_1}\right)^{p-1} + \left( \frac{|\nabla u_2|}{u_2}\right)^{p-1} \right] \delta_{\Om}^{-\al} \dx \qquad \forall r \in (0,a) .
\end{align*}
Consequently, by \eqref{extra_cond}, we have $\liminf_{r \ra 0} |I_2(r)|=0$. On the other hand, we claim that $I_1(r) \leq 0$ for all $r \in (0,a)$. Indeed, using \eqref{eq:nabla_w}  one sees that $I_1(r)$  can be written in the form
$$I_1(r)=-\int_{E} H(u_1,u_2,\nabla u_1, \nabla u_2) \psi_r \delta_{\Om}^{-\al} \dx ,$$
where the function $H$ is given by
\begin{align}
    H(t_1,t_2,x_1,x_2) & \!= \!\left[1+(p-1)\bigg(\frac{t_2}{t_1}\bigg)^{\!p} \right]|x_1|^p -p \left[|x_1|^{p-2}\bigg(\frac{t_2}{t_1}\bigg)^{p-1} \!+ |x_2|^{p-2}\bigg(\frac{t_1}{t_2}\bigg)^{p-1} \right] x_1 \!\cdot\! x_2 \nonumber \\
    & \qquad \qquad \qquad \qquad \qquad \qquad + \left[1+(p-1)\bigg(\frac{t_1}{t_2}\bigg)^p \right]|x_2|^p 
\end{align}
for all $t_1,t_2 \in \R^+$, $x_1,x_2 \in \R^N.$ It was proved by Anane \cite{Anane} that $H(t_1,t_2,x_1,x_2) \geq 0$ for all $t_1,t_2 \in \R^+$, $x_1,x_2 \in \R^N.$ Thus, $I_1(r) \leq 0.$ Since $I_1(r)$ is non-decreasing, it follows that either $I_1(r) =0$ for all $r \in (0,a)$, or there exist $a_0>0$ and $\gamma_0>0$ such that $I_1(r) \leq - \gamma_0$ for all $r \in (0,a_0)$. This is not possible as $I_1+I_2 \geq 0$ and therefore, $\limsup_{r \ra 0} I_1(r) \geq 0.$ On the other hand, if $I_1(r) = 0$ for all $r \in (0,a)$, then $H(u_1,u_2, \nabla u_1, \nabla u_2) =0$ on $E$. Thus, it follows from \cite[Proposition~1]{Anane} that $u_1 \nabla u_2 - u_2 \nabla u_1 =0$ in $E$. Hence ${u_2}/{u_1}$ is constant in every connected component of the open set $E$. Furthermore, there exists such a component $E'$ with $\partial E' \cap \Om \neq \varnothing$, and if $\xi \in \partial E' \cap \Om \neq \varnothing$, then $u_1(\xi) = u_2(\xi)$. Since ${u_2}/{u_1}$ is constant in $E'$, it follows that $u_1=u_2$ in $E'$, a contradiction to the definition of $E$. This contradiction shows that $E$ cannot have a positive measure. Since $E$ is open, it is empty.

In the general case, assuming only \eqref{sub-sup:ineq}, we apply the above result to the functions $u_1$ and $(1+\varepsilon)u_2$ where $\varepsilon >0$.  It follows that $u_1 \leq (1+\varepsilon)u_2$ in $ \Om_a$ and $\varepsilon>0$ is independent of $a$.  Since this inequality
holds for arbitrary $\varepsilon >0$, we obtain the desired result. 
\end{proof}
Next we prove a similar WCP near infinity for an exterior domain $\Om\!\subsetneq \!\R^N$. Let us recall that  for $R \!>\!0$, set $$\Om^R:=\{x\in \Om: \delta_{\Om}(x)>R\}, \ \ \Sigma_R:=\{x\in \Om: \delta_{\Om}(x)=R\} .$$  
\begin{lemma} \label{lem:comparison1}
Let $\Om\subsetneq \R^N$ be an exterior domain with compact boundary and let $R_0>0$ be sufficiently large. Assume that $g \in L^{\infty}(\Om^{R_0})$ and $u_1,u_2 \in W^{1,p}_\loc(\Om^{R_0}) \cap C(\Om^{R_0})$ be two positive functions such that
\begin{align} \label{sub-sup:eq1}
    \left(-\De_{\al,p}-\frac{g}{\delta_{\Om}^{\al+p}}\mathcal{I}_p \right)u_1 \leq 0 \leq \left(-\De_{\al,p}-\frac{g}{\delta_{\Om}^{\al+p}}\mathcal{I}_p \right)u_2  \quad  \mbox{in }  \Om^{R_0} .
\end{align}
In addition, suppose that  the following growth condition holds true
\begin{align} \label{extra_cond_inf}
    \liminf_{R \ra \infty} \frac{1}{R}\int_{D^R} u_1^p \left[\left( \frac{|\nabla u_1|}{u_1}\right)^{p-1} + \left( \frac{|\nabla u_2|}{u_2}\right)^{p-1} \right] \delta_{\Om}^{-\al} \dx =0 ,
\end{align}
where $D^R=\{x \in  \Om: {R}/{2} < \delta_{\Om}(x)<R\}$. If 
$    u_1 \leq u_2  \mbox{ on }  \Sigma_{R_1}$
for some $R_1>R_0$, then 
$    u_1 \leq u_2  \mbox{ in } \Om^{R_1}$.
\end{lemma}
\begin{proof}
We replace $r, \Gw_r, \Gs_r$, etc.  in the proof of Lemma~\ref{lem:comparison}, with $R, \Gw^R, \Gs_R$, etc., and take $h \in C^1(\R)$ such that  
	\begin{align} \label{h-fn1}
	0 \leq h \leq 1, \quad h(t) = 0 \ \mbox{ if } t \geq 1, \quad  h(t) = 1 \ \mbox{ if }  t \leq \frac{1}{2}, \quad h'(t) \geq 0  \ \mbox{ if }  t> 0.  
	\end{align}
	 Then the proof of the lemma is literally the same as the proof of Lemma~\ref{lem:comparison}, and therefore, it is omitted.  
\end{proof}
\section{Construction of sub and supersolutions}
We assume that $\al+p \neq 1$. First we  construct sub and supersolutions near the compact boundary of a domain $\Gw\in C^{1,\gg}$  for the equation  
\begin{equation}\label{eq:mu}
(-\De_{\al,p}-\gm \gd^{-(\ga+p)}_{\Om}\mathcal{I}_p)w=0 \qquad \mbox{ in } \Gw
\end{equation}
using Agmon's trick (see, \cite[Lemma~3.4]{Lamberti} and references therein).  As a motivation, we note that in the model case $\Omega=\R_+^N$, the function $\delta_{\R_+^N}^\nu=x^\nu_N$ is a solution of \eqref{eq:mu}, if and only if $\nu$ and $\mu$ satisfy the following indicial equation
$$|\nu|^{p-2}\nu[\al+(1-\nu)(p-1)] = \mu.$$
 We note that for every $p \in (1,\infty)$, the function $\nu \mapsto |\nu|^{p-2}\nu[\al+(1-\nu)(p-1)]$ is nonnegative and strictly decreasing either on the interval $[\frac{\al+p-1}{p},\frac{\al+p-1}{p-1}]$ if  $\al+p>1$, or $[\frac{\al+p-1}{p},0]$  if $\al+p<1$, and attains its maximum value $c_{\al,p,1}=\left| \frac{\al+p-1}{p}\right|^p$ at $\nu=\frac{\al+p-1}{p}$.
Therefore, the equation 
$$ |\nu|^{p-2}\nu[\al+(1-\nu)(p-1)] = \mu \qquad  \mbox{for } \mu \in [0,c_{\al,p,1}] $$
has exactly one root in $[\frac{\al+p-1}{p},\frac{\al+p-1}{p-1}]$ if $\al+p>1$ or in $[\frac{\al+p-1}{p},0]$ if $\al+p<1$. This root will be denoted by $\nu_{\al,p}(\mu)$. Notice that $\nu_{\al,p}(c_{\al,p,1})=\frac{\al+p-1}{p}$. 
\begin{lemma} \label{lem:sub_sup_per_Agmon2}
    Let $\Om$ be a $C^{1,\gamma}$-domain with compact  boundary, and let $0<\Psi \in W^{1,p}_{\loc}(\Om)\cap C^{1,\tilde{\gamma}}(\Om)$ be a minimal positive Green function of $-\De_{p}$ in $\Om$.  Let $\nu < \beta < \nu + \tilde{\gamma}$, such that  either $\nu , \beta \in [\frac{\al+p-1}{p},\frac{\al+p-1}{p-1}]$ if $\al+p>1$, or $\nu , \beta \in [\frac{\al+p-1}{p},0]$ if $\al+p<1$. Then there exists an open neighborhood  $ \mathcal{N}_{\partial \Om}\subset \Gw$ of $\partial \Om$ such that the functions $\Psi^{\nu} + \Psi^{\beta }$, and $\Psi^{\nu} - \Psi^{\beta }$ are positive subsolution and supersolution, respectively, 
for the equation $\left(-\De_{\al,p} - \frac{\la_{\nu}}{\delta_{\Om}^{\al+p}}\mathcal{I}_p\right)w=0$ in $\Om \cap  \mathcal{N}_{\partial \Om}$, where $\la_{\nu}=|\nu|^{p-2}\nu[\al+(1-\nu)(p-1)]$. Moreover, $ \mathcal{N}_{\partial \Om}$ can be chosen to be independent of small perturbations of $\nu$ and $\beta$.
\end{lemma}
\begin{proof}
First we consider the case of subsolutions. By \eqref{weak_eq:Green}, it follows that
\begin{align*}
&  -\De_{\al,p}(\Psi^{\nu}+\Psi^{\beta}) 
 \!=\! -|\nu \Psi^{\nu -1}+\beta \Psi^{\beta -1}|^{p-2} \bigg[(p-1)(\nu (\nu -1)\Psi^{\nu -2}+\beta (\beta -1)\Psi^{\beta -2} ) |\nabla \Psi|^{2}  \\
 & \qquad \qquad \qquad \qquad \qquad \qquad \qquad \qquad - \al  (\nu \Psi^{\nu -1}+\beta \Psi^{\beta -1}) \bigg(\frac{\nabla \Psi \cdot \nabla \delta_{\Om}}{\delta_{\Om}}\bigg) \bigg] \delta_{\Om}^{-\al} |\nabla \Psi|^{p-2}  \\
& = \Psi^{\nu (p-1)} |\nu + \beta \Psi^{\beta -\nu}|^{p-2} \bigg[ (p-1)(\nu(1-\nu)+\beta(1-\beta)\Psi^{\beta - \nu}) |\nabla \Psi|^2 \\
 & \qquad \qquad \qquad \qquad \qquad \qquad \qquad \qquad + \al  (\nu +\beta \Psi^{\beta -\nu}) \bigg(\Psi \frac{\nabla \Psi \cdot \nabla \delta_{\Om}}{\delta_{\Om}}\bigg) \bigg] \delta_{\Om}^{-\al} \frac{|\nabla \Psi|^{p-2}}{\Psi^p}\, .
\end{align*}
Now, from \cite[Lemma 3.2 and Assertion (3.16)]{Lamberti} it follows that
\begin{equation}\label{eq_nablaG}
  \frac{\delta_{\Om}\nabla \Psi \cdot \nabla \delta_{\Om}}{\Psi} =  1 +O(\delta_{\Om}^{\tilde{\gamma}}) \quad \mbox{and} \quad    \frac{|\nabla \Psi|^2 \delta_{\Om}^2}{\Psi^2} = 1+O(\delta_{\Om}^{\tilde{\gamma}}) ,
 \end{equation}
uniformly in a relative neighborhood  of $\partial \Om$ as  $\partial \Om$ is compact.
Consequently,  
\begin{align*}
&  -\De_{\al,p}(\Psi^{\nu}+\Psi^{\beta})
 = \Psi^{\nu (p-1)} |\nu + \beta \Psi^{\beta -\nu}|^{p-2} \bigg[ (p-1)(\nu(1-\nu)+\beta(1-\beta)\Psi^{\beta - \nu}) \\
& \qquad \qquad \qquad \qquad \qquad \qquad \qquad \qquad + \al  (\nu +\beta \Psi^{\beta -\nu}) (1+O(\delta_{\Om}^{\tilde{\gamma}}))  \bigg] \delta_{\Om}^{-\al} \left|\frac{\nabla \Psi}{\Psi}\right|^p\\ 
 & = \! \Psi^{\nu (p-1)} |\nu \!+\! \beta \Psi^{\beta -\nu}|^{p-2} \!\!\left[ \frac{\la_{\nu}}{|\nu|^{p-2}} \!+\! \frac{\la_{\beta}}{|\beta|^{p-2}} \Psi^{\beta - \nu} \!\!+\!  \al (\nu \!+\! \beta \Psi^{\beta - \nu}) O(\delta_{\Om}^{\tilde{\gamma}}) \!\right] \!\!\! \left[ \frac{1}{\delta_{\Om}^{\al+p}} \!+\!  O(\delta_{\Om}^{\tilde{\gamma}-(\al+p)})\right]\!\! \,,  
\end{align*}
in a relative neighborhood  of $\partial \Om$ close to  $\partial \Om$.

  Thus, in order to guarantee that $\Psi^{\nu}+\Psi^{\beta}$
is a subsolution as required in the statement, it
suffices to impose the condition
\begin{align*}
\Psi^{\nu (p-1)} |\nu \!+\! \beta \Psi^{\beta -\nu}|^{p-2} \left[ \frac{\la_{\nu}}{|\nu|^{p-2}} \!+\! \frac{\la_{\beta}}{|\beta|^{p-2}} \Psi^{\beta - \nu} \!+\!  \al(\nu \!+\! \beta \Psi^{\beta - \nu}) O(\delta_{\Om}^{\tilde{\gamma}}) \!\right] \!\! \left[ \frac{1}{\delta_{\Om}^{\al+p}}+ O(\delta_{\Om}^{\tilde{\gamma}-(\al+p)}) \right] \\ \leq \frac{\la_{\nu}}{\delta_{\Om}^{\al+p}} (\Psi^{\nu}+\Psi^{\beta})^{p-1} \,,
\end{align*}
which can be written in the form
\begin{align*} 
|\nu \!+\! \beta \Psi^{\beta -\nu}|^{p-2} \!\!\left[ \frac{\la_{\nu}}{|\nu|^{p-2}} \!+\! \frac{\la_{\beta}}{|\beta|^{p-2}} \Psi^{\beta - \nu} \!+\!  \al (\nu \!+\! \beta \Psi^{\beta - \nu}) O(\delta_{\Om}^{\tilde{\gamma}})\!\right] \!\! \left[ 1\!+\! O(\delta_{\Om}^{\tilde{\gamma}})\right] 
\!\leq\! \la_{\nu} (1\!+\!\Psi^{\beta - \nu})^{p-1} .
\end{align*}
Since $\Psi^{\beta -\nu}=0$ on $\partial \Om$, by expanding both sides of the above inequality in $\Psi^{\beta - \nu}$ to the first order, we obtain 
\begin{align} \label{last}
    [\la_{\nu}+A\Psi^{\beta - \nu} +o(\Psi^{\beta-\nu}) + O(\delta_{\Om}^{\tilde{\gamma}}) ] \left[ 1 + O(\delta_{\Om}^{\tilde{\gamma}})\right] \leq  \la_{\nu} (1 + (p-1)\Psi^{\beta -\nu}+ o(\Psi^{\beta -\nu})) \,,
\end{align}
where $$A=(p-2)\frac{\la_{\nu} \beta}{\nu} + \la_{\beta} \frac{|\nu|^{p-2}}{|\beta|^{p-2}} \,. $$
Note that since $\Psi(x)$ is asymptotic to $\delta_{\Om}(x)$ as $x \ra \partial \Om$ (by a Hopf-type lemma, see \cite{Lamberti}) and $\beta - \nu < \tilde{\gamma}$, we have
that $\frac{\delta_{\Om}(x)^{\tilde{\gamma}}}{\Psi^{\beta - \nu}(x)} \ra 0$ as $x \ra \partial \Om$. Moreover,  by a direct computation and using the choice of $\nu, \beta$ (depending on $\al+p<1$ or $\al+p > 1$) and the condition $\frac{\al+p-1}{p} \leq \nu < \beta$,
one verifies that $A< \la_{\nu}(p-1)$, see Proposition \ref{Prop:computation} in Appendix \ref{app2}.

Thus, it follows that condition \eqref{last} is satisfied in $\mathcal{N}_{\partial \Om}\subset  \Gw$ where $ \mathcal{N}_{\partial \Om}$ is a suitable relative
neighborhood  of $\partial \Om$ which can be chosen to be independent of $\nu$ and $\beta$ if $\nu, \beta$ are as in
the statement and belong to small neighborhoods of two fixed parameters $\nu_0, \beta_0$ satisfying
the conditions $ \frac{\al+p-1}{p} \leq \nu_0 < \beta_0$.

We now consider the case of supersolution. Proceeding as above, we see that in order to guarantee that $\Psi^{\nu}-\Psi^{\beta}$
is a positive supersolution as required in the
statement, without loss of generality, we may assume that $\Psi^{\nu}-\Psi^{\beta}$
is positive in a relative neighborhood of $\partial \Om$. So, it suffices to impose the condition
\begin{align*}
\Psi^{\nu (p-1)} |\nu - \beta \Psi^{\beta -\nu}|^{p-2} \left[ \frac{\la_{\nu}}{|\nu|^{p-2}} - \frac{\la_{\beta}}{|\beta|^{p-2}} \Psi^{\beta - \nu} -  \al (\nu -\beta \Psi^{\beta - \nu}) O(\delta_{\Om}^{\tilde{\gamma}}) \right]  \left[ 1 - O(\delta_{\Om}^{\tilde{\gamma}})\right] \\ \geq \la_{\nu} (\Psi^{\nu}-\Psi^{\beta})^{p-1} 
\end{align*}
in a relative neighborhood  of $\partial \Om$.  The latter inequality can be written as
\begin{align*} 
|\nu \!- \! \beta \Psi^{\beta -\nu}|^{p-2} \! \! \left[ \frac{\la_{\nu}}{|\nu|^{p-2}} \!-\! \frac{\la_{\beta}}{|\beta|^{p-2}} \Psi^{\beta - \nu} \! \! \! - \!   \al (\nu -\beta \Psi^{\beta - \nu}) O(\delta_{\Om}^{\tilde{\gamma}}) \!\right] \! \! \! \left[ 1 - O(\delta_{\Om}^{\tilde{\gamma}})\! \right] \! \! \geq \! \la_{\nu} (1-\Psi^{\beta - \nu})^{p-1} \,.
\end{align*}
Expanding both sides we arrive at 
\begin{align*} 
[\la_{\nu}-A\Psi^{\beta - \nu} +o(\Psi^{\beta-\nu}) - O(\delta_{\Om}^{\tilde{\gamma}}) ] \left[ 1 - O(\delta_{\Om}^{\tilde{\gamma}})\right] \geq  \la_{\nu} (1 - (p-1)\Psi^{\beta -\nu}+ o(\Psi^{\beta -\nu})) \,.
\end{align*}
 Again, since $A< \la_{\nu}(p-1)$, where $A$ is the same constant defined above, it follows
as in the case of subsolution that $\Psi^{\nu}-\Psi^{\beta}$ is a positive supersolution as desired.
\end{proof}
\begin{corollary} \label{Cor:lam_inf_est}
 Let $\Om$ be a $C^{1,\gamma}$-bounded domain. Then $\la_{\al,p}^{\infty}(\Om) = c_{\al,p,1}.$
\end{corollary}
\begin{proof}
The case  $\al+p=1$ follows from Corollary \ref{Non-exist_Half_space2}.
Therefore, we may assume that $\al+p \neq 1$. In view of Lemma \ref{lem:sub_sup_per_Agmon2}, we see that for any $\nu < \beta < \nu + \tilde{\gamma}$ such that either  
$\nu , \beta \in [\frac{\al+p-1}{p},\frac{\al+p-1}{p-1}]$ when $\al+p>1$, or $\nu , \beta \in [\frac{\al+p-1}{p},0)$ when $\al+p<1$, we have a positive supersolution $\Psi^{\nu}-\Psi^{\beta}$ of 
the equation $-\De_{\al,p}(\varphi) = \frac{\la_{\nu}}{\delta_{\Om}^{\al+p}}\mathcal{I}_p(\varphi)$ near the boundary of $\Om$, where $\la_{\nu}=|\nu|^{p-2}\nu[\al+(1-\nu)(p-1)] \in [0,c_{\al,p,1}]$. Thus, $\la_{\nu} \leq \la_{\al,p}^{\infty}(\Om)$ for all $ \nu \geq \frac{\al+p-1}{p}$.  Choosing $\la_{\nu}=c_{\al,p,1}$, it follows $\la_{\al,p}^{\infty}(\Om) \geq  c_{\al,p,1}$.

The converse inequality follows from Proposition \ref{Non-exist_Half_space1} by using arguments similar to those used in the proof of Corollary \ref{Non-exist_Half_space2}.
\end{proof}

The rest of this section concerns exterior domains. The first proposition gives the asymptotic behaviour at infinity of the distance function.
\begin{proposition} \label{est:asymp_1}
Let $\Om$ be an exterior domain. Then 
$$\frac{\nabla \delta_{\Om} \cdot x}{\delta_{\Om}(x)} =  1+ O(\delta^{-1}_{\Om}(x)) \qquad \mbox{as } |x|\to \infty.$$
\end{proposition}
\begin{proof}
For any $x \in \Om$, let $P(x)$ be a point on $\partial \Om$ such that $\delta_{\Om}(x)=|x-P(x)|$. It is well known (see for example, \cite{AM98}) that $\delta_{\Om}$ is differentiable at $x$ if and only if the nearest point $P(x)$ is achieved uniquely on the boundary,  and in this case $$\nabla \delta_{\Om}=\frac{x-P(x)}{|x-P(x)|} = \frac{x-P(x)}{\delta_{\Om}(x)} \,.$$
 Since $\delta_\Gw$ is differentiable a.e. in $\Om$, the above identity is valid a.e. in $\Om$. Thus,
$$\frac{\nabla \delta_{\Om} \cdot x}{\delta_{\Om}(x)}=\frac{x-P(x)}{\delta^2_{\Om}(x)} \cdot x= \frac{|x|^2}{\delta^2_{\Om}(x)}-\frac{P(x)\cdot x}{\delta^2_{\Om}(x)} \qquad \mbox{for a.e. } x \in \Om.$$
Since $|x| \!\sim \!\delta_{\Om}(x)$ near infinity and $|P(x)|$ is bounded, it follows that $\frac{\nabla \delta_{\Om} \cdot x}{\delta_{\Om}(x)} =  1+ O(\delta^{-1}_{\Om}(x))$ near infinity.
\end{proof}
In the following lemma, we assume $\al+p \neq N$. Using Agmon's method, we construct sub and supersolution near infinity for the operator $-\De_{\al,p}-\gm \gd^{-(\ga+p)}_{\Om}\mathcal{I}_p$ in an exterior domain.
If $\al+p<N$ (resp., $\al+p>N$), the function $\nu \mapsto |\nu|^{p-2} \nu [(\al-N+1)+(1-\nu)(p-1)]$  is nonnegative and strictly increasing on the interval $[\frac{\al+p-N}{p-1},\frac{\al+p-N}{p}]$ (resp., $[0,\frac{\al+p-N}{p}]$) and attains its maximum over $[\frac{\al+p-N}{p-1},\frac{\al+p-N}{p}]$ (resp., $[0,\frac{\al+p-N}{p-1}]$) at $\nu=\frac{\al+p-N}{p}$, and  its maximum value is $c_{\al,p,N}=\left| \frac{\al+p-N}{p}\right|^p$. Therefore, the indicial equation 
$$ |\nu|^{p-2} \nu [(\al-N+1)+(1-\nu)(p-1)] = \mu \qquad  \mu \in [0,c_{\al,p,N}], $$
has exactly one root in $[\frac{\al+p-N}{p-1},\frac{\al+p-N}{p}]$ (resp., $[0,\frac{\al+p-N}{p}]$). This root is denoted by $\nu_{\al,p,N}(\mu)$. Notice that $\nu_{\al,p,N}(c_{\al,p,N})=(\al+p-N)/p$.
\begin{lemma} \label{lem:agmon_subsup_inf}
	Let $\Om\subsetneq \R^N$ be a $C^{1,\gamma}$-exterior domain, and suppose that $\al+p \neq N$. Assume that $\nu \in [\frac{\al+p-N}{p-1},\frac{\al+p-N}{p}]$  if $\al+p<N$, and $\nu \in [0,\frac{\al+p-N}{p}]$ if $\al+p>N$  respectively, and $\beta \in \R$ be such that $\beta<\nu <\beta +1$. Then there exists $R>0$ such that the function $U_{+}:=|x|^{\nu} + |x|^{\beta}$ (resp., $U_{-}:=|x|^{\nu} - |x|^{\beta}$) is a positive sub (resp., supersolution) of the equation $\left(-\Delta_{\al,p}- \frac{ \hat{\la}_{\nu}}{\delta_{\Om}^{\al+p}}\mathcal{I}_p \right) w=0$ in $\Om^R$ uniformly in $\nu$, where $\hat{\la}_{\nu}=|\nu|^{p-2}\nu [(\al-N+1)+(1-\nu)(p-1)]$.
\end{lemma}
\begin{proof}
	First we consider the case of subsolution. By taking $F(t)=t^{\nu} + t^{\beta}$ and $u(x)=|x|$ in  \eqref{weak_eq}, it follows that
	\begin{align*}
		& -\De_{\al,p}(U_+) 
	 = -\bigg|\nu |x|^{\nu -1}  +  \beta |x|^{\beta -1}\bigg|^{p-2} \bigg[(p-1)(\nu (\nu -1) |x|^{\nu -2}\!+\!\beta (\beta -1) |x|^{\beta -2} ) \delta_{\Om}^{-\al} |\nabla |x| |^{p} \\
		& \qquad\qquad \qquad \qquad \qquad \qquad \qquad \qquad \qquad \qquad \qquad \qquad \qquad   + (\nu |x|^{\nu -1}+\beta |x|^{\beta -1}) \De_{\al,p}(|x|) \bigg] \\
		 &= - |\nu |x|^{\nu -1}+\beta |x|^{\beta -1}|^{p-2} \bigg[(p-1)(\nu (\nu -1) |x|^{\nu -2}+\beta (\beta -1) |x|^{\beta -2} )  \\
		&\qquad \qquad \qquad \qquad \qquad \qquad \qquad \qquad + \frac{(\nu |x|^{\nu -1}+\beta |x|^{\beta -1})}{|x|} \left((N-1) - \al   \frac{x \cdot \nabla \delta_{\Om}}{\delta_{\Om}} \right) \bigg]\delta_{\Om}^{-\al}\\    
		& =  |x|^{\nu (p-1)} |\nu + \beta |x|^{\beta -\nu}|^{p-2} \bigg[ (p-1)(\nu(1-\nu)+\beta(1-\beta) |x|^{\beta - \nu}) \\
		& \qquad\qquad \qquad \qquad \qquad \qquad \qquad \qquad \qquad + (\nu+\beta |x|^{\beta - \nu}) \left( \al  \frac{x \cdot \nabla \delta_{\Om}}{\delta_{\Om}} -(N-1) \right) \bigg]\frac{\delta_{\Om}^{-\al}}{|x|^p}\,.  
		\end{align*}
		Therefore, there exists $R>0$ large enough such that
	\begin{align*}
		& -\De_{\al,p}(U_+)
		 =   |x|^{\nu (p-1)} |\nu + \beta |x|^{\beta -\nu}|^{p-2} \bigg[ (p-1)(\nu(1-\nu)+\beta(1-\beta) |x|^{\beta - \nu}) \\
		& \qquad \qquad \qquad  + (\nu+\beta |x|^{\beta - \nu})\left( \al ( 1 + O(\delta^{-1}_{\Om})) -(N-1) \right) \bigg] \frac{\delta_{\Om}^{-\al}}{|x|^p} \\
		& = \! |x|^{\nu (p-1)}  |\nu \!+\! \beta |x|^{\beta -\nu}|^{p-2}\!\! \left[\! \frac{\hat{\la}_{\nu}}{|\nu|^{p-2}} \!+\! \frac{\hat{\la}_{\beta}}{|\beta|^{p-2}} |x|^{\beta - \nu} + \al(\nu \!+\! \beta |x|^{\beta - \nu}) O(\delta^{-1}_{\Om}) \right] \\ 
		 & \qquad \qquad \qquad \qquad \qquad \qquad \qquad \qquad \qquad \qquad\qquad\qquad\qquad 
		 \times \left[ \frac{1}{\delta_{\Om}^{\al+p}} +  O(\delta_{\Om}^{-1})\delta_{\Om}^{-(\al+p)}\right]  
	\end{align*}
	in $\Gw^R$.  The first asymptotic uses Proposition \ref{est:asymp_1} and  the latter one uses the fact  that $$\delta_{\Om}^{-p} = |x|^{-p} + O(\delta_{\Om}^{-1}) \delta_{\Om}^{-p} \qquad  \mbox{as } |x| \ra \infty,$$ 
	which is valid as $\partial \Om$ is compact.
	Thus, in order to guarantee that $|x|^{\nu}+|x|^{\beta}$
	is a subsolution as required in the statement, it
	suffices to impose the condition
	\begin{align*}
		|x|^{\nu (p-1)} |\nu + \beta |x|^{\beta -\nu}|^{p-2} \left[ \frac{\hat{\la}_{\nu}}{|\nu|^{p-2}} + \frac{\hat{\la}_{\beta}}{|\beta|^{p-2}} |x|^{\beta - \nu} +  \al(\nu + \beta |x|^{\beta -\nu}) O(\delta^{-1}_{\Om}) \right]  \left[ 1 + O(\delta_{\Om}^{-1})\right] \\ \leq \hat{\la}_{\nu} (|x|^{\nu}+ |x|^{\beta})^{p-1} \,,
	\end{align*}
	in a suitable neighborhood  of infinity. The above condition can be written in the form
	\begin{align} \label{Req:ineq_sub_inf}
		|\nu + \beta |x|^{\beta -\nu}|^{p-2} \left[ \frac{\hat{\la}_{\nu}}{|\nu|^{p-2}} + \frac{\hat{\la}_{\beta}}{|\beta|^{p-2}} |x|^{\beta - \nu} + \al(\nu + \beta |x|^{\beta -\nu}) O(\delta^{-1}_{\Om}) \right]  \left[ 1 +  O(\delta_{\Om}^{-1})\right] \nonumber \\ \leq \hat{\la}_{\nu}  (1+|x|^{\beta - \nu})^{p-1} \,.
	\end{align}
	By the assumptions on $\nu, \beta$, we have $|x|^{\beta -\nu} \ra 0$ as $|x| \ra \infty$.  Thus,  by expanding both sides of \eqref{Req:ineq_sub_inf} in $|x|^{\beta - \nu}$ to the first order, we obtain
	\begin{equation} \label{ineq:la_tilde_sub_inf}
		\left[\hat{\la}_{\nu}+A |x|^{\beta - \nu} + o(|x|^{\beta - \nu})+ O(\delta^{-1}_{\Om})\right] \left[  1 +  O(\delta_{\Om}^{-1})\right]   \leq  \hat{\la}_{\nu}  (1 + (p-1)|x|^{\beta -\nu}+ o(|x|^{\beta -\nu})) ,
	\end{equation}
	where $$A=(p-2)\frac{\hat{\la}_{\nu} \beta}{\nu} + \hat{\la}_{\beta} \frac{|\nu|^{p-2}}{|\beta|^{p-2}} \,. $$
	
	Now,  by a direct computation and using the choice of $\nu, \beta$ (depending on $\al+p<N$ or $\al+p > N$) and the condition $\beta<\nu\leq (\frac{\al+p-N}{p}) $,
	one can verify that $A< \hat{\la}_{\nu}(p-1)$, see Proposition \ref{Prop:computation} in Appendix \ref{app2}. Thus, in order to verify that \eqref{ineq:la_tilde_sub_inf} is satisfied near infinity,  it suffices to verify
	that $O(\delta_{\Om}^{-1})|x|^{\nu-\beta}=o(1)$ as $|x| \ra \infty$. This is satisfied as $\nu -\beta<1$ and $|x|$ is asymptotic to $\delta_{\Om}(x)$ near infinity.
	
	We now consider the case of supersolution. Proceeding as above, we see that due to the choice of $\nu,\beta$, the function  $U_{-}=|x|^{\nu}-|x|^{\beta}$
	is indeed positive in a neighborhood  of infinity. So, it suffices to impose the condition
	\begin{align*}
		|x|^{\nu (p-1)} |\nu - \beta |x|^{\beta -\nu}|^{p-2} \left[ \frac{\hat{\la}_{\nu}}{|\nu|^{p-2}} - \frac{\hat{\la}_{\beta}}{|\beta|^{p-2}} |x|^{\beta - \nu} -  \al(\nu -\beta |x|^{\beta - \nu}) O(\delta^{-1}_{\Om}) \right]  \left[ 1 -  O(\delta_{\Om}^{-1})\right] \nonumber \\ \geq \hat{\la}_{\nu} (|x|^{\nu}-|x|^{\beta})^{p-1} 
	\end{align*}
	in a neighborhood  of infinity.  The latter inequality can be written
	in the form
	\begin{align} \label{Req:ineq_sup_inf}
		|\nu - \beta |x|^{\beta -\nu}|^{p-2} \left[ \frac{\hat{\la}_{\nu}}{|\nu|^{p-2}} - \frac{\hat{\la}_{\beta}}{|\beta|^{p-2}} |x|^{\beta - \nu} - \al(\nu - \beta |x|^{\beta -\nu}) O(\delta^{-1}_{\Om}) \right]  \left[ 1 - O(\delta_{\Om}^{-1})\right] \nonumber  \\
		\geq \hat{\la}_{\nu}  (1-|x|^{\beta - \nu})^{p-1} .
	\end{align}
	By the assumptions on $\nu, \beta$, we have $|x|^{\beta -\nu} \ra 0$ as $|x| \ra \infty$.  Thus,  by expanding both sides of \eqref{Req:ineq_sup_inf} in $|x|^{\beta -\nu}$ to the first order, we obtain
	\begin{align*} 
		[\hat{\la}_{\nu} - A |x|^{\beta - \nu} + o(|x|^{\beta - \nu}) - O(\delta^{-1}_{\Om})] \left[  1 -  O(\delta_{\Om}^{-1})\right]   \geq  \hat{\la}_{\nu}  (1 - (p-1)|x|^{\beta -\nu}+ o(|x|^{\beta -\nu})) \,,
	\end{align*}
	where $A$ is the same constant defined above. Using similar arguments as above, it follows that $A< \hat{\la}_{\nu}(p-1)$, and hence, one obtains the desired assertion.
\end{proof}
\section{Bounded domains}\label{bounded_dom}
We start with the first main result of this section.
\begin{theorem} \label{Thm:existence_bdd}
Let $\Om \subset \R^N$ be a bounded domain with a compact $C^{1,\gg}$-boundary. Assume that the spectral gap $\Gamma_{\alpha,p}(\Omega):= \lambda_{\alpha,p}^{\infty}(\Omega)-H_{\alpha,p}(\Omega)$ is strictly positive.  Then, \eqref{Lp_Hardy_const} admits a unique (up to a multiplicative constant) positive minimizer $u \in \widetilde{W}^{1,p;\alpha}_0(\Omega)$, and  $u \asymp  \delta_{\Om}^{\nu}$  in a relative neighborhood  of $\partial \Om$, where $\nu=\nu_{\al,p}(H_{\al,p}(\Om)) $. 
\end{theorem}
\begin{proof} 
By our assumption and Corollary \ref{Cor:lam_inf_est}, $H_{\al,p}(\Om)<c_{\al,p,1}$. 
According to Corollary \ref{Non-exist_Half_space2}, one necessarily has $\al+p \neq 1$. Lemma \ref{lem-spectral-gap} implies that the operator $-\De_{\al,p}-\frac{H_{\al,p}(\Om)}{\delta_{\Om}^{\al+p}}\mathcal{I}_p$ is critical in $\Gw$, and therefore, it admits  
a ground state $u \in W^{1,p}_{\loc}(\Om)$. In particular, $u$ is a minimal positive solution in $\Gw$. Moreover, $u$ is a unique (up to a multiplicative constant) positive (super)solution of the equation $\big(-\De_{\al,p}-\frac{H_{\al,p}(\Om)}{\delta_{\Om}^{\al+p}} \mathcal{I}_p\big) w=0$ in $\Gw$. We prove the following claims, which ultimately imply the result.

\medskip

\noindent{\bf Claim:} $(i)$ $u \in L^p(\Om;\delta_{\Om}^{-(\al+p)})$, $(ii)$ $u \in \widetilde{W}^{1,p;\alpha}_0(\Omega)$, and $(iii)$ $u \asymp \delta_{\Om}^{\nu}$ in a relative neighborhood  of $\partial \Om$, where $\nu=\nu_{\al,p}(H_{\al,p}(\Om)) $.

\medskip

For the proof of the claim, we let $\beta \in (\frac{\al+p-1}{p},\frac{\al+p-1}{p-1}]$ if $\al+p>1$ or $\beta \in (\frac{\al+p-1}{p},0]$ if $\al+p<1$, be such that $\nu<\beta < \nu + \tilde{\gamma}$, where $\nu=\nu_{\al,p}(H_{\al,p}(\Om)) $, and $\tilde{\gamma}$ is such that the Green function $\Psi$ is $C^{1,\tilde{\gamma}}$. Note that by definition of $\nu$, one has $\la_{\nu}= H_{\al,p}(\Om)$ and moreover, the spectral gap assumption implies that $\frac{\al+p-1}{p}<\nu$.

\medskip

$(i)$ By Lemma \ref{lem:sub_sup_per_Agmon2}, it follows that $\Psi^{\nu}-\Psi^{\beta}$ is a supersolution of $$-\De_{\al,p}(\varphi) = \frac{\la_{\nu}}{\delta_{\Om}^{\al+p}} \mathcal{I}_p(\varphi)  = \frac{H_{\al,p}(\Om)}{\delta_{\Om}^{\al+p}} \mathcal{I}_p(\varphi) $$
in $\mathcal{N}_{\partial \Om}$, where $\mathcal{N}_{\partial \Om}\subset \Gw$ is a suitable neighborhood   of $\partial \Om$. Since $u$ is a positive solution of minimal growth
in a neighborhood  of infinity in $\Om$, it follows that $u \leq C (\Psi^{\nu}-\Psi^{\beta}) \leq C\Psi^{\nu}$ in a neighborhood   of $\partial \Om$. Since $\Psi$ is
asymptotic to $\delta_{\Om}$ as $x \ra \partial \Om$ (by Hopf's lemma, see \cite{Lamberti}), we infer that
\begin{align} \label{one_sided_asymp_est_bdd}
u \leq C \delta_{\Om}^{\nu}
\end{align}
for some $C>0$ in a suitable neighborhood   of $\partial \Om$. As $\Om$ is bounded and $\nu > \frac{\al+p-1}{p}$, it follows from Remark \ref{Rmk:integrability} that $u \in L^p(\Om;\delta_{\Om}^{-(\al+p)})$. 
 
$(ii)$ In view of Lemma \ref{lem-smaller_null-seq}, there exists a null-sequence $(u_n)$ in $\widetilde{W}^{1,p;\alpha}_0(\Omega)$  with $u_n \leq u$ such that $u_n \ra u$ in $L^p_{\loc}(\Om)$. Hence,    
\begin{align*}
\|u_n\|_{\widetilde{W}^{1,p;\alpha}_0(\Omega)}^p & =  \mathcal{Q}_{\ga,p,-H_{\al,p}(\Om) \gd_\Gw^{-(\ga+p)}}(u_n) + (H_{\al,p}(\Om)+1) \|u_n\|_{L^p(\Om,\delta_{\Om}^{-(\al+p)})}^p \\ 
&\leq \mathcal{Q}_{\ga,p,-H_{\al,p}(\Om) \gd_\Gw^{-(\ga+p)}}(u_n) + (H_{\al,p}(\Om)+1) \|u\|_{L^p(\Om,\delta_{\Om}^{-(\al+p)})}^p \,.
\end{align*}
Since $\lim_{n \ra \infty} \mathcal{Q}_{\ga,p,-H_{\al,p}(\Om) \gd_\Gw^{-(\ga+p)}}(u_n)=0$, we infer that $(u_n)$ is bounded in  $\widetilde{W}^{1,p;\alpha}_0(\Omega)$. Due to the reflexivity of $\widetilde{W}^{1,p;\alpha}_0(\Omega)$, there exists $v \in \widetilde{W}^{1,p;\alpha}_0(\Omega)$ such that, up to a subsequence, $u_n \wra v$ in $\widetilde{W}^{1,p;\alpha}_0(\Omega)$. Consequently, up to a subsequence, $u_n \ra v$ a.e. in $\Om$. However, since $u_n \ra u$ in $L^p_{\loc}(\Om)$, it follows that $u=v \in \widetilde{W}^{1,p;\alpha}_0(\Omega)$. This proves $(ii)$.

$(iii)$ Notice that, in view of \eqref{one_sided_asymp_est_bdd}, it remains to prove that there exists $\tilde{C}>0$ such that 
\begin{align} \label{other_sided_asymp_est_bdd}
u \geq \tilde{C} \delta_{\Om}^{\nu}
\end{align} in a relative neighborhood  of $\partial \Om$. By Lemma \ref{lem:sub_sup_per_Agmon2}, it follows that $\Psi^{\nu}+\Psi^{\beta}$ is a subsolution of $$-\De_{\al,p}(\varphi) = \frac{\la_{\nu}}{\delta_{\Om}^{\al+p}} \mathcal{I}_p(\varphi)  = \frac{H_{\al,p}(\Om)}{\delta_{\Om}^{\al+p}} \mathcal{I}_p(\varphi) $$
in a suitable neighborhood  $\Om_a\subset \Gw$  of $\partial \Om$ for some $a>0$. Certainly, $u$ is a positive supersolution of the same equation, and let $\tilde{C}>0$ sufficiently small such that $\tilde{C}(\Psi^{\nu}+\Psi^{\beta}) \leq u$ on $\Sigma_a$.  Now we verify that \eqref{extra_cond} is satisfied with $u_1=\tilde{C}(\Psi^{\nu}+\Psi^{\beta})$ and $u_2=u$. Since $\frac{\al+p-1}{p}<\nu<\beta < \nu + \tilde{\gamma}$ and $\Psi$ is asymptotic to $\delta_{\Om}$ as $ \delta_{\Om}(x) \ra 0$ (by Hopf's lemma, see \cite{Lamberti}), Lemma \ref{lem:local_est}-$(ii)$ implies
\begin{align} \label{suppli_1_existence}
   \liminf_{r\ra 0}  \frac{1}{r}\int_{D_r} (\Psi^{\nu}+\Psi^{\beta})^p \left( \frac{|\nabla u|}{u}\right)^{p-1}  \delta_{\Om}^{-\al} \dx =0 .
\end{align}
Again, using the fact that $\Psi$ is asymptotic to $\delta_{\Om}$ as $ \delta_{\Om}(x) \ra 0$ and \eqref{eq_nablaG}, it follows that 
	$$(\Psi^{\nu}+\Psi^{\beta}) |\nabla (\Psi^{\nu}+\Psi^{\beta})|^{p-1} \leq  C \delta_{\Om}^{\nu+(\nu -1)(p-1)}  \quad \mbox{ near } \partial \Om \mbox{ for some } C>0.$$ 
	Consequently, as $\frac{\al+p-1}{p}<\nu<\beta < \nu + \tilde{\gamma}$, we get
\begin{align} \label{suppli_2_existence}
    \liminf_{r\ra 0} \frac{1}{r}\int_{D_r} (\Psi^{\nu}+\Psi^{\beta})^p \left(\frac{\left|\nabla (\Psi^{\nu}+\Psi^{\beta}) \right|}{\Psi^{\nu}+\Psi^{\beta}}\right)^{p-1}  \delta_{\Om}^{-\al} \dx =0 .
\end{align}
Observe that \eqref{suppli_1_existence} and \eqref{suppli_2_existence} yield \eqref{extra_cond} with $u_1=\tilde{C}(\Psi^{\nu}+\Psi^{\beta})$ and $u_2=u$. Hence, the WCP (Lemma \ref{lem:comparison}) implies that $\tilde{C}(\Psi^{\nu}+\Psi^{\beta}) \leq u$ on $\Om_a$. This yields \eqref{other_sided_asymp_est_bdd}. Therefore \eqref{one_sided_asymp_est_bdd}, together with \eqref{other_sided_asymp_est_bdd}, imply that $u \asymp \delta_{\Om}^{\nu}$.
\end{proof}
\begin{remark} \label{Rmk_trivial1} \rm
Suppose that $\al+p<1$, and $\Gw$ is a $C^{1,\gamma}$-bounded domain. Then, by  Theorem \ref{Thm:nonexistence_bdd}, we have $H_{\al,p}(\Om)=0$. Consequently, we always have a spectral gap since $\Gamma_{\al,p}(\Om)=\la_{\al,p}^{\infty}(\Om)-H_{\al,p}(\Om)=\la_{\al,p}^{\infty}(\Om)=c_{\al,p,1}>0$. Hence, by the above theorem, the ground state $u=1 \in \widetilde{W}^{1,p;\alpha}_0(\Omega)$ is a minimizer. Therefore, $-\De_{\al,p}$ is positive-critical in $\Om$ with respect to the weight $\delta_{\Om}^{-1}$. This gives an alternate proof of Proposition \ref{Prop:to_be_added}-$(i)$.
\end{remark}
We now show that the converse of Theorem \ref{Thm:existence_bdd} holds.
\begin{theorem} \label{Thm:Bdd_necess}
Let $\Om \subset \R^N$ be a $C^{1,\gg}$-bounded domain.
 Assume that there exists a minimizer  $u \in \widetilde{W}^{1,p;\alpha}_0(\Omega)$  of \eqref{Lp_Hardy_const}. Then there is a spectral gap $H_{\al,p}(\Om)<\la_{\al,p}^{\infty}(\Om)= c_{\al,p,1}$.
\end{theorem}
\begin{proof}
If $\al+p=1$, then by Theorem \ref{Thm:nonexistence_bdd}, it follows that   $-\De_{1-p,p}$ is null-critical in $\Om$ with respect to the weight $\delta_{\Om}^{-1}$. Hence, there does not exist any minimizer in this case. Thus, the theorem's hypothesis necessarily implies that $\al+p \neq 1$.
Let $u \in \widetilde{W}^{1,p;\alpha}_0(\Omega)$ be a positive minimizer for \eqref{Lp_Hardy_const} i.e.,
$\big(-\De_{\al,p} - \frac{H_{\al,p}(\Om) }{\delta_{\Om}^{\al+p}} \mathcal{I}_p\big)u=0$ in $\Om$. By Theorem \ref{Thm:existence_bdd} we have $H_{\al,p}(\Om) \leq c_{\al,p,1}$. Assume by contradiction  that $H_{\al,p}(\Om) = c_{\al,p,1}.$ Let $\nu \in (\frac{\al+p-1}{p},\frac{\al+p-1}{p-1}]$ if $\al+p>1$ and $\nu \in (\frac{\al+p-1}{p},0]$ if $\al+p<1$   be such that $\frac{\al+p-1}{p} < \nu<\beta < \nu + \tilde{\gamma}$. It follows that $\la_{\nu} <  H_{\al,p}(\Om)$. Further, by Lemma~\ref{lem:sub_sup_per_Agmon2},  $\Psi^{\nu}+\Psi^{\beta}$ is a subsolution of $$\left(-\De_{\al,p}- \frac{\la_{\nu}}{\delta_{\Om}^{\al+p}} \mathcal{I}_p\right)w=0 \qquad \mbox{in }  \Om_a \subset 
 \Om$$
 for some $a>0$ uniformly with respect to $\nu$, and $u$ is a positive supersolution of this equation. Let $\tilde{C}>0$  be  a  constant (which can be taken to be independent of $\nu$) such that $\tilde{C}(\Psi^{\nu}+\Psi^{\beta}) \!\leq\! u$ on $\Sigma_a$.   As in the proof of Theorem~\ref{Thm:existence_bdd},  the WCP (Lemma \ref{lem:comparison}) implies that  $\tilde{C}\delta_{\Om}^{\nu} \leq u$. Thus, by taking $\nu \ra \frac{\al+p-1}{p}$, we get $\tilde{C}\delta_{\Om}^{\frac{\al+p-1}{p}} \leq u$. This implies that $u \notin L^{p}(\Om;\delta_{\Om}^{-(\al+p)})$, and hence $u \notin \widetilde{W}^{1,p;\alpha}_0(\Omega)$, which is a contradiction. Therefore, $H_{\al,p}(\Om)<c_{\al,p,1}$.
\end{proof}
We conclude the present section with a proof of Theorem \ref{thm:main1}  for bounded domains along with some consequences.

\begin{proof}[Proof of Theorem \ref{thm:main1} for $C^{1,\gamma}$-bounded domains]
The proof is a direct consequence of theorems \ref{Thm:existence_bdd} and \ref{Thm:Bdd_necess}.
\end{proof}

As mentioned in the introduction, it has been proved by Avkhadiev that if $\Omega \subsetneq \R^N$ is convex and $\alpha+p>1$, then $H_{\alpha,p}(\Omega)=c_{\alpha,p,1}$. Moreover, it is well-known that if there exists at least one tangent hyperplane for $\partial\Omega$ (and in particular, if $\Omega$ has $C^1$ boundary) then $\lambda^\infty_{\alpha,p}(\Omega)\leq c_{\alpha,p,1}$. Since the inequality $H_{\alpha,p}(\Omega)\leq \lambda^\infty_{\alpha,p}(\Omega)$ always holds, it follows that if $\al+p>1$, then $H_{\alpha,p}(\Omega)=\lambda^\infty_{\alpha,p}(\Omega)=c_{\alpha,p,1}$ for convex $C^1$-domains, i.e., there is no spectral gap. One can in fact generalize a little bit this argument, and show the following result.
\begin{proposition} \label{Prop:6.4} 
Let $\Omega \subsetneq\R^N$ be a $C^{1,\gg}$-domain. Moreover, assume either that $\alpha+p\geq 1$ and $\Omega$ is bounded with mean convex boundary, or that $\alpha+p\leq 1$ and $\Omega$ is an exterior domain with mean concave boundary. Then, there is no spectral gap, i.e. $H_{\alpha,p}(\Omega)=\lambda_{\alpha,p}^\infty(\Omega)$.
\end{proposition}
\begin{proof}
Assume first that $\alpha+p\geq 1$ and $\Omega$ is a bounded mean convex domain.  By Corollary~\ref{Cor:lam_inf_est}
$$\lambda_{\alpha,p}^\infty(\Omega)=c_{\alpha,p,1}.$$
Since there always holds that $H_{\alpha,p}(\Omega)\leq \lambda^\infty_{\alpha,p}(\Omega)$, it is enough to prove that $H_{\alpha,p}(\Omega)\geq c_{\alpha,p,1}$. For this, according to the AAP-type theorem (Theorem \ref{aap_thm}), it is enough to find a positive supersolution to the equation
\begin{equation}\label{eq:lam}
(-\De_{\al,p}-c_{\alpha,p,1} \gd^{-(\ga+p)}_{\Om}\mathcal{I}_p)w=0 \qquad \mbox{ in } \Gw \,.
\end{equation}
Note that since $\Omega$ is mean-convex, one has $-\Delta\delta_\Omega\geq 0$ in the sense of distributions (see \cite[Theorem 3.4]{Psaradakis}). Also, note that $\delta_\Omega \in C(\Omega)\cap W^{1,\infty}(\Omega)$ and that $|\nabla \delta_\Omega|=1$ a.e. Then, according to Lemma \ref{weak_lapl}, one computes that in the sense of distributions,
$$-\Delta_{\alpha,p}(\delta_\Omega^\nu)=|\nu|^{p-2}\nu\delta_\Omega^{(\nu-1)(p-1)}[(p-1)(1-\nu)\delta_\Omega^{-\alpha-1}-\Delta_{\alpha,p}(\delta_\Omega)].$$
Moreover,
$$\Delta_{\alpha,p}(\delta_\Omega)=-\alpha\delta_\Omega^{-\alpha-1}+\delta_\Omega^{-\alpha}\Delta\delta_\Omega\leq -\alpha\delta_\Omega^{-\alpha-1}.$$
Hence, as $\nu=\nu_{\alpha,p}(c_{\alpha,p,1})=(\alpha+p-1)/p \geq 0$, we obtain
\begin{eqnarray*}
-\Delta_{\alpha,p}(\delta_\Omega^\nu) &\geq & |\nu|^{p-2}\nu[(1-\nu)(p-1)+\alpha]\delta_\Omega^{-\alpha-p}\mathcal{I}_p(\delta_\Omega^\nu) \\
&= & c_{\alpha,p,1}  \mathcal{I}_p(\delta_\Omega^\nu).
\end{eqnarray*}
Therefore, $\delta_\Omega^\nu$ is a positive supersolution of \eqref{eq:lam}, and this concludes the proof in this case. The proof for exterior domains is similar and will be skipped.
\end{proof}
\begin{remark}[Bounded domains with positive spectral gap] \label{Ex:bdd} \rm
Here we provide some examples of bounded domains where the spectral gap condition holds. Let $\al=0,$ and $p=N=2$.
From \cite[Theorem 4.1]{Davies} and \cite{Barbatis} respectively, it follows that certain sectorial regions and non-convex quadrilateral in two dimension have the Hardy constant $H_{0,2}(\Om)<1/4$. 
Then, using the same smooth approximation arguments as in \cite[Example 6]{MMP}, we can have a family of smooth domains $(\Om_n)_{n \in \N}$ such that $H_{0,2}(\Om_n) \leq H_{0,2}(\Om)<1/4 $. Since $(\Om_n)$ are smooth, $\la_{0,2}^{\infty}(\Om_n)=1/4$ for all $n$. Hence, $(\Om_n)$  satisfy the spectral gap condition. Also, there are annular domains that admit a positive spectral gap, see \cite[Example 4]{MMP}. Moreover, we have seen in Remark \ref{Rmk_trivial1} that any smooth bounded domain satisfies the spectral gap condition when $\al+p<1$.
\end{remark}
\section{Unbounded domains}
 In this section, we prove Theorem~\ref{thm:main1} for the case of exterior domains. We start with the following result, which allows us to compute $\lambda_{\alpha,p}^\infty(\Omega)$ for all values of $\alpha$ and $p$.

\begin{theorem} \label{Thm:la_inf_est2}
Let $\Om$ be a $C^{1,\gamma}$-exterior domain with a compact boundary. Then $\la_{\al,p}^{\infty}(\Om) =  c_{\al,p}=\min\{c_{\al,p,1}, c_{\al,p,N}\}$. In particular,  if  $\al+p \in \{1,N\}$, then $H_{\al,p}(\Om)=\la_{\al,p}^{\infty}(\Om)=0$.
\end{theorem}
\begin{proof}
First assume that $\al+p=1$. Then, by Corollary \ref{Non-exist_Half_space2}, we have $\la_{\al,p}^{\infty}(\Om)=c_{\al,p,1}=c_{\al,p}=0$. Next, we consider $\al+p \neq 1$.
We write $\R^{N}_{*}=\R^{N}\setminus \{0\}$. It is well known that  $H_{\al,p}(\R^N_{*}) = c_{\al,p,N}$, see \cite{Avkhadiev} for instance. Notice that for every $\varphi\in C_c^\infty(\R^N_*)$ and every compact set $K\subset \R^N$, there exists a constant $c>0$ such that the function $\psi$ defined by
$$\psi(x)=\varphi(cx)$$
has support in $\R^N_*\setminus K$. Hence, by using the scale invariance of the Rayleigh quotient defining the Hardy constant, one sees that
$$\lambda^\infty_{\alpha,p}(\R^N_*)=H_{\alpha,p}(\R^N_*)=c_{\alpha,p,N}.$$
Now, we write $\partial \Omega=\Gamma_1\cup\cdots \cup\Gamma_\ell$, where by assumption $\Gamma_i\cap \Gamma_j=\emptyset$ for $i\neq j$, and $\Gamma_i$ is compact, connected, and $C^{1,\gamma}$. There exists a compact set $K\subset \Omega$ such that 
$$\Omega\setminus K=\mathcal{N}_1\cup\cdots\cup \mathcal{N}_{\ell+1},$$
where $\mathcal{N}_i\cap\mathcal{N}_j=\emptyset$ for $i\neq j$, the set $\mathcal{N}_i$ is a relatively compact neighborhood  of $\Gamma_i$ for $1\leq i\leq \ell$, and $\mathcal{N}_{\ell+1}$ is the exterior of a large ball centered at the origin in $\R^N$. For $1\leq j\leq \ell+1$, define $\lambda^{\infty,j}_{\alpha,p}$, the Hardy constant of $\Om$ at infinity around $\Gamma_j$ for $1\leq j\leq \ell$ (resp., the best Hardy constant of $\Om$ when $|x|\to+\infty$ for $j=\ell+1$) by
\begin{align*}
 \la_{\al,p}^{\infty,j}(\Om) :=\sup_{U\Subset \mathcal{N}_j}\, 
 \inf_{W^{1,p}_{c}(U)} \left\{\int_{U} |\nabla \varphi|^p \delta_{\Om}^{-\al} \dx  \biggm| \int_{U} |\varphi|^p \delta_{\Om}^{-(\al+p)} \dx=1 \right\}.
 \end{align*}
Then, one sees easily from the support consideration that
$$\lambda_{\alpha,p}^\infty(\Om)=\min_{1\leq j\leq \ell+1}\la_{\al,p}^{\infty,j}(\Om).$$
However, by the scale invariance of the Hardy constant,
$$\lambda^{\infty,\ell+1}_{\alpha,p}(\Om)=\lambda_{\alpha,p}^\infty(\R^N_*)=c_{\alpha,p,N}.$$
Moreover, the proof of Corollary~\ref{Cor:lam_inf_est} implies that for every $1\leq j\leq \ell$,
$$\lambda^{\infty,j}_{\alpha,p}(\Om)=c_{\alpha,p,1}.$$
Thus, we obtain that
$\lambda^\infty_{\alpha,p}(\Om)=\min\{c_{\alpha,p,1},c_{\alpha,p,N}\}=c_{\alpha,p}. $
\end{proof}
We are now ready to prove the analogue of Theorem \ref{Thm:existence_bdd} for the $C^{1,\gamma}$-exterior domain.
\begin{theorem} \label{Thm:existence_exterior}
Let $\Om \subsetneq \R^N$ be a $C^{1,\gamma}$-exterior domain. Assume that $\Gamma_{\alpha,p}(\Om)>0$, i.e., there is a spectral gap. Then, there exists a unique (up to a multiplicative constant) positive minimizer $u \in \widetilde{W}^{1,p;\alpha}_0(\Omega)$ for \eqref{Lp_Hardy_const}.  Furthermore,  $u \asymp \delta_{\Om}^{\nu}$  in a relative neighborhood  of $\partial \Om$, where $\nu = \nu_{\al,p}(H_{\al,p}(\Om))$, 
and $u \asymp |x|^{\tilde \nu}$  near infinity, where $\tilde \nu = \nu_{\al,p,N}(H_{\al,p}(\Om))$.
\end{theorem}
\begin{proof}
By Theorem \ref{Thm:la_inf_est2}, we have $\la_{\al,p}^{\infty}(\Om) = c_{\al,p}$, so the assumption that there is a spectral gap is equivalent to $H_{\alpha,p}(\Om)<c_{\alpha,p}$. In particular, since $c_{\alpha,p}=0$ for $\al+p \in \{1,N \}$, the assumptions necessarily imply that $\al+p\notin \{1,N\}$.  Also,  if $\al+p>N$, then by \cite[Theorem 1.1]{GPP} $H_{\al,p}(\Om) \geq c_{\al,p,N}=c_{\al,p}$   and therefore in this case, 
$H_{\al,p}(\Om) = c_{\al,p}=\la_{\al,p}^{\infty}(\Om)$. Thus, the assumptions of the theorem necessarily imply that $1 \neq \al+p<N$.

According to Lemma~\ref{lem-spectral-gap}, the operator $-\De_{\al,p}-\frac{H_{\al,p}(\Om)}{\delta_{\Om}^{\al+p}} \mathcal{I}_p$ admits a ground state $u \in W^{1,p}_{\loc}(\Om)$ in $\Gw$. In particular, $u$ is the unique (up to a multiplicative constant) positive (super)solution of the equation 
$\big(-\De_{\al,p}-\frac{H_{\al,p}(\Om)}{\delta_{\Om}^{\al+p}} \mathcal{I}_p\big)w=0$ in $\Gw$. Thus, it remains to show that $(i)$ $u \in L^p(\Om;\delta_{\Om}^{-(\al+p)})$, $(ii)$ $u \in \widetilde{W}^{1,p;\alpha}_0(\Omega)$, and $(iii)$ that $u$ satisfies the required asymptotics near $\partial \Gw$ and infinity. 

$(i)$ Repeating the arguments of Theorem \ref{Thm:existence_bdd}, it follows that 
\begin{align} \label{ineq:Unbdd_est_bdary}
u\asymp \delta_{\Om}^{\nu}
\end{align}
 in a relative neighborhood of $\partial \Om$, where  $\nu=\nu_{\al,p}(H_{\al,p}(\Om)) \in (\frac{\al+p-1}{p},\frac{\al+p-1}{p-1}]$ if $\al+p>1$, and $\nu=\nu_{\al,p}(H_{\al,p}(\Om)) \in (\frac{\al+p-1}{p},0]$ if $\al+p<1$. On the other hand, since $H_{\al,p}(\Om)<c_{\al,p,N}$, it follows from Lemma \ref{lem:agmon_subsup_inf} that $U_-=|x|^{\tilde \nu}-|x|^{\beta}$ is a supersolution of $\left[-\De_{\al,p}-\frac{\hat{\la}_{\tilde \nu} }{\delta_{\Om}^{\al+p}} \mathcal{I}_p\right]w=0$ near infinity, where $\tilde{\nu} \in [\frac{\al+p-N}{p-1},\frac{\al+p-N}{p})$  is such that $\beta<\tilde \nu<\beta+1$ and 
 \begin{equation}\label{eq-tilde-nu}
 \hat{\la}_{\tilde  \nu}=|\tilde \nu|^{p-2}\tilde \nu [(\al-N+1)+(1-\tilde \nu)(p-1)] =H_{\al,p}(\Om) \,,
 \end{equation}
so $\tilde \nu=\nu_{\al,p,N}(H_{\al,p}(\Om))$.
Being a positive solution of minimal growth at infinity, $u$ satisfies 
\begin{align} \label{ineq:Unbdd_est_inf}
u(x) \leq C U_-(x) \leq C|x|^{\tilde \nu} \qquad  \mbox{in} \ B_R^c
\end{align}
 for some $C>0$ and $R$ large enough. Therefore, combining \eqref{ineq:Unbdd_est_bdary} (as $\nu > \frac{\al+p-1}{p}$) and \eqref{ineq:Unbdd_est_inf} (as $\tilde \nu < \frac{\al+p-N}{p}$), we infer that $u \in L^p(\Om;\delta_{\Om}^{-(\al+p)})$. Now, part  $(ii)$ follows  
from the same arguments as in the proof of Theorem \ref{Thm:existence_bdd}.

$(iii)$ Recall that the required asymptotic \eqref{ineq:Unbdd_est_bdary} of $u$ near the boundary was demonstrated  above.    In light of \eqref{ineq:Unbdd_est_inf} , it remains to show that there exists $\tilde{C}>0$ such that 
\begin{align} \label{ineq:Unbdd_otherest_inf}
u(x) \geq \tilde{C} |x|^{\tilde \nu}
\end{align}
near infinity, where $\tilde{\nu}=\nu_{\al,p,N}(H_{\al,p}(\Om)) \in [\frac{\al+p-N}{p-1},\frac{\al+p-N}{p})$ satisfies  \eqref{eq-tilde-nu}. 
Let  
$$ \beta<\tilde \nu<  \min \left\{\beta+1,\frac{\al+p-N}{p} \right\}<0,$$ 
 and note that this choice of $\beta, \tilde{\nu}$ is possible as $H_{\al,p}(\Om)<c_{\al,p,N}$. It follows from Lemma \ref{lem:agmon_subsup_inf} that $U_+=|x|^{\tilde \nu}+|x|^{\beta}$ is a subsolution of $\left[-\De_{\al,p}-\frac{H_{\al,p}(\Om)  }{\delta_{\Om}^{\al+p}} \mathcal{I}_p \right]w=0$ in a suitable neighborhood  $\Om^R$ of infinity for some $R>0$ large enough. Clearly, $u$ is a positive supersolution of the same equation. Let $\tilde{C}>0$ be sufficiently small such that $\tilde{C} U_+ \leq u$ on $\Sigma_R$. 

In order to apply the appropriate WCP (Lemma~\ref{lem:comparison1}), we need to verify that the growth condition \eqref{extra_cond_inf} is satisfied with $u_1=\tilde{C}U_+$ and $u_2=u$. Due to the choice of $\tilde{\nu}$ and $\beta$, and using Lemma~\ref{lem:local_est}-$(ii)$, we infer that
\begin{align} \label{suppli_1_existence_unbdd}
   \liminf_{R\ra \infty}  \frac{1}{R}\int_{D^R} U_+^p \left( \frac{|\nabla u|}{u}\right)^{p-1}  \delta_{\Om}^{-\al} \dx =0 .
\end{align}
Notice that $U_+ |\nabla U_+|^{p-1} \leq  C |x|^{\tilde{\nu}+(\tilde{\nu} -1)(p-1)} $ near infinity for some $C>0$. Consequently, by the choice of $\tilde{\nu}$ and $\beta$, we get
\begin{align} \label{suppli_2_existence_unbdd}
    \liminf_{R\ra \infty} \frac{1}{R}\int_{D^R} U_+^p \left(\frac{\left|\nabla U_+ \right|}{U_+}\right)^{p-1}  \delta_{\Om}^{-\al} \dx =0 .
\end{align}
Therefore,  \eqref{suppli_1_existence_unbdd} and \eqref{suppli_2_existence_unbdd} yield \eqref{extra_cond_inf} with $u_1=\tilde{C}U_+$ and $u_2=u$. Hence, the WCP (Lemma \ref{lem:comparison1}) implies that $\tilde{C}U_+ \leq u$ in $\Om^R$. This yields \eqref{ineq:Unbdd_otherest_inf}. Therefore, \eqref{ineq:Unbdd_est_inf} together with \eqref{ineq:Unbdd_otherest_inf} implies that $u \asymp |x|^{\tilde{\nu}}$ near infinity.
  \end{proof}
As a simple but quite surprising corollary of the above theorem, we obtain the following result when $\al+p < 1$. Recall that in this case, if $\Om$ is a $C^{1,\gamma}$-bounded domain, then $H_{\al,p}(\Om)=0$ (Theorem \ref{Thm:nonexistence_bdd}), i.e., the weighted Hardy inequality \eqref{Lp_Hardy} does not hold. However, \eqref{Lp_Hardy} does hold if $\Om$ is a $C^{1,\gamma}$-exterior domain and $\al+p < 1$.

\begin{corollary} \label{Cor7.3}
Let $1\neq \al+p < N$ and $\Om$ be a $C^{1,\gamma}$-exterior domain. Then $H_{\al,p}(\Om)>0$.
\end{corollary}
\begin{proof}
Let $\Om$ be a $C^{1,\gamma}$-exterior domain. We give a proof for $\al+p < 1$. The case $1<\al+p<N$, follows using similar arguments. By Theorem \ref{Thm:la_inf_est2} we have $\la_{\al,p}^{\infty}(\Om)\!=\!c_{\al,p}\!=\!c_{\al,p,1}$.  Suppose that  $H_{\al,p}(\Om)=0$, then the corresponding spectral gap $\Gamma_{\al,p}(\Om)=c_{\al,p,1}$ is positive. Hence, $-\Gd_{\ga,p}$ is critical in $\Om$ (by Lemma \ref{lem-spectral-gap}). Further, it follows from Theorem~\ref{Thm:existence_exterior} that there exists a minimizer $u \in \widetilde{W}^{1,p;\alpha}_0(\Omega)$ such that $u \asymp 1$ near $\partial \Om$ and $u \asymp |x|^{\frac{\al+p-N}{p-1}}$ near infinity. Notice that $w=1$ is also a positive solution of the equation $-\Gd_{\ga,p}w=0$. But, since $-\Gd_{\ga,p}$ is critical in $\Om$, there exists a unique (up to a multiplicative constant) positive supersolution in $W^{1,p}_{\loc}(\Om)$ to the equation $-\Gd_{\ga,p}w=0$ \cite[Theorem 4.15]{Yehuda_Georgios}. Thus, we arrived at a contradiction. Hence, $H_{\al,p}(\Om)>0$.       
\end{proof}

\begin{remark} \rm
As a consequence of the above corollary and the fact that $H_{\al,p}(\Om) \geq c_{\al,p,N}$ for $C^{1,\gamma}$-exterior domain if $\al+p>N$, it follows that the weighted Hardy inequality holds in $C^{1,\gamma}$-exterior domain for any $\al \in \R,p\in (1,\infty)$ with $\al+p \notin \{1,N\}$.

\end{remark}
  
\begin{theorem} \label{Thm:nonexistence}
Let $\Om \subsetneq \R^N$ be an exterior  domain with a compact $C^{1,\gg}$-boundary. Assume that  either $\al+p=1$ or $\al+p \geq N$.  Then \eqref{Lp_Hardy_const} does not admit any minimizer in 
$ \widetilde{W}^{1,p;\alpha}_0(\Omega)$.
\end{theorem}  
\begin{proof}
	{\bf{The case  $\al+p =1$} or $\al+p = N$:} 
	By Theorem~\ref{Thm:la_inf_est2}, $H_{\al,p}(\Om)=\la_{\al,p}^{\infty}(\Om)=0$.
	Assume that  $u \in \widetilde{W}^{1,p;\alpha}_0(\Omega)$ is a positive minimizer. In particular, $-\Gd_{\ga,p}$ is critical,  and consequently, the equation $-\Gd_{\ga,p}w=0$ admits a unique positive (super)solution in $\Gw$.  Therefore, $u=\mathrm{constant}>0$, 
	which is a contradiction as $1 \notin L^p(\Om_r, \delta_{\Om}^{-1}) \cup L^p(\Om_r, \delta_{\Om}^{-N})$. Hence, \eqref{Lp_Hardy_const} does not admit any minimizer in $ \widetilde{W}^{1,p;\alpha}_0(\Omega)$. 
	
 {\bf{The case $\al+p > N$}:} Suppose by contradiction that  $u \in \widetilde{W}^{1,p;\alpha}_0(\Omega)$ is a minimizer of \eqref{Lp_Hardy_const}. Since  $\al+p>N$, it follows that  $H_{\al,p}(\Om)=c_{\al,p,N}$ (see \cite[Theorem 1.1]{GPP}). 

Lemma \ref{lem:agmon_subsup_inf} implies that for every $\nu, \beta$ close to $\frac{\al+p-N}{p}$ with $\beta<\nu \leq \frac{\al+p-N}{p}$ there exists $R>0$ (independent of $\nu$) such that the function $U_+=|x|^{\nu}+|x|^{\beta}$ is a positive subsolution of $\left[-\De_{\al,p}-\frac{\hat{\la}_{\nu} }{\delta_{\Om}^{\al+p}} \mathcal{I}_p \right]w=0$ in $\Om^R$ for some $R>0$ sufficiently large, where $$\hat{\la}_{\nu}=|\nu|^{p-2}\nu [(\al-N+1)+(1-\nu)(p-1)].$$  Since $\hat{\la}_{\nu} \leq H_{\al,p}(\Om)$, it is clear that $u$ is a supersolution of the same equation above. Let $\tilde{C}$ be a positive constant (independent of $\nu$) such that $\tilde{C}U_+ \leq u$ on $\Sigma_R$. Using the WCP (Lemma \ref{lem:comparison1}) as in the proof of Theorem \ref{Thm:existence_exterior}, for $\beta<\nu < \frac{\al+p-N}{p}$, it follows that  $\tilde{C}|x|^{\nu} \leq u$ near infinity (uniformly in $\gn$).
Thus, by taking $\nu \ra \frac{\al+p-N}{p}$, we get $\tilde{C}\delta_{\Om}^{\frac{\al+p-N}{p}} \leq u$ near infinity. But, this implies that $u \notin L^{p}(\Om;\delta_{\Om}^{-(\al+p)})$, and hence $u \notin \widetilde{W}^{1,p;\alpha}_0(\Omega)$, which is a contradiction. Therefore,
\eqref{Lp_Hardy_const} does not admit a minimizer in $ \widetilde{W}^{1,p;\alpha}_0(\Omega)$.
\end{proof}
\begin{theorem} \label{Thm:necess_exterior}
Let $\Om \subset \R^N$ be an exterior  domain with a compact $C^{1,\gg}$-boundary, and assume that \eqref{Lp_Hardy_const} admits a minimizer in $ \widetilde{W}^{1,p;\alpha}_0(\Omega)$. Then $H_{\al,p}(\Om)<c_{\al,p}$.
\end{theorem}
\begin{proof}
In view of Theorem \ref{Thm:nonexistence} it is clear that we need to consider only the case 
$1\neq \al+p<N$.
Let $u \in \widetilde{W}^{1,p;\alpha}_0(\Omega)$ be a positive minimizer of \eqref{Lp_Hardy_const}. By Theorem \ref{Thm:la_inf_est2}, $H_{\al,p}(\Om) \leq c_{\al,p}$. Suppose that  $H_{\al,p}(\Om) = c_{\al,p}$. By repeating the arguments of Theorem \ref{Thm:Bdd_necess}, it follows that that $H_{\al,p}(\Om) <  c_{\al,p,1}$. Thus, $H_{\al,p}(\Om) = c_{\al,p,N}$. In this case, following the arguments of part $(iii)$ in the proof of Theorem \ref{Thm:existence_exterior} we obtain that there exist $\tilde{C}>0$ and $R>0$  such that    
\begin{align*} 
u(x) \geq \tilde{C} |x|^{\tilde \nu} \quad \ \mbox{in} \ \Gw^R \  \mbox{for all} \ \tilde \nu \ \mbox{s.t.} \ \frac{\al+p-N}{p-1} \leq\tilde \nu<\frac{\al+p-N}{p} \, . 
\end{align*}
Letting  $\tilde \nu\to (\al+p-N)/{p}$, we obtain that  
$u(x) \geq \tilde{C} |x|^{(\al+p-N)/{p}}$ in $\Gw^R$. Consequently, $u\not\in  \widetilde{W}^{1,p;\alpha}_0(\Omega)$ which is a a contradiction to our assumption. Hence, $H_{\al,p}(\Om)< c_{\al,p}$.
\end{proof}

We can now give the proof of Theorem \ref{thm:main1} for the case of exterior $C^{1,\gamma}$-domains.
\begin{proof}[Proof of Theorem \ref{thm:main1} for exterior domain with compact $C^{1,\gamma}$-boundary]
It is a direct consequence of Theorems \ref{Thm:existence_exterior} and \ref{Thm:necess_exterior}.
\end{proof}

\begin{remark} \label{Ex:unbdd} \rm
$(i)$ In the case of exterior domains, we observe that if $\alpha +p \geq N$, then there is no spectral gap (\cite[Theorem 4]{Avkhadiev_Makarov} and also \cite[Theorem 1.1]{GPP}), since in this case $c_{\al,p,N} \leq H_{\al,p}(\Om)\leq \la_{\al,p}^{\infty}(\Om)=c_{\al,p,N}$. However, it remains open whether there are exterior domains with a positive spectral gap when $1 <\alpha +p < N$.
 We consider $\al=0, p=2$.  For $N \geq 7$, Robinson \cite[Example 6.5]{Robinson} showed that there are exterior domains $\Om$ that have the {\it{Hardy constant near the compact boundary}} (see the precise definition in \cite{Robinson}) strictly less than $1/4$. Then it follows that $\la_{0,2}^{\infty}(\Om)<1/4$ for these domains, and consequently, $H_{0,2}(\Om)<1/4$. Now using the smooth approximation arguments as in \cite[Example 6]{MMP}, we can have family of smooth exterior domains $(\Om_n)$ satisfying $H_{0,2}(\Om_n)<1/4$. Since all the $\Om_n$ are smooth, $\la^{\infty}_{0,2}(\Om_n)=1/4$. Therefore, each $\Om_n$  admit a  positive spectral gap.

$(ii)$ At a first glance, it may seem to be not worthy to study the spectral gap phenomenon for exterior domains when $\alpha +p \geq N$, since in this case there is no spectral gap.  However, this is not the case as we give both necessary and sufficient condition for the existence of minimizers.  For instance, when $\alpha +p \geq N$, using our results we conclude that the weighted Hardy constant $H_{\al,p}(\Om)$ is not attained in the exterior domains.

$(iii)$ In light of the above remark one might think that it is not useful to carry out the decay analysis for the existence of minimizers in exterior domains when $\alpha +p \geq N$. However, the reason why this analysis is still relevant is justified by considering a perturbation of $-\Gd_{\ga,p}$ by a potential $W$ with compact support (see Remark \ref{Rmk:8.1}-$(ii)$).
\end{remark}
\section{Concluding Remarks}
We conclude with some remarks that summarize our main results concerning the spectral gap phenomenon for domains with compact $C^{1,\gamma}$-boundary, and with a number of open problems.

\begin{remark} \label{Rmk:8.1} \rm 
$(i)$ For domains with compact $C^{1,\gamma}$-boundary, we proved that $H_{\al,p}(\Om)$, the $L^{\al,p}$-Hardy constant,  is attained in $\widetilde{W}^{1,p;\alpha}_0(\Omega)$ if and only if there is a positive spectral gap ($\Gamma_{\al,p}(\Om)>0$). The weighted Hardy constant at infinity is obtained explicitly as $\la_{\al,p}^{\infty}(\Om)=c_{\al,p,1}$ for $C^{1,\gamma}$-bounded domains and $\la_{\al,p}^{\infty}(\Om)=c_{\al,p}$ for $C^{1,\gamma}$-exterior domains. In both the bounded and exterior domain cases, the sharp two sided decay estimates of a minimizer (whenever exists) are obtained.
 In the case of bounded domains, for $\al+p \leq 1$ we have $H_{\al,p}(\Om)=0$ and
$0<\la_{\al,p}^{\infty}(\Om) \ra 0=\la_{1-p,p}^{\infty}(\Om)$ as $\al+p \ra 1^{\pm}$.   It turns out that the gap phenomenon holds rather trivially for the case $\al+p \leq 1$, while it is more subtle for $\al+p > 1$ as it depends on the geometry of the boundary. Whereas, for exterior domains, we see that $H_{\al,p}(\Om)=0$ if and only if $\al+p$ is either $1$ or $N$, and $\la_{\al,p}^{\infty}(\Om) \ra 0$ as $\al+p \ra 1^{\pm}$ or $N^{\pm}$.   In exterior domains, the gap phenomenon holds trivially if $\al+p$ is either $1$ or $N$, while all the other cases need special attention.

$(ii)$ For $1-p<\alpha < 0$, we remark that the Green's function of $-\Delta_{\alpha,p}$ on a $C^{1,\gamma}$ bounded or exterior domain does not satisfy the Hopf's lemma. Because if it does, then one can follow the arguments of \cite[Lemma 3.4]{Lamberti} and obtain that $\lambda_{\alpha,p}^{\infty}(\Omega) \geq \left(\frac{p-1}{p} \right)^p$. But, we have seen in the previous remark that $\lambda_{\alpha,p}^{\infty}(\Omega) \leq c_{\alpha,p,1}$. This leads to a contradiction.

$(iii)$   One can verify that our analysis also goes through if we consider the following minimization problem for a given $W \in C_c(\Om)$: 
\begin{align*} 
 H_{\al,p,W}(\Om) :=\inf \left\{\mathcal{Q}_{\al,p,W}(\varphi) \biggm| \int_{\Om} |\varphi|^p \delta_{\Om}^{-(\al+p)} \dx\! =\!1, \varphi \in \widetilde{W}^{1,p;\alpha}_0(\Omega) \right\} ,
\end{align*}
where $\mathcal{Q}_{\al,p,W}(\varphi):=\int_{\Om} \left[|\nabla \varphi|^p \delta_{\Om}^{-\al} + W|\varphi|^p \right] \dx$. Such minimization problems are related to the following perturbation problem:
\begin{align*}
-\Delta_{\al,p} \varphi + W\mathcal{I}_p(\varphi) = \frac{\la}{\delta_{\Om}^{\al+p}} \mathcal{I}_p(\varphi) \quad \mbox{in} \ \  \widetilde{W}^{1,p;\alpha}_0(\Omega) \,.
\end{align*}
In this case, following the same arguments of the proof of Theorem \ref{thm:main1}, we can establish the spectral gap phenomenon in a domain with compact $C^{1,\gamma}$-boundary i.e., $H_{\al,p,W}(\Om)$ is attained in $\widetilde{W}^{1,p;\al}_0(\Om)$ if and only if $ H_{\al,p,W}(\Om)< \la_{\al,p}^{\infty}(\Om)$. Moreover, we can obtain the tight decay estimates for the corresponding minimizers which are actually solutions of the above PDE.

$(iv)$ The assumption that the domain has $C^{1,\gamma}$-regularity is crucial for our results.  We recall that the tubular neighbourhood theorem holds for $C^2$-domains but not for $C^{1.\gg}$-domains with $0 < \gg < 1$ (see the discussion of this point in \cite{Lamberti}). Moreover, the distance function $\delta_\Gw$ is not guaranteed to be differentiable near the boundary. Therefore, we replace the distance function with the minimal positive Green function of the $p$-Laplacian and use the Hopf-type lemma substantially in our analysis. It is important to mention that the minimal positive Green function of the $p$-Laplacian does not satisfy the Hopf-type lemma on $C^1$-domains, see \cite[Section 3.2]{GT}. One may also observe that the $C^{1,\gamma}$-regularity of the domain is essential for the estimates of the weighted Hardy constant at infinity which plays a significant role in the gap phenomenon.
\end{remark} 
Finally we propose three open problems:  
\begin{enumerate}
    \item  Does the gap phenomenon hold in domains with compact $C^1$-boundary? 
    \item For $1 < \alpha +p < N$ with $N<7$, are there $C^{1,\gamma}$-exterior domains with a positive spectral gap? (cf. Remark~\ref{Ex:unbdd}).
    \item  Let $\Gw$ be $C^{1,\gamma}$-exterior domain and $\ga+p\in\{1,N\}$. Is $-\Gd_{\ga,p}$ critical in $\Gw$?
\end{enumerate}
\appendix
\section{A chain rule}\label{app_1}
In this appendix we prove a chain rule for  $\mathcal{I}_p$. In the case $p\geq 2$, the mapping $t\mapsto |t|^{p-2}t$ is of class $C^1$, so the chain rule for $\mathcal{I}_p$ simply follows from the standard chain rule in $W^{1,p}$. In the case $p<2$, one needs to take care of the singularity at zero of $t\mapsto |t|^{p-2}t$. We prove the chain rule in the framework of a smooth manifold $M$.  For our purposes one should take $M=\Gw$. 
\begin{lemma}\label{lem_app}
	Let $v \in C^\infty(M)$, and $|v|^{p-2}|\nabla v| \in L^1_{\loc}(M)$; $p \in (1,\infty)$. Then, in the sense of distributions, we have $$\nabla \mathcal{I}_p(v)=(p-1)|v|^{p-2}\nabla v \,.$$ 
\end{lemma}
\begin{proof}
	Take $X \in C_c^{\infty}(TM)$ a vector field. Since $\mathcal{I}_p(v)$ is continuous, we have
	\begin{align} \label{eqn:main}
	-\int_{M} \mathcal{I}_p(v) {\text{div}}X \dx= -\lim_{\varepsilon \ra 0} \int_{\{|v|>\varepsilon\}\cup \{v=0\}} \mathcal{I}_p(v) {\text{div}} X\dx = -\lim_{\varepsilon \ra 0} \int_{\{|v|>\varepsilon\}} \mathcal{I}_p(v) {\text{div}} X \dx.
	\end{align}
	The last equality follows as $\mathcal{I}_p(0)=0$. According to the  coarea formula
	\begin{align*}
	\int_{\text{supp}(X)} |\nabla v| \dx= \int_{-\infty}^{\infty} \sigma(\{|v|=t\} \cap {\text{supp}}(X)) \dt  \,.
	\end{align*}
	Since $\nabla v \in L^1_{\loc}(M)$ (because it is continuous), the left hand side of the above equality is finite. Hence, one can choose a sequence $\varepsilon_n \ra 0$ such that the sequence $(\sigma (\{|v|=\varepsilon_n\}\cap {\text{supp}}(X)))_{n\in \N}$ is bounded. Moreover, by Sard's theorem (which can be applied since $v \in C^\infty(M)$), one can assume that the sequence $(\varepsilon_n )_{n\in \N}$ is chosen such that $\{|v| > \varepsilon_n\}$ are $C^1$-domains.  Then, by Green's formula
	\begin{align} \label{eqn:Green}
	-\int_{\{|v| >\varepsilon_n\}} \mathcal{I}_p(v) {\text{div}}X\dx = \int_{\{|v| > \varepsilon_n\}} \mathcal{I}_p'(v) \braket{\nabla v,X} \dx -  \int_{\{|v|=\varepsilon_n\}} \mathcal{I}_p(v) \braket{\nu,X} \d \sigma \,,
	\end{align}
	where $\nu$ is the exterior normal to $\{|v|=\varepsilon_n\}$. Now
	\begin{align} \label{eqn:bdary_term}
	\left|\int_{\{|v|=\varepsilon_n\} \cap {\text{supp}}(X)} \mathcal{I}_p(v) \braket{\nu,X} \d \sigma \right|\leq \sigma (\{|v|=\varepsilon_n\}\cap {\text{supp}}(X)) \mathcal{I}_p(\varepsilon_n) \|X\|_{L^{\infty}} \,.
	\end{align}
	Since $\mathcal{I}_p(\varepsilon_n) \ra 0$, we obtain
	$$\mathcal{I}_p(\varepsilon_n) \sigma (\{|v|=\varepsilon_n\}\cap {\text{supp}}(X) )  \ra 0 \, \text{ as } \ n \ra \infty \,.$$
	Consequently, from \eqref{eqn:bdary_term}, \eqref{eqn:Green} and \eqref{eqn:main} we get
	\begin{align}\label{A4}
	-\int_{M} \mathcal{I}_p(v) {\text{div}}X\dx = \lim_{\varepsilon_n \ra 0} \int_{\{|v|>\varepsilon_n\}} \mathcal{I}_p'(v) \braket{\nabla v, X} \dx.
	\end{align}
	Define 
	\begin{align*}
	\eta(x):= \begin{cases} \dfrac{\nabla v(x)}{|\nabla v(x)|} &   \nabla v(x) \neq 0, \\[2mm]
	\ 0   &\nabla v(x) =0.
	\end{cases}
	\end{align*}
	Then 
	\begin{align*}
	\int_{\{|v|>\varepsilon_n\}} \mathcal{I}_p'(v) \braket{\nabla v, X}\dx  = \int_{\{|v|>\varepsilon_n\}} \mathcal{I}_p'(v) \braket{\eta, X} |\nabla v| \dx  = \int_{\{|t|>\varepsilon_n \}} \mathcal{I}_p'(t) g(t) \dt   \,,
	\end{align*}
	where $g(t)= \int_{\{v=t\}} \braket{\eta, X} \d \sigma$. The last equality follows from the coarea formula. Let $T>0$ be such that $v|_{\text{supp}(X)} \in [-T,T]$. Then, in the above integral one can replace $ \int_{|t|>\varepsilon_n} $ by $ \int_{\{|t|>\varepsilon_n\} \cap [-T,T]} $. Now we claim that $\chi_{\{t \neq 0\}} \mathcal{I}_p'(\cdot) \ g (\cdot) \in L^1([-T,T])$. Indeed, again by the  coarea formula
	\begin{align*}
	\int_{\{t\neq 0\} \cap [-T,T]} |\mathcal{I}_p'(t)| |g(t)| \dt & \leq \int_{\{t \neq 0\} \cap v({\text{supp}}(X))} |\mathcal{I}_p'(t)| \left(\int_{\{v=t\}} \|X\|_{L^{\infty}} \d\sigma \right) \dt \\
	& = \|X\|_{L^{\infty}} \int_{\text{supp}(X) \cap \{v \neq 0\}} (p-1) |v|^{p-2} |\nabla v| \dx < \infty \,,
	\end{align*}
	since by  our assumption $|v|^{p-2}|\nabla v| \in L^1_{\loc}(M)$. This proves $\chi_{\{t \neq 0\}} \mathcal{I}_p'(\cdot) g (\cdot)\in L^1([-T,T])$. Subsequently, the dominated convergence theorem implies that
	\begin{align*}
	\lim_{\varepsilon_n \ra 0} \int_{\{|t|> \varepsilon_n\}} \mathcal{I}_p'(t) g(t) \dt = \int_{\{t \neq 0\}} \mathcal{I}_p'(t) g(t) \dt = \int_{\R} \mathcal{I}_p'(t) g(t) \dt,
	\end{align*}
	even if $\mathcal{I}_p'(0)=\infty$. Indeed, one has $\int_A (+\infty) \cdot \d\mu =0$ if $\mu(A)=0$. Then,  by \eqref{A4} and the  coarea formula, we get
	\begin{align*}
	-\int_{M} \mathcal{I}_p(v) {\text{div}}X \dx=\int_{\R} \mathcal{I}_p'(t) g(t) \dt = \int_{M} \mathcal{I}_p'(v) \braket{\nabla v,X} \dx.
	\end{align*}
	Hence, in the distribution sense, we have $\nabla \mathcal{I}_p(v)= \mathcal{I}_p'(v) \nabla v = (p-1) |v|^{p-2} \nabla v$. 
\end{proof} 

Suppose now that $F$ satisfies the assumptions of Lemma~\ref{weak_lapl}, and set  $v=F'(u)$, where  $u \in W^{1,p}_{\loc}(\Gw)\cap C(\Om)$. We need to prove that  in the sense of distributions, 
\begin{equation}\label{eq:distrib}
\nabla (\mathcal{I}_p (v))= \nabla (\mathcal{I}_p (F'(u))) = (p-1)  |F'(u)|^{p-2} F''(u) \nabla u.
\end{equation}
  Note that we have $\nabla v=F''(u) \nabla u \in L^p_{\loc}(\Gw)$. Thus, $v \in W^{1,p}_{\loc}(\Gw) \cap C(\Gw)$.

To prove the above assertion, take a smooth compactly supported vector field $X$ on $\Om$, and a sequence $(u_n) \in C^{\infty}(\Gw)$ such that $u_n \wra u$ in $W^{1,p}_{\loc} (\Gw)$ and $u_n \ra u$ uniformly on the support of $X$. Let $v_n =F'(u_n)$. Then, for every $n \in \N$, according to the previous lemma, we get
\begin{align} \label{eqn:gen}
-\int_{\Gw} \mathcal{I}_p(v_n) {\text{div}}X \dx= \int_{\Gw} (p-1) |v_n|^{p-2} \braket{\nabla v_n,X} \dx.
\end{align}
We have $v_n \ra v$ uniformly on $\text{supp}(X)$ and $\nabla v_n = F''(u_n) \nabla u_n \ra F''(u) \nabla u$ in $L^p({\text{supp}}(X))$. So, $\mathcal{I}_p(v_n)=v_n|v_n|^{p-2} \ra \mathcal{I}_p(v)$ uniformly on the support of $X$, and the left hand side of \eqref{eqn:gen} converges to $-\int_{\Gw} \mathcal{I}_p(v) {\text{div}}X \dx$. The right hand side of \eqref{eqn:gen} reads as 
$$(p-1) \int_{\Gw} |F'(u_n)|^{p-2}F''(u_n)\braket{\nabla u_n,X} \dx.$$
By the assumption on $F$, the function $|F'|^{p-2}F''$ is continuous on $(0,\infty)$, and since there exists $\ge>0$ such that $u(\text{supp}(X))\subset  (\vge , \vge ^{-1})$, one concludes that
$$|F'|^{p-2}(u_n)F''(u_n)\rightarrow |F'|^{p-2}(u)F''(u)$$
uniformly on the support of $X$.  Since $\nabla u_n \ra \nabla u$ in $L^p({\text{supp}(X)})$, we conclude by the dominated convergence theorem that the right hand side of \eqref{eqn:gen} converges to 
$$(p-1) \int_{\Gw} |F'(u)|^{p-2} F''(u) \braket{\nabla u,X} \dx.$$ 
Hence, we infer that \eqref{eq:distrib} holds in the distribution sense.

\begin{remark} \rm Actually, the above proof shows that it is enough to assume $X \in W^{1,q} (\Gw)$ with compact support and $q=p':=p/(p-1)$.
\end{remark}

\section{Auxiliary inequalities}\label{app2}
In this appendix we prove two elementary inequalities that play a crucial role in the proofs of Lemma~\ref{lem:sub_sup_per_Agmon2} and Lemma~\ref{lem:agmon_subsup_inf}.
\begin{proposition} \label{Prop:computation} The following holds.
	\begin{enumerate}[(i)]
		\item Let $\nu < \beta$, such that,  either $\nu \in [\frac{\al+p-1}{p},\frac{\al+p-1}{p-1}]$ if $\al+p>1$, or $\nu \in [\frac{\al+p-1}{p},0]$ if $\al+p<1$, respectively. Then 
			\begin{align} \label{to_prove1}
			(p-2)\frac{\la_{\nu} \beta}{\nu} + \la_{\beta} \frac{|\nu|^{p-2}}{|\beta|^{p-2}} <  \la_{\nu} (p-1) \,,
			\end{align}
			where $\la_{\nu}=|\nu|^{p-2}\nu[\al+(1-\nu)(p-1)].$
		\item Let $\beta<\nu $, such that either $\nu \in [\frac{\al+p-N}{p-1},\frac{\al+p-N}{p}]$  if $\al+p<N$, or $\nu \in [0,\frac{\al+p-N}{p}]$ if $\al+p>N$  respectively. Then 
			\begin{align} \label{to_prove2}
			(p-2)\frac{\hat{\la}_{\nu} \beta}{\nu} + \hat{\la}_{\beta} \frac{|\nu|^{p-2}}{|\beta|^{p-2}} <  \hat{\la}_{\nu} (p-1) \,,
			\end{align}
			where $\hat{\la}_{\nu}=|\nu|^{p-2}\nu[(\al-N+1)+(1-\nu)(p-1)].$
	\end{enumerate} 
\end{proposition}
\begin{proof}
	$(i)$ Notice that \eqref{to_prove1} holds if 
		\begin{align} \label{s1}
		(p-2)\frac{\beta}{\nu} + \frac{\la_{\beta}}{\la_{\nu}} \frac{|\nu|^{p-2}}{|\beta|^{p-2}} <   (p-1) \,.
		\end{align}
		An elementary computation shows that \eqref{s1} is true if
		$$\left(\frac{\beta}{\nu}\right) \frac{-1[\al+(1-\nu)(p-1)] +[\al+(1-\gb)(p-1)]}{\al+(1-\nu)(p-1)} < \frac{(p-1)(\nu-\beta)}{\nu} \,, $$ 
		 which is equivalent to 
		\begin{align} \label{s2}
		\left(\frac{\beta}{\nu} \right) \frac{(\nu-\beta)}{\al+(1-\nu)(p-1)} < \frac{(\nu-\beta)}{\nu} \,.
		\end{align}
		Since $\nu<\beta$ and $\nu [\al+(1-\nu)(p-1)]$ is nonnegative, \eqref{s2} holds if
		\begin{align} \label{s3}
		\beta > \al+(1-\nu)(p-1) \,,
		\end{align}
		which is indeed true as  $\al+(1-\nu)(p-1)\leq \nu$.
	
$(ii)$ The proof of \eqref{to_prove2} follows from analogous computation as above.
\end{proof}

\section{Covering lemma}\label{appendix1}
\begin{lemma}
	Let $\Gw\subset \R^N$ be a $C^{0,1}$-domain with a compact boundary. Then the  minimal number of balls of radius $r$ to cover  $D_r=\{x\in\Omega \mid r/2 \leq \delta_\Gw(x) \leq r \}$ is $O(r^{1-N})$. 
\end{lemma}
\begin{proof}
By our assumption $\Gw$ is  a Lipschitz domain. Therefore, locally at the boundary, the domain is a subgraph of a Lipschitz continuous function $\varphi$. Thus, there is a small neighborhood $U\subset \Gw$ of the boundary (up to a translation and rotation) which is given by $U= \{x\in \Gw  \mid 0<x_N<\varphi (x_1, \ldots , x_{N-1})\}$.

Moreover, in the neighborhood  $U$, the Euclidean distance to the boundary $\gd_\Gw$ is equivalent to the graph distance $d_N(x):=|\varphi(x_1, \dots, x_{N-1})- x_N |$ (see \cite[Lemma~7.2]{Burenkov_Lamberti}).
Namely, there exists a constant $C>0$ such that
$$\gd_\Gw(x)\leq d_N(x)\leq C \gd_\Gw(x) \qquad \forall x\in U.$$
Consequently, it is enough to cover  the set
$D_{N,r}:=\{x\in U \mid \{r/2 \leq d_N(x)\leq r\}$ by $O(r^{1-N})$ balls of radius $r$. 

Recall that any compact set in $\R^{N-1}$ can be covered by $O(r^{1-N})$ balls of radius $r$.  So, we can 
 cover the $(N-1)$-dimensional base domain by $O(r^{1-N})$ balls of radius $r$. Using the Lipschitz continuity of $\varphi$,  there is a parallelpiped with sides of order $r$ which covers $D_{N,r}$ by $O(r^{1-N})$ balls of radius $r$. 
\end{proof}
\section{Subcriticality implies the existence of Green function}\label{AppendixD}

Let $\mathcal{Q}=\mathcal{Q}_{p,A,V}$ be a nonnegative functional on $W^{1,p}(\Om)\cap C_c(\Om)$ of the form  
		$$\mathcal{Q}(\varphi):=\displaystyle\int_{\Om} [|\nabla \varphi|_A^p + V|\varphi|^p] \dx \qquad \varphi \in  W^{1,p}(\Om)\cap C_c(\Om) \,,$$
{where $ A: =(a_{ij}) \in  L^{\infty}_{\mathrm{loc}}(\Omega;\mathbb{R}^{N\times N})$ be  a symmetric and locally uniformly
		positive definite matrix, and 
		$|\xi|_A^2:= \displaystyle \sum_{i,j=1}^N a_{ij}(x) \xi_i \xi_j$, $\xi \in \mathbb{R}^N$ and $V$ belongs to a local Morrey space as in \cite{Yehuda_Georgios}.
		 
		\begin{definition}[$\mathcal{Q}$-capacity \cite{Das_Pinchover1}] \label{Def:Cap} 
			{\em 
				Let ${u} \in W^{1,p}_{\loc}(\Om)\cap C(\Om)$ be a positive function.  
		For a compact set $F \Subset \Om$, the {\em $\mathcal{Q}$-capacity of $F$ with respect to $(\Om,u)$} is defined by
		$${\mathrm{Cap}}(F, \Omega,u):=\inf \{\mathcal{Q}(\varphi) \mid \varphi \in 
		 W^{1,p}(\Om)\cap C_c(\Om), \  \varphi \geq u \ \text{ on } \ F\}.$$
	}
\end{definition}
\begin{theorem}\label{thm_AppendixD}
Let $\mathcal{Q}=\mathcal{Q}_{p,A,V}\geq 0$ in $\Gw$,  $x_0\in \Gw$ and $u\in \mathcal{M}_{\Gw\setminus \{x_0\}}$. In addition to our regularity assumptions on $A$ and $V$, assume that for some $r_0>0$, $B:= B_{r_0}(x_0)\subset \Gw_1$ such that $\Gw\setminus B$ is connected, and $V\in L^\infty(U_{r_0})$, where $U_{r_0}\subset \Gw_1$ is a connected neighborhood of $\partial B$. If $u$ has a removable singularity at $x_0$, then $\mathcal{Q}$ is critical in $\Gw$.	 In other words, if $\mathcal{Q}$ is subcritical, then it admits a minimal positive Green function at $x_0$. 
\end{theorem}
\begin{proof}
Let $(\Gw_m)$ be a smooth compact exhaustion of $\Gw$ and denote $Q:=Q_{p,A,V}$-{the quasilinear operator associated to $\mathcal{Q}_{p,A,V}$}.  By our assumption,  $u$ is a continuous solution of $Q[u]=0$ in $\Om$, since $u$ has removable singularity at $x_0$.  Let {$u_n^{B}$} be the solution of the Dirichlet problem:
$$Q[w]=0 \mbox{ in } \Gw_n\setminus \overline{B}; \ \ w=u \mbox{ on }  \partial B; \ \ w=0 \mbox{ on } \partial \Gw_n.$$
Then $0<u_n^{B} \leq u$ and $u_n^{B}\nearrow u^{B}\leq u$  in   
$\Gw\setminus B$. On the other hand, by the definition of minimal growth, $u \leq u^{B}$ in   
$\Gw\setminus B$. Therefore,  $u^{B}=u$ in $\Gw\setminus B$.

Denote $u_n= u_n^{B}$.
Extend $u_n$ by $u$ and $0$ in $B$ and outside $\Omega_n$, respectively,  and denote this extension by $\bar u_n$. Then by the pasting lemma \cite[Lemma~3.1]{BMR}, we have $\bar u_n \in W_0^{1,p}(\Gw_n)$ and $Q(\bar u_n)\geq 0$ in $\Gw_{n}$. Following \cite[Proposition 4.2.]{BMR}, let $h\in \mathrm{Lip(\R)}$ be such that 
	$$h(t)=\begin{cases}
	0  &t < 0,\\
	t  & t\in[0,1],\\
	1 & t>1,
	\end{cases}$$ 
and let $\gr$ be the signed distance from $\partial B$, that is, $\gr(x)= \mathrm{dist}(x,\partial B)$ if $x\not \in B$, $\gr(x)= -\mathrm{dist}(x,\partial B)$ if $x\in B$. For small $\vge >0$ set $h_\vge(t)=h(t/\vge)$ in  
$\Gw_n \setminus \bar B$. Using the equation $Q(u_n)=0$ in $\Gw_n \setminus \bar B$ with the test function $h_\vge(\gr)u_n$ we have in $\Gw_n \setminus \bar B$
\begin{align*}
0=Q(u_n)[h_\vge(\gr)u_n]&= 
\int_{\Gw_n \setminus \bar B} h_\vge(\gr)(|\nabla u_n|_A^p+Vu_n^p)\dx\\
 +&
\vge^{-1}\int_{0<\gr<\vge}u_n|\nabla u_n|_A^{p-2}A\nabla u_n\cdot \frac{\nabla(\gr)}{|\nabla(\gr)|}|\nabla(\gr)|\dx.
 \end{align*}
 By the coarea formula for the second term of the above equation, we obtain
\begin{align*}
\vge^{-1}\!\int_{\{0<\gr<\vge\}\cap\{\Gw\setminus B\}}\!\!\!\!\!\!\!u_n|\nabla u_n|_A^{p-2}A\nabla u_n \!\cdot\! \frac{\nabla(\gr)}{|\nabla(\gr)|}|\nabla(\gr)|\dx
= \vge^{-1}\!\int_0^\vge\!\left\{\int_{\{\gr=t\}\cap\{\Gw\setminus B\}} \!\!\!\!\!u_n|\nabla u_n|_A^{p-2} \frac{\partial u_n}{\partial \gn_A}\,\mathrm{d}\gs\right\}\!\dt,
\end{align*}
where $\frac{\partial v}{\partial \gn_A}:=A\nabla v\cdot \gn$ is the conormal derivative of $v$ pointing outward of $B$.

Letting $\vge\to 0$, we obtain 
\begin{align*}
0= \int_{\Gw_n \setminus \bar B} (|\nabla u_n|_A^p+Vu_n^p)\dx
+\int_{\partial B} u|\nabla u_n|_A^{p-2} \frac{\partial u_n}{\partial \gn_A}\,\mathrm{d}\gs.
\end{align*}
Similarly, use the equation $Q(u)=0$ in $B$ with the test function $h_\vge(-\gr) u$, and apply the coarea formula and then let  $\vge\to 0$, we finally obtain 
\begin{align*}
0= \int_{B} (|\nabla u|_A^p+Vu^p)\dx
- \int_{\partial B} u|\nabla u|_A^{p-2} \frac{\partial u}{\partial \gn_A}\,\mathrm{d}\gs.
\end{align*}
Adding the two equations, it follows that in $\Gw$ we have 
\begin{align*}
    \mathcal{Q}(\bar u_n) =\int_{\partial B} u\left[   |\nabla \bar u_n|_A^{p-2}  \frac{\partial \bar u_n}{\partial \gn_A} -   |\nabla  u|_A^{p-2} \frac{\partial u}{\partial \gn_A} \right] \d \sigma \,.
\end{align*}
Now, due to Lieberman's up to the boundary regularity estimate \cite{Lieberman}, we can show using the Arzelà-Ascoli theorem that $|\nabla \bar u_n|_A^{p-2} A\nabla \bar u_n \cdot \nu  \rightarrow   |\nabla  u|_A^{p-2} A\nabla  u\cdot \nu$ uniformly on $\partial B$ as $n \rightarrow \infty$. Hence, $\mathcal{Q}(\bar u_n) \rightarrow 0$ as $n \rightarrow \infty$. In addition,   $\bar u_n\to u$ locally uniformly in $\Gw$. Therefore,    
$ \{\bar u_n\}_{n}$ is a null-sequence for $\mathcal{Q}$. Hence $\mathcal{Q}$ is critical,  and $u$ is the ground state.

{{\bf Alternative proof for $1<p\leq N$} (without the extra $V\in L^\infty(U_{r_0})$ assumption) }

As above, without loss of generality, we may assume $B_{1}(x_0)\subset \Gw_1$. Denote $u_n^m= u_n^{B_{1/m}(x_0)}$.
Extend $u_n^m$ by $u$ and $0$ in $B_{1/m}(x_0)$ and outside $\Omega_n$, respectively,  and denote this extension by $\bar u_n^m$. Then by the pasting lemma \cite[Lemma~3.1]{BMR} we have $\bar u_n^m \in W_0^{1,p}(\Gw_{n})$ and $Q_{p,A,V}(\bar u_n^m)\geq 0$ in $\Gw_{n}$.  

For a fixed $m\geq 1$ we have 
$$\lim_{n\to \infty} \bar u_n^{m}=u \quad \mbox{in } \Gw.$$
Using Cantor's diagonal argument, it follows that there is an increasing  subsequence $\{m(n)\}_{n=1}^\infty$ of integers such that  $\bar u_n^{m(n) }\to u$ in $\Gw$ as $n\to \infty$.

Moreover,  $\bar u_n^m$ is $\mathcal{Q}$-Capacitor of $(B_{1/m}, \Omega_n,u)$, see \cite[Proposition 4.1]{BMR}. That is,
$${\mathrm{Cap}}(B_{1/m}, \Omega_n,u)=\mathcal{Q}(\bar u_n^m) \,.$$

Since $$ {\mathrm{Cap}}(B_{1/m(n)}, \Omega_1,u)\geq {\mathrm{Cap}}(B_{1/m(n)}, \Omega_n,u),$$
and ${\mathrm{Cap}}(B_{1/m(n)}, \Omega_1,u) \rightarrow 0$ as $n\to\infty$, it follows that 
$\mathcal{Q}(\bar u_n^{m(n)}) \to 0$. In addition,   $\bar u_n^{m(n)}\to u$ in $\Gw$. Therefore,    
$ \big\{\bar u_n^{m(n)}\big\}_{n=1}^\infty$ is a null-sequence for $\mathcal{Q}$. Hence $\mathcal{Q}$ is critical,  and $u$ is the ground state. 
\end{proof}
\begin{remark} \rm
The extra regularity $V\in L^\infty(U_{r_0})$, where $U_{r_0}\subset \Gw_1$ is a connected neighborhood of $\partial B$ is needed to guarantee the up to the boundary $C^{1,\ga}$ regularity estimate. Note that {we} provide the alternative proof for $1<p \leq N$ without using the Lieberman's boundary regularity result. So, in this case, Theorem \ref{thm_AppendixD} is valid for more general potentials $V$ in the local Morrey space (see \cite[Section 2.1]{Yehuda_Georgios} for the definition), provided the pasting lemma \cite[Lemma~3.1]{BMR} holds for such potentials. 
\end{remark}

\begin{remark} \rm
If the strong comparison principle holds for $Q$ in any Lipschitz bounded subdomain of $\Gw$ (e.g. $p=2$), then Theorem~\ref{thm_AppendixD} follows.

Indeed, let $v$ be a continuous positive supersolution of the equation $Q(w)=0$ in $\Gw\setminus K$ such that $u\leq v$ on $\partial K$, where $K\Subset \Gw$ is a compact set such that   $\mathring{K}$ is a Lipschitz domain, and we may assume that there exists $x_1\in \partial K$ such that $u(x_1)=v(x_1)$.

If $x_0\in \mathring{K}$, then by definition,  $u\leq v $ in $\Gw\setminus K$. 

Suppose now that $x_0\not \in \mathring{K}$. {If $u(x_0) <  v(x_0)$}, then using the weak comparison principle in $\Gw_k \setminus (K\cup B_\gd(x_0))$ for $k\geq k_0$ and $\gd>0$ sufficiently small, and letting $k\to \infty$ and then $\gd\to 0$, we obtain $u\leq v$ in $\Gw\setminus K$. On the other hand, if $u(x_0)=cv(x_0)$ with {$c\geq 1$}, then by the above exhaustion argument we obtain $u\leq cv$ in $\Gw\setminus K$. The strong comparison principle implies now that $u=cv$ in $\Gw\setminus K$, which contradicts the assumption {$u(x_1)=v(x_1)$} unless $c=1$. 

Thus, we prove that $u \leq v$ in $\Omega \setminus K$. Hence, $u \in \mathcal{M}_{\Omega}$, which means $\mathcal{Q}$ is critical.           
\end{remark}
\begin{center}
	{\bf Acknowledgments}
\end{center} 
The authors wish to thank Y.~Hou and P.D.~Lamberti for a valuable discussions, and the anonymous referees for their careful reading and valuable comments.
 U.D. and Y.P. acknowledge the support of the Israel Science Foundation \!(grant $637/19$) founded by the
Israel Academy of Sciences and Humanities. U.D. is also supported in
part by a fellowship from the Lady Davis Foundation. B.D. acknowledges the support of the French ANR project RAGE (ANR- 18-CE40-0012). B.D. was also supported in the framework of the “Investissements d’avenir” program (ANR-15-IDEX-02) by the LabEx PERSYVAL (ANR-11-LABX-0025-01).


\begin{thebibliography}{10}
\bibitem{Agmon} S.~Agmon. 
On positivity and decay of solutions of second order elliptic equations on Riemannian manifolds. Methods of functional analysis and theory of elliptic equations (Naples, 1982), 19--52, Liguori, Naples, 1983.
\bibitem{AFP}
L. Ambrosio, N. Fusco and D. Pallara.
\newblock Functions of Bounded Variation and Free Discontinuity Problems.
	{\em Oxford Mathematical Monographs}, Oxford Science, Oxford, 2000.

	\bibitem{AM98}
	L.~Ambrosio and C.~Mantegazza.
	\newblock Curvature and distance function from a manifold.
	\newblock {\em J. Geom. Anal.}, 8(5):723--748, 1998.
		
	\bibitem{Anane}
	A.~Anane.
	\newblock Simplicit\'{e} et isolation de la premi\`ere valeur propre du
	{$p$}-laplacien avec poids.
	\newblock {\em C. R. Acad. Sci. Paris S\'{e}r. I Math.}, 305(16):725--728,
	1987.

\bibitem{Ando} H.~Ando, H.~Toshio. Noncritical weighted Hardy's inequalities with compact perturbations. {\em ArXiv}: 2008.05167, 2020.


	
	\bibitem{Avkhadiev}
	F.G. Avkhadiev.
	\newblock Hardy type inequalities in higher dimensions with explicit estimate
	of constants.
	\newblock {\em Lobachevskii J. Math.}, 21:3--31, 2006.
	
	\bibitem{Avkhadiev_sharp}
	F.G. Avkhadiev.   
\newblock Sharp constants in Hardy type inequalities.
	\newblock {\em Izv. Vyssh. Uchebn. Zaved.
Mat.}, 2:61–65, 2015.

\bibitem{Avkhadiev_selected} F.G. Avkhadiev. 
Selected results and open problems on Hardy-Rellich and Poincar\'{e}-Friedrichs inequalities. 
{\em Anal. Math. Phys.}, 11(3): paper no. 134, 20 pp., 2021. 


\bibitem{Avkhadiev_Makarov} F.G. Avkhadiev, R.V. Makarov. Hardy Type Inequalities on Domains with Convex Complement and Uncertainty Principle of Heisenberg. {\em Lobachevskii J. Math.}, 40:1250–1259, 2019.
	
	\bibitem{Balinsky} A.A.~Balinsky, W. D.~Evans and R. T.~Lewis. The Analysis and Geometry of Hardy's Inequality. Universitext. Springer, Cham, 2015.
	
	
\bibitem{Barbatis}	 G. Barbatis, A. Tertikas. On the Hardy constant of non-convex planar domains: The case of the
 quadrilateral. {\em J. Funct. Anal.}, 266:3701--3725, 2014.
	

\bibitem{BMR} B.~Bianchini, L.~Mari, and M.~Rigoli, Yamabe type equations with a sign-changing nonlinearity, and the prescribed curvature problem, {\em J. Differential Equations}, 260:7416--7497, 2016.
	
	\bibitem{Burenkov_Lamberti} V.~Burenkov and P.D.~Lamberti, 
	Spectral stability of higher order uniformly elliptic operators. in:  Sobolev Spaces in Mathematics. II,
	Int. Math. Ser. (N.Y.), 9:69--102, Springer, New York, 2009.
	
	\bibitem{Byeon}  J.~Byeon and S.~Jin. The Legendre-Hardy inequality on bounded domains. {\em Trans. Amer. Math. Soc. Ser. B}, 9:208--257, 2022.
	
	\bibitem{Chabrowski}
	J.~Chabrowski and M.~Willem.
	\newblock Hardy's inequality on exterior domains.
	\newblock {\em Proc. Amer. Math. Soc.},
	134(4):1019--1022, 2006.
	
\bibitem{Colin} F. Colin.	Hardy's inequality in unbounded domains. {\em Topol. Methods Nonlinear Anal.}, 17(2):277--284, 2001.
	
	\bibitem{Das_Pinchover1}
	U.~Das and Y.~Pinchover.
	\newblock The space of hardy-weights for quasilinear equations: Maz'ya-type
	characterization and sufficient conditions for existence of minimizers.  {\em J. Anal. Math.}, 153:331--366, 2024.

 \bibitem{Cone}
	U.~Das and Y.~Pinchover.
	\newblock A lower bound for the weighted-Hardy constant for domains satisfying a uniform exterior cone condition. J. Geom. Anal., 35, article number 132, 15pp., 2025.
	
	
\bibitem{Davies} E.B. Davies. The Hardy constant. {\em Quart. J. Math. Oxford}, 46:417--431, 1995.

	
\bibitem{Delfour} M.C. Delfour, and J.-P. Zol\'esio. Shapes and geometries: Metrics, Analysis, Differential Calculus, and
Optimization. 2nd edn. SIAM, Philadelphia, 2011. 	
	
	\bibitem{DP}
	B.~Devyver and Y.~Pinchover.
	\newblock Optimal {$L^p$} {H}ardy-type inequalities.
	\newblock {\em Ann. Inst. H. Poincar\'{e} Anal. Non Lin\'{e}aire},
	33(1):93--118, 2016.

	
	\bibitem{DELT09}
	J.~Dolbeault, M.J. Esteban, M.~Loss, and G.~Tarantello.
	\newblock On the symmetry of extremals for the {C}affarelli-{K}ohn-{N}irenberg
	inequalities.
	\newblock {\em Adv. Nonlinear Stud.}, 9(4):713--726, 2009.


\bibitem{Dosly} O.~Do\v{s}l\'{y}, and P.~\v{R}eh\'{a}k.
Half-Linear Differential Equations.
North-Holland Mathematics Studies, 202. Elsevier Science B.V., Amsterdam, 2005. 



	\bibitem{Farah}
	A.~Farah.
	\newblock Proving the regularity of the reduced boundary of perimeter minimizing sets with the De Giorgi lemma.
	\newblock{Honors Thesis}, The University of Texas at Austin, 2020. http://dx.doi.org/10.26153/tsw/8365

\bibitem{GT} D. Gilbarg and N. S. Trudinger. Elliptic Partial Differential Equations of Second Order.
 Reprint of the 1998 edition, Classics in Mathematics. Springer-Verlag, Berlin, 2001.
 
	\bibitem{GPP}
	D.~Goel, Y.~Pinchover, and G.~Psaradakis.
	\newblock On weighted ${L}^p$-{H}ardy inequality on domains in $\mathbb{R}^n$,
	 in: Special Issue on Analysis and PDE Dedicated to Professor Shmuel Agmon. {\em Pure Appl. Funct. Anal.,} 7(3):1025--1023, 2022.
	
	
	\bibitem{Hardy}
	G.H. Hardy, J.E. Littlewood, and G.~P\'{o}lya.
	\newblock Inequalities.
	\newblock Cambridge Mathematical Library. Cambridge University Press,
	Cambridge, 1988.
	\newblock Reprint of the 1952 edition.
	

\bibitem{Kesavan_PDE}
S. Kesavan. 
\newblock Topics in
Functional Analysis
and Applications.
\newblock Revised edition.
Wiley, 2015.

\bibitem{Leherback1}
P. Koskela and J. Lehrb\"ack.
\newblock Weighted pointwise Hardy inequalities.
\newblock {\em London Math. Soc.(2)} 79:757--779, 2009.
	
	\bibitem{Kraft}
	D.~Kraft.
	\newblock {Measure-theoretic properties of level sets of distance functions}.
	\newblock {\em J. Geom. Anal.}, 26(4):2777--2796, 2016.	
	
\bibitem{Kufner}
A. Kufner.
\newblock Weighted Sobolev Spaces. 
\newblock John Wiley \& Sons, 1985.
	
	\bibitem{Lamberti}
	P.~D. Lamberti and Y.~Pinchover.
	\newblock {$L^p$} {H}ardy inequality on {$C^{1,\gamma}$} domains.
	\newblock {\em Ann. Sc. Norm. Super. Pisa Cl. Sci. (5)}, 19(3):1135--1159,
	2019.
	
	
	\bibitem{Leherback2}
J. Lehrb\"ack.	
\newblock Weighted Hardy inequalities and the size
of the boundary.
\newblock {\em Manuscripta Math.} 127:249--273, 2008.	

\bibitem{Lewis}
J. L. Lewis.
\newblock Uniformly fat sets. \newblock {\em Trans. Amer. Math. Soc.,} 308:177--196, 1988.

	\bibitem{Lieberman} G. M. Lieberman.
Boundary regularity for solutions of degenerate elliptic equations.
\newblock {\em Nonlinear Anal.}, 12(11): 1203-121, 1988.
	
	\bibitem{MMP}
	M.~Marcus, V.~J. Mizel, and Y.~Pinchover.
	\newblock On the best constant for {H}ardy's inequality in {$\mathbb{R}^n$}.
	\newblock {\em Trans. Amer. Math. Soc.}, 350(8):3237--3255, 1998.
	
	\bibitem{Itai}
	M.~Marcus and I.~Shafrir.
	\newblock An eigenvalue problem related to Hardy's $L^p$ inequality.
	\newblock {\em Ann. Scuola Norm. Sup. Pisa Cl. Sci.}, 29(3):581--604, 2000.

\bibitem{Muckenhoupt} B.~Muckenhoupt. Benjamin Hardy's inequality with weights. {\em Studia Math.}, 44:31--38,  1972.

\bibitem{Necas}
J. Ne\v cas. 
\newblock Sur une m\'ethode pour r\'esoudre les \'equations aux d\'eriv\'ees partielles du type elliptique, voisinede la variationnelle.
\newblock {\em Ann. Scuola Norm. Sup. Pisa (3)}, 16:305--326, 1962.	
	
	
	\bibitem{Yehuda_Georgios}
	Y.~Pinchover and G.~Psaradakis.
	\newblock {On positive solutions of the $(p,A)$-Laplacian with potential in
		Morrey space}.
	\newblock {\em Analysis \& PDE}, 9(6):1317--1358, 2016.
	
	
\bibitem{Psaradakis} 
G.~Psaradakis.
 \newblock {$L^1$ Hardy inequalities with weights}.
 \newblock {\em J. Geom. Anal.}, 23(4):1703--1728, 2013.
	
	
\bibitem{Robinson}  D.W.~Robinson.  The weighted Hardy constant. {\em J. Funct. Anal.}, 281(8): paper no. 109143, 36 pp, 2021.

\bibitem{Tomaselli}  G.~Tomaselli. A class of inequalities. {\em Boll. Un. Mat. Ital.}, 2(4):622--631. 1969. 
\bibitem{Wannebo2}
A. Wannebo.
\newblock Hardy inequalities. \newblock {\em Proc. Amer. Math. Soc.}, 109:85--95, 1990. 

\bibitem{Wannebo}
A. Wannebo.
\newblock Hardy inequalities and imbedding in domains generalizing  $C^{0,\la}$ domains.
\newblock {\em Proc.
Amer. Math. Soc.}, 122:1181--1190,  1994.	
	
\end{thebibliography}

\end{document}